\documentclass{amsart}
\usepackage{MnSymbol}

\usepackage{hyperref}
\hypersetup{
    colorlinks,
    citecolor=black,
    filecolor=black,
    linkcolor=black,
    urlcolor=black
}

\usepackage{comment}

\DeclareFontFamily{OT1}{pzc}{}
\DeclareFontShape{OT1}{pzc}{m}{it}{<-> s * [1.100] pzcmi7t}{}
\DeclareMathAlphabet{\mathpzc}{OT1}{pzc}{m}{it}

    \usepackage{etoolbox}
    \patchcmd{\section}{\scshape}{\large\bfseries}{}{}
    \makeatletter
    \renewcommand{\@secnumfont}{\bfseries}
    \makeatother

\usepackage{tikz-cd}
\tikzcdset{arrow style=math font}
\tikzcdset{scale cd/.style={every label/.append style={scale=#1},
    cells={nodes={scale=#1}}}}

\definecolor{color1}{rgb}{0.9,0.9,0.9}
\definecolor{color2}{rgb}{0.5,0.5,0.5}

\usepackage[maxbibnames=99]{biblatex}
\usepackage{booktabs}
\bibliography{references}
\DeclareFieldFormat[article,periodical]{volume}{\mkbibbold{#1}} 

\numberwithin{equation}{section}
\newtheorem{theorem}{Theorem}[section]
\newtheorem{corollary}[theorem]{Corollary}
\newtheorem{lemma}[theorem]{Lemma}
\newtheorem{proposition}[theorem]{Proposition}
\theoremstyle{definition}

\newtheorem{definition}[theorem]{Definition}
\newtheorem{remark}[theorem]{Remark}
\newtheorem{example}[theorem]{Example}

\def\CC{\mathsf{C}}
\def\DD{\mathsf{D}}
\def\NN{\mathsf{N}}

\def\PPP{\PPP}

\def\Ker{\mathsf{Ker}}
\def\Im{\mathsf{Im}}
\def\epi{\twoheadrightarrow}
\def\MMM{\mathpzc{M}}
\def\mono{\rightarrowtail}
\def\PPi{\mathsf{P}\Pi}
\def\PS{\mathsf{PS}}
\def\Quiv{\mathsf{Quiv}}
\def\qq{\mathsf{q}}
\def\PPP{\mathpzc{P}}
\def\DDD{\mathpzc{D}}

\def\SSS{\mathpzc{S}}
\def\KK{\mathbb{K}}
\def\CCC{\mathpzc{C}}
\def\LLL{\mathpzc{L}}
\def\AAA{\mathpzc{A}}
\def\ZZZ{\mathpzc{Z}}
\def\FFF{\mathpzc{F}}
\def\GGG{\mathpzc{G}}
\def\JJJ{\mathpzc{J}}
\def\ZZ{\mathbb{Z}}
\def\Nrv{\mathsf{Nrv}}

\def\PH{\mathsf{PH}}

\title[Simplicial approach to path homology]{Simplicial approach to path homology of quivers, marked categories, groups and algebras}

\author{Sergei O. Ivanov} 
\email{ivanov.s.o.1986@gmail.com, ivanov.s.o.1986@bimsa.cn}

\author{Fedor Pavutnitskiy}
\email{fedor.pavutnitskiy@gmail.com, fyo@bimsa.cn}

\thanks{The work is supported by Yanqi Lake Beijing Institute of Mathematical Sciences and Applications (BIMSA)}

\begin{document}

\maketitle

\begin{abstract}
We develop a generalisation of the path homology theory introduced by Grigor'yan, Lin, Muranov and Yau (GLMY-theory) in a general simplicial setting. The new theory includes as particular cases the GLMY-theory for path complexes and new homology theories:  path homology of categories with a chosen set of morphisms (marked categories) groups with a chosen subset (marked groups) and path Hochschild homology of algebras with chosen vector subspaces (marked algebras). Using our general machinery, we also introduce a new homology theory for quivers that we call square-commutative homology of quivers and compare it with the theory developed by Grigor’yan, Muranov, Vershinin and Yau. 
\end{abstract}

\setcounter{tocdepth}{1}
\tableofcontents

\section{Introduction}

For the first time the notion of path homology was introduced by Grigoryan, Lin, Muranov, Yau in an unpublished preprint \cite{grigor2012homologies}. They developed a homology theory for directed graphs and for path complexes. We will call it GLMY-theory. Since then, several articles have been published on this topic  \cite{grigor2014homotopy, grigor2017homologies, grigor2020path, grigor2014graphs, grigor2018path, grigor2018path2, kempton2021homology}. 
In fact, the definition of cohomology of digraphs can be found in earlier works of Dimakis and M\"uller-Hoissen \cite{dimakis1994discrete, dimakis1994differential} but a theory was not developed in them. There is a relation of this theory with the theory of magnitude homology \cite{asao2023magnitude}. These ideas are also used in applied mathematics \cite{grbic2022aspects, chowdhury2018persistent, chowdhury2019path}.

The main aim of this paper is to develop a theory in a general simplicial setting for general structures that we call \emph{path pairs}, which includes as particular cases the original GLMY-theory for path complexes (regular version) and new interesting homology theories: path homology of marked categories and marked groups, and path Hochschild homology of marked algebras. We also introduce a new homology theory for quivers that we call square-commutative homology of quivers, which is the path homology of a marked category associated with a quiver. We compare it with the theory developed in \cite{grigor2018path}, which we call $k$-power homology theory. It seems, our general approach allows to develop many other similar theories theories for ``marked structures'', for example, path homology of Lie algebras with a chosen vector subspace. 

For each of the theories we are interested in two questions:  ``is it homotopy invariant in some sense?'' and ``is it compatible with some product?''. The  GLMY-theory of digraphs and path complexes answers affirmatively on these questions. It has two key theorems: the theorem about homotopy invariance of the homology; and an analogue of the Eilenberg–Zilber theorem together with the K\"unneth formula. 
Versions of both of these theorems were proved for path pairs and we deduce some versions of these theorems for square-commutative homology of quivers and for marked categories and groups. We also obtain a version of Eilenberg-Zilber theorem for path Hochschild homology of marked algebras. 

Recall that a simplicial set is a presheaf on the simplicial indexing category $\Delta.$ We consider its wide subcategory $\Pi\subseteq \Delta,$ whose morphisms are order preserving maps with ``connected'' image i.e. the image is of the form $\{k,k+1,\dots,l-1,l\}$ (Subsection \ref{subsection:path_indexing_cat}). Equivalently this subcategory can be defined as the least subcategory containing all codegeneracy maps $s^i:[n+1]\to [n]$ and all \emph{exterior} coface maps $d^0,d^n:[n-1]\to [n].$ This subcategory $\Pi$ is called \emph{path indexing category} and a \emph{path set} is defined as a presheaf on this category. We can also define path objects in any category as functors from $\Pi^{\sf op}.$ In particular, we will consider path modules.    

A \emph{path pair of modules} (over a commutative ring $\KK$) is a couple $\PPP=(A,B),$ where $A$ is a simplicial module and $B$ is its path submodule. In other words, $A$ is a simplicial module and $B$ is a sequence of submodules $B_n\subseteq A_n$ which are closed with respect to degeneracy maps and exterior face maps (but not necessarily with respect to all face maps). 
Generalising the definition given in \cite{grigor2012homologies} we define a chain complex $\Omega \PPP,$ whose homology are called homology of $\PPP$. We prove homotopy invariance for this definition:
we define a notion of homotopic morphisms of path pairs $f\sim g:\PPP\to \PPP'$ and prove that they induce chain homotopic 
morphisms of chain complexes $\Omega f\sim \Omega g: \Omega \PPP \to \Omega \PPP'.$ We define a box product of path pairs of modules $\PPP\square \PPP'$  which is in some sense generalises the box product of digraphs, and prove a version of the Eilenberg-Zilber theorem: if $\KK$ is a principal ideal domain, under some conditions on path pairs of modules $\PPP$ and $\PPP'$ we obtain an isomorphism chain complexes (Theorem \ref{th:EZ}):
\begin{equation}\label{eq:intro:EZ}
    \Omega \PPP \otimes \Omega \PPP' \cong \Omega(\PPP \square \PPP').
\end{equation}
Note that here we have not just a homotopy equivalence of chain complexes, as in the classical Eilenberg–Zilber theorem, but we have an isomorphism of chain complexes. 
So, this theorem can't be considered as a generalization of the classical Eilenberg–Zilber theorem. 
This is because the box-product of path pairs is not a generalization of the tensor product of simplicial modules. 

Similarly to the definition of a path pair of modules one can define a path pair of sets. Any path pair of sets defines a path pair of free modules. In Section \ref{sec:path_pairs_of_sets} we show that all theorems about path pairs of modules imply some versions of these theorems for path pairs of sets. We also show that path complexes defined by Grigor’yan, Lin, Muranov and Yau in \cite{grigor2012homologies} are particular cases of path pairs of sets, and homotopy invariance theorem \cite[Th. 3.3]{grigor2014homotopy} and the Eilenberg-Zilber theorem \cite[Th.7.6]{grigor2012homologies} follow from the corresponding theorems for path pairs. 

Mimicking the definition of marked simplicial sets \cite[\S 3.1]{lurie2009higher} we define marked categories as couples $\MMM = (\CCC, M),$ where $\CCC$ is a category and $M$ is a set of morphisms, called marked morphisms, containing all identity morphisms. The set of marked morphisms defines a path subset in the nerve of $\CCC,$ so each marked category defines a path pair of sets. In Section \ref{sec:embedded_quivers} we define the chain complex $\Omega\MMM$ as the complex corresponding to the path pair of sets and show that there are corresponding versions for homotopy invariance theorem, Eilenberg–Zilber theorem and the K\"unneth theorem in this setting. If $\KK$ is a field, we also prove an interesting inequality for dimensions of the components of this complex: for any marked category $\MMM$ and natural numbers $k,l$ we have 
\begin{equation}\label{eq:intro:inequality}
    {\sf dim}(\Omega_{k+l}\MMM) \leq {\sf dim}(\Omega_k \MMM)\cdot {\sf dim}(\Omega_l \MMM).
\end{equation}
We also define path cohomology of a marked category and show that this is a graded algebra with respect to the cup-product (Subsection \ref{subsec:cohomology_of_embedded_quivers}). The original path homology of digraphs is a particular case of this theory: a digraph $G=(V,E)$ defines a marked category $({\sf c}(V),E),$ where ${\sf c}(V)$ is the category whose objects are elements of $V,$ with one morphism $(u,v)$ from $u$ to $v$, and marked morphisms are edges of the digraph; the path homology of this marked category is the original path homology of the digraph.  

In Section \ref{sec:linearly_embedded_quiver}  we consider some slight generalisation of the notion of marked category, marked linear category, and generalise some statements to this case. Further, in Section \ref{sec:k-power} we use the machinery of marked linear categories to introduce another approach to $k$-power homology theory of quivers developed by Grigor’yan, Muranov, Vershinin and Yau in  \cite{grigor2018path}.

In Section \ref{sec:square-commutative} we define a new version of homology of quivers that we call square-commutative homology of quivers $H^{\sf sc}_*(Q)$ as the path homology of a marked category $(\ZZZ(Q),Q_1)$ associated with $Q.$ We prove some versions of homotopy invariance theorem for square-commutative homology. A variant for the Eilenberg-Zilber theorem here was proved only for the case of digraphs (treated as quivers). We also compare this theory with the path homology of digraphs. We prove that if a digraph $G$ has no non-degenerated directed triangles and double edges, the square-commutative homology coincides with the path homology $H_*^{\sf sc}(G)\cong \PH_*(G).$ We show that for any simplicial complex $S,$ if we denote by $G(S)$ the associated graph considered in \cite{grigor2014graphs}, then
\begin{equation}
H_*(S)=H^{\sf sc}_*(G(S)).    
\end{equation}
So, the square-commutative homology can be as complicated as the homology of simplicial complexes. We also compare the square commutative homology $H_*^{\sf sc}(Q)$ with $k$-power 
homology of $H_*^{(k)}(Q)$ defined and studied in \cite{grigor2018path}. The power of a quiver $Q$ is the maximal number of arrows with equal head and tail. We show that if the power of $Q$ is strictly less than $k$ and $k\cdot 1_\KK$ is invertible in $\KK,$ then 
\begin{equation}
H^{\sf sc}_*(Q)\cong H^{(k)}_*(Q).  
\end{equation}
In particular, we obtain that, if $k\cdot 1_\KK$ and $l\cdot 1_\KK$ are invertible in $\KK,$ and the power of $Q$ is strictly less then both $k$ and $l,$ then $H_*^{(k)}(Q)\cong H_*^{(l)}(Q).$ 

Section \ref{sec:homology_of_subsets} is devoted to path homology of marked groups. We say that a subset of a group is pointed, if it contains the unit. A marked group is a group $G$ with a chosen pointed subset $M\subseteq G$. Since a group can be treated as a category with one object, a marked group is a particular case of a marked category. As a corollary of our general theorem we obtain a version of Eilenberg–Zilber theorem for marked groups:
\begin{equation}
    \Omega(G\times G',M \vee M') \cong \Omega(G,M) \otimes \Omega(G',M'),
\end{equation}
where $M\vee M'=(M\times 1)\cup (1\times M')$ and prove some other properties for this theory. In Section \ref{sec:Hochschild} we develop a similar theory for Hochschild homology of marked algebras (algebras with a chosen vector subspace).  

In Section \ref{sec:homology_of_subsets} we  study coacyclic subsets of groups. A pointed subset $M$ of $G$ is called coacyclic if the map $\PH_*(G,M)\to H_*(G)$ is an isomorphism for $\KK=\ZZ$. We prove some properties of coacyclic subsets and show several examples of them. 
In all these examples of coacyclic subsets the complexes $\Omega(G,M)$  are much smaller then the standard complex for the group. It would be interesting to develop this theory further and to understand how convenient it is to control the homology of a group $G$ using subsets of $G$. In particular,  it would be  interesting to find a connection between (co)homological dimension of a group and ``nice'' subsets of this group. These ``nice'' subsets should be not just the coacyclic subsets, of course, the definition should take account of (co)homology with coefficients in some way.

In the end of the paper we have an appendix of a more categorical flavor. We show that the box product of path pairs can be defined using Day convolution with respect to a promonoidal structure on the category $\Pi.$

\subsection*{Acknowledgements} We are grateful to Jie Wu for his useful remarks.

\section{Weak cylinder functors}

We will need to define homotopic morphisms in several different categories and prove homotopy invariance of different types of  path homology. A uniform approach to the definitions and proofs is via weak cylinder functors that we define in this section. 

A weak cylinder functor on a category $\CCC$ is a functor ${\sf cyl}:\CCC\to \CCC$ equipped with two natural transformations $i^0,i^1:{\sf Id}\to {\sf cyl}.$  We say that two morphisms $f,g:c\to c'$ of the category $\CCC$ are one-step homotopic (with respect to the cylinder functor), if there is a morphism $H:{\sf cyl}(c)\to c'$ such that $hi^0_c = f$ and $hi^1_c=g.$ We consider the the minimal equivalence relation on the hom-set $\CCC(c,c')$ that contains the relation of being one-step homotopic. Two morphisms are homotopic if they are equivalent with respect to this equivalence relation. Note that if we have two homotopic morphisms $f\sim g:c\to c'$ and a morphism $f':c'\to c'',$ then $f'f\sim f'g.$

\begin{proposition} \label{prop:cyl}
Let $\CCC$ and $\widetilde{\CCC}$ be two categories with weak cylinder functors $({\sf cyl},i^0,i^1)$ and $(\widetilde{\sf cyl}, \tilde{i}^0,\tilde{i}^1).$ Assume that $F:\CCC \to \widetilde{\CCC}$ is a functor and there is a natural transformation
$
\varphi : \widetilde{\sf cyl} \; F \to F \; {\sf cyl} 
$ such that $\varphi \circ (\tilde i^n F) = F i^n$ for any $n=0,1.$
\begin{equation}
\begin{tikzcd}
& F \ar[dl,"\tilde{i}^n F"'] \ar[dr,"F i^n"] & \\
\widetilde{\sf cyl}\: F \ar[rr,"\varphi"]  & & F \:  {\sf cyl}.
\end{tikzcd}
\end{equation}
Then $F$ takes homotopic morphisms to homotopic morphisms.
\end{proposition}
\begin{proof}
If $H:{\sf cyl}(c) \to c'$ is a homotopy between one-step homotopic morphisms $f$ and $g,$ then $ F(H) \circ \varphi_c: \widetilde{\sf cyl} (F(c)) \to F(c') $ is a homotopy between $F(f)$ and $F(g).$ 
\end{proof}

\section{Graded submodules of chain complexes} \label{section:omega}

\subsection{Complexes \texorpdfstring{$\omega$}{ω} and \texorpdfstring{$\psi$}{ψ}}

In this section we denote by $\KK$ a commutative ring and assume that all modules, chain complexes and tensor products are over $\KK.$

Let $C$ be a non-negatively graded chain complex over a commutative ring $\KK$ and $D$ be its graded submodule $D_n\subseteq C_n,$ which is not necessarily a subcomplex. Then we denote by $\omega(C,D)$ the maximal subcomplex of $C$ whose homogeneous components are submodules of $D.$ In other words $\omega(C,D)$ is a subcomplex of $C,$ whose homogeneous components are given by the formula
\begin{equation}
    \omega(C,D)_n=D_n\cap \partial^{-1}(D_{n-1}).
\end{equation}
We will also consider the minimal subcomplex, whose components contain $D,$ and denote it by $\omega'(C,D):$
\begin{equation}
    \omega'(C,D)_n=D_n+\partial(D_{n+1}).
\end{equation}

\begin{remark}\label{rem:embedding}
Note that if $C$ is a chain subcomplex of $C'$ then 
$\omega(C,D)=\omega(C',D)$ and $\omega'(C,D)=\omega'(C',D).$ Slightly more generally we can say that, if $f:C \to C'$ is a monomorphism of chain complexes, then $f$ induces an isomorphism $\omega(C,D)\cong \omega(C',f(D)).$
\end{remark}

The following proposition  follows from \cite[Prop.2.3]{grbic2022aspects} but we add it here with a proof for convenience. 

\begin{proposition}[{cf. \cite[Prop.2.3]{grbic2022aspects}}]
The inclusion $\omega(C,D) \hookrightarrow \omega'(C,D)$ is a quasi-isomorhism
\begin{equation}
    H_*(\omega(C,D))\cong H_*(\omega'(C,D)).
\end{equation}
\end{proposition}
\begin{proof}
Set $K_n=\Ker(\partial_n:C_n\to C_{n+1}).$ Using that $K_n\subseteq \partial^{-1}(D_{n-1})$ and $\partial(D_{n+1} \cap \partial^{-1}(D_{n}))=\partial(D_{n+1}) \cap D_{n},$ we obtain \begin{equation}
H_n( \omega(C,D) ) = \frac{D_n \cap K_n}{\partial(D_{n+1}) \cap D_{n}}.
\end{equation}
Using the modular law, we obtain
$(D_n+\partial(D_{n+1})) \cap K_n = (D_n\cap K_n)+\partial(D_{n+1}),$ and hence 
\begin{equation}
H_n(\omega'(C,D)) = \frac{(D_n\cap K_n) + \partial(D_{n+1})}{\partial(D_{n+1})}
\end{equation}
Then the second isomorphism theorem ($(X+Y)/Y\cong X/(Y\cap X)$) and the inclusion  $\partial(D_{n+1})\subseteq K_n$ imply the assertion.  
\end{proof}

Further, we denote by $\psi(C,D)$ the maximal quotient complex of $C$ such that the map $D\hookrightarrow C\epi \psi(C,D)$ is trivial. It is easy to see that 
\begin{equation}
\psi(C,D)=C/\omega'(C,D).
\end{equation}
Further we can define path homology and copath homology of the couple $(C,D)$ as follows 
\begin{equation}
H_*(C,D):=  H_*(\omega(C,D)), \hspace{1cm}   H_*^{\sf c}(C,D):=H_*( \psi(C,D) ). 
\end{equation}

\begin{corollary}\label{cor:long} For any graded submodule $D$ of a chain complex $C,$ there is a long exact sequence 
\begin{equation}
  \dots  \to \PH_n(C,D)\to H_n(C) \to \PH^{\sf c}_n(C,D) \to \PH_{n-1}(C,D) \to \dots.
\end{equation}
\end{corollary}

\begin{lemma}
For any chain complex with graded submodule $(C,D)$ and any $n$ there is an exact sequence 
\begin{equation}\label{eq:exact_sequence_omega}
0 \to    \omega_n(C,D) \to C_n \overset{f}\to C_n/D_n \oplus C_{n-1}/D_{n-1},
\end{equation}
where $f(c)=(c+D_n,\partial(c)+D_{n-1}).$
\end{lemma}
\begin{proof}
Obvious.
\end{proof}

\subsection{Functorial properties of \texorpdfstring{$\omega$}{ω}}

A morphism of chain complexes with graded submodules $f:(C,D)\to (C',D')$ is a morphism of chain complexes $f:C\to C'$ such that $f(D)\subseteq D'.$ It is easy to see that $\omega$ and $\psi$ define functors
\begin{equation}
\omega, \psi : \{\text{complexes with graded submodules} \} \longrightarrow \{ \text{complexes}\}.    
\end{equation}

\begin{proposition}\label{prop:restrictions}
Let $f,g:(C,D)\to (C',D')$ be two morphisms of chain complexes with graded submodules such that the restrictions coincide $f\mid_{D} = g\mid_{D}.$ Then 
\begin{equation}
\omega(f)=\omega(g): \omega(C,D) \longrightarrow \omega(C',D').    
\end{equation}
In particular, if $f:(C,D) \to (C,D)$ is an endomorphism such that $f$ is identical on $D,$ then $\omega(f)={\sf id}_{\omega(C,D)}.$
\end{proposition}
\begin{proof}
The proof is obvious.
\end{proof}

\begin{proposition}[Isomorphism-lemma for chain complexes]
\label{prop:E-iso}
Let $f:(C,D)\to (C',D')$ be a morphism of chain complexes with graded submodules and let  $E\subseteq C$ and $E'\subseteq C'$ be graded submodules such that  $\partial(D)\subseteq E$ and $\partial(D')\subseteq E'.$ Assume that $f$ induces isomorphisms $D\cong D'$ and $E\cong E'.$ Then $f$ induces  isomorphisms 
\begin{equation}
\omega(C,D)\cong \omega(C',D'), \hspace{1cm} \omega'(C,D)\cong \omega'(C',D').    
\end{equation}
\end{proposition}
\begin{proof}
The commutative square with vertical isomorphisms 
\begin{equation}
\begin{tikzcd}
 D_n \ar[r,"\partial"] \ar[d,"f","\cong"'] & E_{n-1} \ar[d,"f","\cong"'] \\
 D'_n \ar[r,"\partial"] & E'_{n-1}
\end{tikzcd}
\end{equation}
proofs that $f$ induces an isomorphism between $\partial^{-1}(D)\cong \partial^{-1}(D')$ and $\partial(D)\cong \partial(D').$ The assertion follows.
\end{proof}

\subsection{Homotopy invariance of \texorpdfstring{$\omega$}{ω} and \texorpdfstring{$\psi$}{ψ}}

For a chain complex $C$ we denote by ${\sf Cyl}(C)$ the chain complex, whose $n$-th homogeneous component is $C_n\oplus C_{n-1} \oplus C_n$ and the differential is given by the matrix
\begin{equation}
{\sf cyl}(C)_n=C_n\oplus C_{n-1} \oplus C_n, \hspace{1cm} d^{{\sf cyl}(C)}=\left(\begin{smallmatrix} 
d & 1 & 0 \\ 
0 & -d & 0 \\ 
0 & -1 & d
\end{smallmatrix}\right).
\end{equation}
Denote by $I^{\sf c}$ the cylinder of the chain complex $\KK[0]$ concentrated in zero degree
$
I^{\sf c} = {\sf cyl}(\KK[0]).
$
This complex concentrated in  degrees $0,1$  
\begin{equation}
I^{\sf c}: \hspace{1cm} {\dots} \to 0\to \KK \overset{ \binom{1}{-1} }\to \KK^2 \to 0 \to \dots 
\end{equation}
It is easy to check that there is an isomorphism
\begin{equation}
    {\sf cyl}(C) \cong C\otimes I^{\sf c}.
\end{equation}
There are two natural transformations
$i^0,i^1: {\sf Id} \to {\sf cyl}$
defined by 
$i^0=\left(\begin{smallmatrix} 1 \\ 0 \\ 0 \end{smallmatrix}\right)$ and $i^1=\left(\begin{smallmatrix} 0 \\ 0 \\ 1 \end{smallmatrix}\right),$ that 
make ${\sf cyl}$ a weak cylinder functor. It is well known that $f$ and 
$g$ are homotopic with respect to this weak cylinder functor if and only if there exists a morphism of degree -1 of underlying graded modules $h:C\to C'$ 
such that $f-g= dh + hd.$

For the category of chain complexes with graded submodules we define a weak cylinder functor ${\sf cyl}$ by the formulas
\begin{equation}
    {\sf cyl}(C,D) = ( {\sf cyl}(C),{\sf cyl}(D)), \hspace{1cm} {\sf cyl}(D)_n= D_n \oplus D_{n-1} \oplus D_n,
\end{equation}
$i^0=\left(\begin{smallmatrix} 1 \\ 0 \\ 0 \end{smallmatrix}\right)$ and $i^1=\left(\begin{smallmatrix} 0 \\ 0 \\ 1 \end{smallmatrix}\right).$
Homotopic morphisms of chain complexes with graded submodules are defined via this cylinder functor. Note that 
\begin{equation}\label{eq:cyl_I}
{\sf cyl}(C,D) = (C\otimes I^{\sf c},D\otimes I^{\sf g}),  
\end{equation}
where $I^{\sf g}$ is the underlined graded vector space of $I^{\sf c}.$

\begin{proposition}
Two morphisms of chain complexes with graded submodules $f,g: (C,D)\to (C',D')$ are homotopic if and only if there exist a chain homotopy $h_n:C_n\to C'_{n+1}$  such that 
$f-g=hd+dh$ and  $h(D)\subseteq D'$.
\end{proposition}
\begin{proof}
Let $H=(f,h,g):{\sf cyl}(C,D) \to (C',D')$ be a morphism.  The equation $Hd=dH$ is equivalent to $f-g=hd+dh$ and the inclusion $H( {\sf cyl}(D) )\subseteq D'$ is equivalent to $h(D)\subseteq D'.$ 
\end{proof}

\begin{proposition}\label{prop:homotopy_grad}
Let $f\sim g:(C,D)\to (C',D')$ be homotopic morphisms of chain complexes with graded submodules. Then the induced morphisms on $\omega,$ $\omega'$ and $\psi$ are homotopic 
\begin{equation}
\begin{split}
\omega(f) \sim \omega(g) : \omega(C,D) &\longrightarrow \omega(C',D'), \\
\omega'(f) \sim \omega'(g) : \omega'(C,D) &\longrightarrow \omega'(C',D'), \\
\psi(f) \sim \psi(g) : \psi(C,D) &\longrightarrow \psi(C',D'). 
\end{split}
\end{equation}
\end{proposition}
\begin{proof}
Let $h:C\to C'$ be the homotopy  such that $f-g=hd+dh$ and $h(D)\subseteq D'.$  We claim that $h(d(D))\subseteq D'+d(D').$ Indeed, for $x\in D$ we have $hd(x)=f(x)-g(x)-dh(x)$ and $f(x),g(x)\in D'$ and $dh(x)\in d(D').$  Therefore, $h(\omega'(C,D))\subseteq \omega'(C',D').$ It follows that $h$ induces a chain homotopy  $\omega'(f)\sim \omega'(g)$ and $\psi(f)\sim \psi(g).$ We also claim that $h(D \cap d^{-1}(D) )\subseteq D'\cap d^{-1}(D').$ Indeed, if $x\in D \cap d^{-1}(D),$ then $h(x)\in D'$ and $d(h(x))=f(x)-g(x)-hd(x) \in D'.$ Thus $h(\omega(C,D))\subseteq \omega(C',D')$ and $h$ induces a chain homotopy $\omega(f)\sim \omega(g).$
\end{proof}

\subsection{Relation of \texorpdfstring{$\omega$}{ω} and \texorpdfstring{$\psi$}{ψ} to the tensor product}

For any modules $M,M'$ and their submodules $N\subseteq M$ and $N'\subseteq M'$ we set 
\begin{equation}
N\bar \otimes N' = \Im(N\otimes N' \to M\otimes M').
\end{equation} Note that there is an isomorphism
\begin{equation}\label{eq:tensor_of_quo}
M/N \otimes M'/N' \cong (M \otimes M')/ (M\bar \otimes N' + N\bar \otimes M').
\end{equation}

\begin{proposition}
Let $\KK$ be a commutative ring, $C,C'$ be chain complexes over $\KK$ and $D,D'$ be their graded submodules. Then there is an isomorphism 
\begin{equation}
\psi(C,D) \otimes \psi(C',D') \cong \psi(C\otimes C', C\bar \otimes D' + D \bar \otimes C').    
\end{equation}
\end{proposition}
\begin{proof}
By \eqref{eq:tensor_of_quo} the $n$-th component of $\psi(C,D) \otimes \psi(C',D')$ is isomorphic to the quotient of  $\bigoplus_{i+j=n} C_i\otimes C'_j$ by 
\begin{equation}\label{eq_psi_prod1}
\bigoplus_{i+j=n} \Big( C_i \bar \otimes (D'_j+\partial(D'_{j+1})) + (D_i+\partial(D_{i+1})) \bar \otimes C_j' \Big).
\end{equation} 
By the definition of $\psi$ we obtain that the $n$-th component of $\psi(C\otimes C', C\bar \otimes D' + D \bar \otimes C')$ is equal to the quotinet of  $(C \otimes C')_n$ by
\begin{equation}
(C\bar \otimes D' + D \bar \otimes C')_n + \partial^{C\otimes C'}( (C\bar  \otimes D' + D\bar \otimes C')_{n+1}).
\end{equation}
It is easy to see that 
\begin{equation}\label{eq_psi_prod2}
(C \bar \otimes D' + D \bar \otimes C')_{n}  = \bigoplus_{i+j=n} \Big( C_i \bar \otimes D'_j + D_i \bar \otimes C'_j  \Big)
\end{equation}
and 
\begin{equation}\label{eq_psi_prod3}
\begin{split}
&\partial^{C\otimes C'}( (C \bar \otimes D' + D \bar \otimes C')_{n+1}) = \\ 
&= \bigoplus_{i+j=n} \Big( \partial(C_{i-1}) \bar \otimes D'_j + C_{i}  \bar\otimes \partial(D'_j)+\partial (D_{i+1}) \bar \otimes C'_j + D_i \bar \otimes \partial(C'_{j+1})  \Big)
\end{split}
\end{equation}
Using that $\partial(C_{i-1}) \bar \otimes D_j' \subseteq C_i \bar \otimes D_j'$ and $D_i \bar \otimes \partial(C'_{j+1}) \subseteq D_i \bar \otimes C_j',$ we obtain that the sum of \eqref{eq_psi_prod3} and \eqref{eq_psi_prod2} equals to \eqref{eq_psi_prod1}. The assertion follows. 
\end{proof}

\begin{lemma}\label{lemma_direct_summand}
Let $\mathbb{K}$ be a principal ideal domain and $D$ be a graded submodule of a chain complex $C$ over $\KK.$ Assume that $C_n$ is free and $D_n$ is a direct summand of $C_n$ for any $n.$ Then $\omega_n(C,D)$ is a direct summand of $D_n$ for any $n.$ 
\end{lemma}
\begin{proof}
A submodule of a free module over a principal ideal domain is free. Hence for any homomorphism $f:M\to M',$ if $M'$ is free, then $\Ker(f)$ is a direct summand of $M$ (because $\Im(f)$ is free and the short exact sequence $\Ker(f)\mono M \epi \Im(f)$ splits). Note also that the module $C_n/D_n$ is free for any $n$ because $D_n$ is a direct summand of $C_n.$ Then the equation 
\begin{equation}
\label{eq:omega}
\omega(C,D)_n=\Ker\left(f:D_n \to C_{n-1}/D_{n-1} \right),
\end{equation}
where $f(x)=\partial x +D_{n-1},$
implies that $\omega(C,D)_n$ is a direct summand of $D_n.$
\end{proof}

\begin{proposition}\label{prop_omega_direct}
Let $\mathbb{K}$ be an principal ideal domain and let $C,C'$ be chain complexes over $\KK$ and $D,D'$ be their graded submodules. Assume that $C_n,C_n'$ are free modules and $D_n,D'_n$ are their direct summands respectively for any $n$. Then the tensor product of the canonical embeddings $\iota : \omega(C,D)\to C$ and $\iota': \omega(C',D')\to C'$ induces an isomorphism
\begin{equation}
    \omega(C,D) \otimes \omega(C',D') \cong \omega(C\otimes C', D \otimes D').
\end{equation}
\end{proposition}
\begin{proof}
Since $\mathbb{K}$ is a principal integral domain and $C_n,C'_n$ are free, we obtain that the modules $D_n,D_n',\omega_n(C,D),\omega_n(C',D')$ are also free and the map $\omega(C,D)\otimes \omega(C',D')\to C\otimes C'$ is injective. We will identify $\omega(C,D)\otimes \omega(C',D')$ with its image in $C\otimes C'.$ So we need to prove that $\omega(C,D)\otimes \omega(C',D')=\omega(C\otimes C', D'\otimes D').$ Since $\omega(C,D)\otimes \omega(C',D')$ is a subcomplex of $C\otimes C,$ whose components lie in $D\otimes D',$ we have $\omega(C,D)\otimes \omega(C',D')\subseteq \omega(C\otimes C',D\otimes D').$ So, it is enough to prove the opposite inclusion.

Take $x\in \omega_n(C\otimes C',D\otimes D')$ and prove that $x\in (\omega(C,D)\otimes \omega(C',D'))_n.$ Decompose $x$ as $x=\sum_{k+l=n} x_{k,l},$ where $x_{k,l}\in D_k\otimes D'_l.$
Take a basis $(b_{i,k})_{i\in I_k}$ of $D_k.$ Then $x_{k,l}=\sum  b_{i,k} \otimes y_{i,k,l}$ for some $y_{i,k,l}\in D'_l.$ The component of $\partial x$ in the summand $C_{k}\otimes C_{l-1}$ is
\begin{equation}\label{eq_sum_diff}
\begin{split}
(\partial\otimes 1)(x_{k+1,l-1}) + 
(1\otimes \partial)(x_{k,l})
=\\
=\sum_{i\in I_{k+1}} \partial b_{i,k+1} \otimes y_{i,k+1,l-1}  +(-1)^k\sum_{i\in I_k} b_{i,k}\otimes  \partial y_{i,k,l}  \end{split}    
\end{equation}
Since $\partial x \in D\otimes D'$ its image under the map $1\otimes {\sf pr}:C\otimes C' \epi C\otimes C'/D'$ is  trivial, where ${\sf pr}:C'\epi C'/D'$ is the canonical projection. The image of the left hand summand of 
\eqref{eq_sum_diff} is also trivial in $C\otimes (C'/D')$ because $y_{i,k+1,l-1}\in D'_{l-1}.$ Then $\sum_{i\in I_{k}} b_{i,k}\otimes {\sf pr}(\partial y_{i,k,l})=0.$ Since $D'_{l-1}$ is a direct summand of $C'_{l-1}$ and $\mathbb{K}$ is a principal ideal domain, we obtain that $C'_{l-1}/D'_{l-1}$ is also a free module. Therefore the equation $\sum_{i\in I_{k}} b_{i,k}\otimes {\sf pr}(\partial y_{i,k,l})=0$ implies 
${\sf pr}(\partial y_{i,k,l})=0$ for any $i\in I_k.$ Then $\partial y_{i,k,l}\in D'_{l-1}$ and $y_{i,k,l}\in \omega_l(C',D').$ Thus $x_{k,l}\in D_k\otimes \omega_l(C',D').$ Similarly we prove that $x_{k,l}\in \omega_k(C,D)\otimes D'_l.$ By Lemma \ref{lemma_direct_summand} the modules $\omega_k(C,D)$ and $\omega_l(C',D')$ are direct summands of $D_k$ and $D_l'$ respectively. Then $(\omega_k(C,D)\otimes D'_l)\cap (D_k\otimes \omega_l(C',D'))=\omega_k(C,D)\otimes \omega_l(C'D').$ Therefore $x_{k,l}\in \omega_k(C,D)\otimes \omega_l(C'D'),$ and hence, $x\in \omega(C,D)\otimes \omega(C',D').$
\end{proof}

\begin{corollary}\label{cor:omega_direct}
Under the assumption of Proposition  \ref{prop_omega_direct} the map 
\begin{equation}
 \iota\otimes \iota ' : \omega(C,D) \otimes \omega(C',D) \longrightarrow C\otimes C' 
\end{equation}
is injective. 
\end{corollary}

\begin{remark}
The assumption that $\KK$ is a principal ideal domain in Proposition \ref{prop_omega_direct} and Corollary \ref{cor:omega_direct} is essential. For a example take $\KK=\ZZ/4$ and the chain complex of length one
\begin{equation}
 C: \hspace{1cm} 0 \to \ZZ/4 \overset{\cdot 2} \to \ZZ/4 \to 0     
\end{equation}
concentrated in degrees $0$ and $1.$ Consider the graded submodule $D\subseteq C$ defined by the equations $D_1=C_1=\ZZ/4$ and $D_0=0.$ Then $\omega_1:=\omega_1(C,D)= 2\ZZ/4\ZZ.$ The embedding $\iota_1 : \omega_1 \to C_1$ is isomorphic to the embedding $\cdot 2:\ZZ/2 \mono \ZZ/4.$ Since $2\cdot 2=4,$ it follows that $\iota_1 \otimes \iota_1 : \omega_1 \otimes \omega_1 \to C_1 \otimes C_1$ is isomorphic to the zero map $0:\ZZ/2 \to \ZZ/4,$ and hence, it is equal to zero:
\begin{equation}
    \ZZ/2\cong \omega_1 \otimes \omega_1  \overset{0}\longrightarrow C_1 \otimes C_1\cong \ZZ/4.
\end{equation}
\end{remark}

\subsection{DG-(co)algebras} 

\begin{proposition}\label{prop:psi-dg}
Let $C$ be a dg-algebra and $D$ be a graded ideal of the underlying graded algebra of $C.$ Then $\omega'(C,D)$ is a dg-ideal of $C$ and $\psi(C,D)$ inherits a structure of dg-algebra.  
\end{proposition}
\begin{proof}
By the definition $\omega'(C,D)$ is subcomplex. So we only need to prove that $\omega'(C,D)$ is an ideal. Indeed, for any $a\in D_n, b\in D_{n+1}$ and $x\in C_m$ we have $x(a+\partial(b))=xa+ x\partial(b) = xa \pm  \partial(x)b  \pm xb \in D_{n+m}$ and similarly $(a+\partial(b))x\in D_{n+m}.$
\end{proof}

For a coalgebra $C$ we say that $D$ is a split sub-coalgebra of $C,$ if $D$ is a submodule of $C,$ which is a direct summand, and such that $\nu(D)\subseteq D\otimes D,$ where $\nu:C\to C\otimes C$ is a comultiplication. Since, $D$ is a direct summand, we can identify $D^{\otimes n}$ with a submodule of $C^{\otimes n}.$  This defines a structure of coalgebra on $D.$ The same definition can be generalised to graded coalgebras and dg-coalgebras.

\begin{proposition}\label{prop:coalg_omega}
Let $\KK$ be a principal ideal domain, $C$ be a dg-coalgebra and $D$ be its graded split sub-coalgebra. Assume that $C_n$ is a free module for any $n$. Then $\omega(C,D)$ is a split sub-dg-coalgebra. 
\end{proposition}
\begin{proof} By Lemma \ref{lemma_direct_summand} the embedding $\omega(C,D)\to C$ splits. So 
we just need to prove that $\nu(\omega(C,D))\subseteq \omega(C,D) \otimes \omega(C,D),$ where $\nu:C\to C\otimes C$ is the comultiplication. Since the map $\nu :(C,D) \to (C\otimes C, D\otimes D)$ is a morphism of chain complexes with graded submodules, we obtain $\nu( \omega(C,D)) \subseteq \omega(C\otimes C,D\otimes D).$ Then the assertion follows from Proposition \ref{prop_omega_direct}
\end{proof}

\subsection{Duality over fields}

For any $\KK$-module $M$ we set 
\begin{equation}
    M^\vee = {\sf Hom}_\KK(M,\KK).
\end{equation}
If $C$ is a chain complex, then  $C^\vee$ is a cochain complex such that $(C^\vee)^n=C_{n}^{\vee}.$ 

\begin{proposition}\label{prop:duality}
Let $\KK$ be a field and $(C,D)$ be a chain complex with graded submodule over $\KK.$ Then 
\begin{equation}
\omega(C,D)^\vee \cong \psi(C^\vee,{\sf Ker}(C^\vee \to D^\vee)).    
\end{equation}
\end{proposition}
\begin{proof} Set $K^{n}={\sf Ker}(C_n^\vee \to D_n^\vee).$ Note that $K^n\cong (C_n/D_n)^\vee.$
Applying the duality to the exact sequence \eqref{eq:exact_sequence_omega} we obtain an exact sequence
\begin{equation}
K^n \oplus K^{n-1} \overset{f'}\to C_n^\vee \to \omega_n(C,D)^\vee \to 0.
\end{equation}
It is easy to see that the image of $f'$ equals to $\omega'_n(C^\vee,K).$ The assertion follows. 
\end{proof}

\section{Quivers}
\subsection{Definition of a quiver}
\label{subsection:definition_of_quiver} A quiver is usually defined as a couple of sets $Q_0,Q_1$ together with a couple of maps $h,t:Q_1\to Q_0,$ however, we prefer another definition, which allows to define morphisms in a more appropriate way for our reasons. For each vertex $v\in Q_0$ we can add a ``degenerated loop'' $s(v)$ to the set of edges and consider a new set $\tilde Q_1 = Q_1 \sqcup s(Q_0).$ Such degenerate loops are included in the set of edges in our definition. So, our definition is the following.

\begin{definition}
A {\it quiver} $Q$ is a couple of sets $Q_0,Q_1$ together with three maps $h,t:Q_1\to Q_0$ and $s:Q_0\to Q_1$ satisfying $hs={\sf id}=ts.$ 
\end{definition}

The elements of $s(Q_0)$ are called {\it degenerated arrows} (and when we draw pictures of quivers we don't draw them, we identify them with vertices), and elements of $Q_1\setminus s(Q_0)$ are called non-degenerated arrows. We also set 
\begin{equation}
    Q_1^{\sf D}=s(Q_0), \hspace{1cm} Q_1^{\sf N}=Q_1\setminus Q_1^{\sf D}.
\end{equation}

We often think about quivers as about categories without compositions but with identity morphisms, which are the degenerated arrows. Sometimes we will use notation $1_v=s(v).$ Moreover, for any two vertices $u,v\in Q_0$ we will also use the following notation
\begin{equation}
Q(u,v)=\{\alpha\in Q_1\mid t(\alpha)=u, h(\alpha)=v\}.
\end{equation}

\begin{definition}[Morphism of quivers]
A {\it morphism of quivers} $f:Q\to R$ is a couple of maps $f_0:Q_0\to R_0$ and $f_1:Q_1\to R_1$ such that $hf_1=f_0h, tf_1=f_0t$ and $sf_0=f_1s.$ Note that in this definition of a morphism we allow the situation when a non-degenerate edge maps to a degenerate map. Intuitively this means that ``an edge can be mapped to a vertex''.  The need to consider such morphisms has led us to define the quiver in this way. The category of quivers is denoted by ${\sf Quiv}.$ One can note that ${\sf Quiv}$ is equivalent to the full subcategory of $1$-dimensional simplicial sets; or to the category of $1$-truncated simplicial sets. 
\end{definition}

\subsection{Paths in a quiver}
For any $n\geq 0$ we define a quiver ${\sf q}^n$ such that ${\sf q}^n_0=\{0,1,\dots,n\}$ and 
\begin{equation}
{\sf q}^n_1=\{(0,0),(0,1),(1,1),(1,2),(2,2),\dots,(n-1,n),(n,n)\},    
\end{equation}
where $h(n,m)=m,$  $t(n,m)=n$ and $1_i=(i,i)$ \begin{equation}
\qq^n: \hspace{1cm} 0 \to 1 \to {\dots} \to n.
\end{equation}
In particular, $\qq^0$ is the one-point quiver. Note that 
\begin{equation}
(\qq^n)_1^{\sf N}=\{(0,1),(1,2),\dots,(n-1,n)\}, \hspace{5mm} (\qq^n)_1^{\sf D} = \{(0,0),(1,1),\dots,(n,n)\}.
\end{equation}
For a quiver $Q$ a morphism  $\alpha:\qq^n\to Q$ is defined by a sequence of arrows $\alpha_0,\dots,\alpha_{n-1}$ $\in Q_1$ such that $h(\alpha_i)=t(\alpha_{i+1}),$ where $\alpha_i=\alpha((i,i+1)).$ Such a morphism $\qq^n\to Q$ is called $n$-path of $Q$. The set of $n$-paths is denoted by 
\begin{equation}
{\sf P}_nQ=\Quiv(\qq^n,Q).
\end{equation}
If at least one of the edges $\alpha_i$ is degenerated, the path $\alpha$ is called degenerated, otherwise it is called non-degenerated. This gives us a partition of the set of all paths in two subsets, non-degenerated and degenerated paths:
\begin{equation}
{\sf P}_nQ={\sf P}_n^{\sf N}Q\sqcup {\sf P}_n^{\sf D}Q.
\end{equation}

\subsection{Box product of quivers} 
The category of quivers has obvious product $Q\times R$  which is defined component-wise $(Q\times R)_1=Q_1\times R_1$ and $(Q\times R)_0=Q_0\times R_0.$ In the graph theory this categorical product is known as the strong product. However, we will be interested in another monoidal structure on the category of quivers that we call box product, which is also known as  ``Cartesian product'' in the graph theory. The box product of two quivers $Q$ and $Q'$ is defined as a subquiver in the product
\begin{equation}
    Q\square R \subseteq Q\times R
\end{equation}
such that 
\begin{equation}
 (Q\square R)_0=Q_0\times R_0, \hspace{1cm} (Q\square R)_1 = (Q_1\times R_1^{\sf D})\cup (Q_1^{\sf D} \times R_1).   
\end{equation}
If we treat $Q_0$ as discrete quivers, then the degeneracy map can be treated as a morphism of quivers $Q_0\to Q$ and the box product can be defined as the pushout in the category of quivers
\begin{equation}
\begin{tikzcd}
Q_0\times R_0\ar[r] \ar[d] & Q_1 \times R_0 \ar[d] \\ 
Q_0\times R_1 \ar[r] & Q\square R.
\end{tikzcd}
\end{equation}

For example, the quiver $\qq^4\times \qq^2$ can be drawn as 
\begin{equation}
\begin{tikzcd}[scale cd=0.3]
\bullet\ar[r]\ar[d] \ar[rd] & \bullet \ar[r]\ar[d] \ar[rd] & \bullet \ar[r]\ar[d]\ar[rd] & \bullet\ar[r]\ar[d]\ar[rd] & \bullet \ar[d] \\
\bullet\ar[r]\ar[d]\ar[rd] & \bullet\ar[r]\ar[d]\ar[rd] & \bullet\ar[r]\ar[d]\ar[rd] & \bullet\ar[r]\ar[d]\ar[rd] & \bullet\ar[d] \\
\bullet\ar[r] & \bullet\ar[r] & \bullet\ar[r] & \bullet\ar[r] & \bullet 
\end{tikzcd}    
\end{equation}
and $\qq^4\square \qq^2$ can be drawn as follows.
\begin{equation}
\begin{tikzcd}[scale cd=0.3]
\bullet\ar[r]\ar[d]  & \bullet \ar[r]\ar[d]  & \bullet \ar[r]\ar[d] & \bullet\ar[r]\ar[d] & \bullet \ar[d] \\
\bullet\ar[r]\ar[d] & \bullet\ar[r]\ar[d] & \bullet\ar[r]\ar[d] & \bullet\ar[r]\ar[d] & \bullet\ar[d] \\
\bullet\ar[r] & \bullet\ar[r] & \bullet\ar[r] & \bullet\ar[r] & \bullet
\end{tikzcd}    
\end{equation}

\section{Path objects} 

\subsection{Simplicial indexing category \texorpdfstring{$\Delta$}{Δ}}

We denote by $\Delta$ the {\it simplicial indexing category}, whose objects are non-empty finite ordinals $[n]=\{0,\dots,n\}, n\geq 0$ and morphisms are order-preserving maps. It is well known that it is generated by two types of morphisms: coface maps $d^{(i,n)}:[n-1]\to [n]$ and codegeneracy maps $s^{(i,n)}:[n+1]\to [n]$ for $0\leq i\leq n$. The coface map $d^{(i,n)}$
is the only injective order-preserving map whose image does not contain $i$, and codegeneracy map $s^{(i,n)}$ is the only surjective order-preserving map such that $i$ has two preimages. When $n$ is obvious from the context, these maps are denoted by $d^{(i)}$ and $s^{(i)}.$ 
The category $\Delta$ can be also described as the category generated by the coface and codegeneracy maps modulo  relations:
\begin{align} \label{eq_simp_ident_1}
    d^{(j,n)} d^{(i,n-1)} &= d^{(i,n)} d^{(j-1,n-1)}, \hspace{1cm} \text{if } i<j;\\ \label{eq_simp_ident_2}
     s^{(j,n)}s^{(i,n+1)}&=s^{(i,n)}s^{(j+1,n+1)} , \hspace{1cm} \text{if } i\leq j;\\
    \label{eq_simp_ident_3}
    s^{(j,n)}d^{(i,n+1)} &= \begin{cases}
    d^{(i,n)} s^{(j-1,n-1)},& \text{if } i<j; \\ 
    {\sf id},&\text{if } i\in \{ j,j+1\}; \\ 
    d^{(i-1,n)}s^{(j,n-1)}, & \text{if } j+1<i.
    \end{cases}
\end{align}
(see \cite[\S 2.2]{gabriel2012calculus}).
Moreover, any morphism $f: [n]\to [m]$ in $\Delta$ can be uniquely presented as  
\begin{equation}\label{eq_decomp_delta}
f=d^{(i_0,m)}d^{(i_1,m-1)} \dots d^{(i_{k},m-k)} s^{(j_0,n-l)} s^{(j_2,n-l+1)} \dots s^{(j_l,n)}     
\end{equation}
for $m\geq i_1 > \dots > i_k \geq 0$ and $0\leq j_1<\dots <j_l\leq n-1.$

\subsection{Path indexing category \texorpdfstring{$\Pi$}{Π}} \label{subsection:path_indexing_cat}

A subset of $[n]$ is said to be {\it connected}, if it has the form $\{k,k+1,\dots,l-1,l\}$ for some $0\leq k\leq l\leq  n.$ An order preserving map $f: [n]\to [m]$ is called connected, if its image is connected. Equivalently, a connected order preserving map $f:[n]\to [m]$ is an order preserving map such that for any $0\leq i<n$ either $f(i+1)=f(i)+1$ or $f(i+1)=f(i).$ We denote by $\Pi$ a wide subcategory of $\Delta,$ whose morphisms are connected order preserving maps.  It is easy to check that the composition of connected order preserving maps is connected, so this subcategory is well defined. The category $\Pi$ will be refereed as {\it the path indexing category}. For example all codegeneracy maps are in $\Pi$ and the {\it exterior} coface maps $d^{(0,n)},d^{(n,n)}:[n-1]\to [n]$ are also in $\Pi.$ The exterior coface maps are denoted by  
\begin{equation}
t^{(n)}=d^{(0,n)}, \hspace{1cm} h^{(n)}=d^{(n,n)}.    
\end{equation}
It is easy to check that all morphisms in $\Pi$ are compositions of codegeneracy maps and exterior coface maps.  

It is well known that the category $\Delta$ is equivalent to the full subcategory of the category of small categories ${\sf Cat},$ whose objects are free categories generated by the quivers  $\qq^n.$ A similar statement holds for the category $\Pi,$ it is equivalent to a full subcategory of $\Quiv$ whose objects are the quivers $\qq^n.$ To be more precise we formulate it as follows.

\begin{proposition}\label{prop:q_embedding}
There is a fully faithful functor
\begin{equation}
    {\sf q} : \Pi \to {\sf Quiv}, \hspace{1cm} [n] \mapsto {\sf q}^n
\end{equation}
such that ${\sf q}(f)_0=f$ and ${\sf q}(f)_1(i,j)=(f(i),f(j)).$ 
\end{proposition}
\begin{proof}
The proof is standard.
\end{proof}

\begin{corollary}\label{cor_Pi_Quiv}
The category $\Pi$ is isomorphic to the full subcategory of ${\sf Quiv},$ whose objects are ${\sf q}^n.$
\end{corollary}

 Note that the equations \eqref{eq_simp_ident_1}, \eqref{eq_simp_ident_2}, \eqref{eq_simp_ident_3} imply the following relations 
\begin{alignat}{2} \label{eq_path_ident_1}
    t^{(n)} h^{(n-1)} &= h^{(n)}t^{(n-1)};\\
    \label{eq_path_ident_2}
     s^{(j,n)}s^{(i,n+1)}&=s^{(i,n)}s^{(j+1,n+1)}, &  \text{ if } i&\leq j;\\
    \label{eq_path_ident_3}
    s^{(i,n)}h^{(n+1)} &= 
    h^{(n)} s^{(i-1,n-1)}, &  \text{ if } i&>0; 
    \\
    \label{eq_path_ident_4}
    s^{(i,n)}t^{(n+1)} &= 
    t^{(n)}s^{(i,n-1)}, &  \text{ if } i&<n, \\
    \label{eq_path_ident_5}
    s^{(n,n)}t^{(n+1)}&={\sf id}=s^{(0,n)}h^{(n+1)}
\end{alignat}
    
\begin{lemma} \label{lemma:relations}
The category $\Pi$ is generated by the morphisms $h,t,s^{(i)}$ modulo relations \eqref{eq_path_ident_1}, \eqref{eq_path_ident_2}, \eqref{eq_path_ident_3}, \eqref{eq_path_ident_4}. Moreover, any morphism in $\Pi$ can be uniquely presented as
\[h^kt^l s^{(i_1)}s^{(i_2)} \dots s^{(i_m)},\]
where $i_1<i_2<\dots <i_m$ and $m\geq 0.$
\end{lemma}
\begin{proof}
It is easy to see that any morphism $f$ in $\Delta$ can be uniquely decomposed as $\alpha \sigma,$ where $\alpha$ is injective and $\sigma$ is surjective. Moreover, the image of $\alpha\sigma$ is equal to the image of $\alpha,$ and hence, $\alpha\sigma$ is in $\Pi$ if and only if $\alpha$ is in $\Pi.$ It is easy to see that any injective map $\alpha$ in $\Pi$ can be uniquely presented as $h^kt^l$ and it is well known that any surjective map in $\Delta$ can be uniquely presented as a composition $s^{(i_1)}s^{(i_2)} \dots s^{(i_m)},$ where $i_1<i_2< \dots<i_m$ \cite[\S 2.2]{gabriel2012calculus}. In particular, $\Pi$ is generated by the maps $h,t,s^{(i)}.$ 

Denote by $\Pi'$ the category with the same objects as in $\Pi$ generated by the morphisms $t,h,s^{(i)}$  modulo relations \eqref{eq_path_ident_1}, \eqref{eq_path_ident_2}, \eqref{eq_path_ident_3}, \eqref{eq_path_ident_4}, \eqref{eq_path_ident_5}. 
Since these relations hold in $\Pi$ and $\Pi$ is generated by $h,t,s^{(i)}$ we obtain a full functor $\Pi'\to \Pi.$ Analysing the relations it is easy to check that any morphism in $\Pi'$ can be also presented in the form $h^kt^ls^{(i_1)}s^{(i_2)} \dots s^{(i_m)},$ where $i_1<i_2< \dots<i_m.$ It follows that this functor is also faithful, and hence, it is an isomorphism. 
\end{proof}

\subsection{Path sets}

If $\mathcal{C}$ is a category, a path object in $\mathcal{C}$ is a functor $P:\Pi^{\sf op}\to \mathcal{C}.$  The morphism $s_i=P(s^{(i)})$ are called degeneracy maps of $P$ and the maps $t_n=P(t^{(n)})$ and $h_n=P(h^{(n)})$ are called {\it exterior face maps}. The morphism  $t_n:P_n\to P_{n-1}$ will be also called {\it the tail map} and the map $h_n:P_n\to P_{n-1}$ will be called {\it the head map}. 
The category of path objects will be denoted by ${\sf p}\CCC.$ Lemma \ref{lemma:relations} implies that a path object can be defined as a sequence of objects $P_0,P_1,\dots  $ together with morphisms $h_n,t_n:P_n\to P_{n-1}$ and $s_i:P_n\to P_{n+1}$ for $0\leq i\leq n$ satisfying the following relations
\begin{alignat}{2}
     h_{n-1}t_{n} &= t_{n-1}h_{n},\\
     s_{i}s_{j}&=s_{j+1}s_{i} , & \text{if } i&\leq j;\\
    h_{n+1}s_{i} &= s_{i-1} h_{n}, & \text{if } i&>0; \\
    t_{n+1} s_{i} &=s_{i}t_{n}, & \text{if } i&<n, \\   
    t_{n+1} s_n &={\sf id}=h_{n+1}s_0.
\end{alignat}
And a morphism $f:P\to Q$ of path objects is a collection of morphisms $f_n:P_n\to Q_n$ commuting with these structure morphisms.

A path object in the category of sets is called path set. The category of path sets will be denoted by ${
\sf pSets}.$

\begin{example}
Any simplicial set $X$ defines a path set via the composition with $\Pi\hookrightarrow \Delta.$ By abuse of notation we denote the path set by the same letter $X.$  Since any morphism of $\Pi$ is the composition of codegeneracy maps and exterior coface maps, a collection of subsets $P_n\subseteq X_n$ is a path subset if and only if it is closed with respect to the degeneracy and exterior face maps. Many examples of path sets arise naturally as path subsets of simplicial sets. However, not any path set can be embedded into a simplicial set (Proposition \ref{prop:nonemb}).
\end{example}

\begin{example}\label{example:p}
For any $m$ we consider the path set ${\sf p}^m$ defined  by the formula
\begin{equation}
{\sf p}^m := \Pi(-,[m]).
\end{equation}
Then $n$th component of ${\sf p}^m$ consists of connected order preserving maps $[n]\to [m].$
This defines a functor
\begin{equation}
{\sf p} : \Pi \longrightarrow {\sf pSets}, \hspace{1cm} [m]\mapsto {\sf p}^{m}.
\end{equation}
In particular, for any connected order preserving map $f:[m]\to [l]$ we have a morphism of path sets ${\sf p}^f:{\sf p}^m \to {\sf p}^l.$
\end{example}

\begin{example}\label{ex:PQ}
Any quiver $Q$ defines a path set
\begin{equation}
{\sf P}Q=\Quiv(\qq^{(-)},Q),
\end{equation}
whose $n$-th component is the set of $n$-paths ${\sf P}_nQ.$ This construction is similar to the construction of nerve of a category.
\end{example}

The degeneracy maps for ${\sf P}_nQ$ act as follows 
\begin{equation}
s_i(\alpha_0,\dots,\alpha_{n-1})=(\alpha_0,\dots,\alpha_{i-1},1_{v_i},\alpha_i,\dots,\alpha_{n-1}),
\end{equation}
where $v_i=t(\alpha_{i-1})$ for $i<n$ and $v_n=h(\alpha_{n-1}).$ This description implies the following lemma.
\begin{lemma}\label{lemma:degenerated_path}
Let $\alpha\in {\sf P}_nQ$ and $\mu=(\mu_0,\dots,\mu_{k-1}),$ where $0\leq \mu_0<\dots < \mu_{k-1}\leq n-1.$ Then the following conditions are equivalent
\begin{itemize}
\item $\alpha_{\mu_i}\in Q_1^{\sf D}$ for any $i$;
\item $\alpha=s_\mu(\alpha')$ for some $\alpha'\in {\sf P}_{n-k}Q.$
\end{itemize}
\end{lemma}
\begin{proof}
The proof is straightforward. 
\end{proof}

\subsection{A path set non-embeddable to a simplicial set}

Consider a path set $E$ defined as the pushout in ${\sf pSets}:$ 
\begin{equation}
\begin{tikzcd}
{\sf p}^3 \ar[r,"{\sf p}^{s^1}"] \ar[d,"{\sf p}^{s^1}"'] & {\sf p}^2 \ar[d,"i_1"] \\
{\sf p}^2 \ar[r,"i_2"] & E
\end{tikzcd}    
\end{equation}
(for the definition of ${\sf p}^n$ see Example \ref{example:p}).

\begin{lemma}\label{lemma:s_1_not_inj}
The map $s_1:E_2 \to E_3$ is not injective.
\end{lemma}
\begin{proof} Set $e_1=i_1(1_{[2]})\in E_2$ and $e_2=i_2(1_{[2]})\in E_2.$ It is sufficient to prove that $e_1 \ne e_2$ and $s_1(e_1) = s_1(e_2).$

Prove that $s_1(e_1)=s_1(e_2).$ It follows from the computation:
\begin{equation}
\begin{split}
    s_1(e_1) &= s_1(i_1(1_{[2]})) = i_1(s_1(1_{[2]})) = i_1(s^1) \\
    & = i_1( {\sf p}^{s^1}(1_{[3]})) = i_2({\sf p}^{s^1}(1_{[3]}) ) \\
    & = i_2(s^1) = i_2(s_1(1_{[2]}))=s_1(i_1(1_{[2]}))=s_1(e_2).
\end{split}    
\end{equation}

Prove that $e_1\ne e_2.$ Note that in the category of sets for any pushout diagram 
\begin{equation}
\begin{tikzcd}
S_0\ar[r,"f_1"] \ar[d,"f_2"] & S_1  \ar[d,"i_1"] \\ 
S_2 \ar[r,"i_2"] & S,
\end{tikzcd}
\end{equation}
if $s_1\in S_1$ and $s_2\in S_2$ are elements such that $s_1\notin f_1(S_0)$ and $s_2\notin f_2(S_0),$ then $i_1(s_1)\ne i_2(s_2).$ Since the pushout in the category functors ${\sf pSets}={\sf Funct}(\Pi^{\sf op},{\sf Sets})$ is defined object-wise, we just need to prove that $1_{[2]} $ not in the image of  $ ({\sf p}^{s^1})_2:({\sf p}^{3})_2\to ({\sf p}^2)_2.$ In other words, we need to prove that $1_{[2]}$ can't be presented as $1_{[2]}=s^1f,$ where $f:[2]\to [3]$ is a connected order-preserving map. Indeed, it is easy to check that the only two order preserving maps $f:[2]\to [3]$ satisfying $1_{[2]}=s^1f$ are the maps $f=d^1:[2]\to [3]$ and $f=d^2:[2]\to [3],$ and they are not connected.  
\end{proof}

\begin{proposition}\label{prop:nonemb}
The path set $E$ can't be embedded to a simplicial set. 
\end{proposition}
\begin{proof}
For any simplicial set $X$ the maps $s_1:X_2\to X_3$ and $d_1:X_3\to X_2$ satisfy the relation $d_1s_1=1_{X_2}.$ It follows that $s_1:X_2 \to X_3$ is injective. Then the assertion follows from Lemma \ref{lemma:s_1_not_inj}.
 \end{proof}

\section{Combinatorics of pairs of connected maps}

Further in the discussion of path pairs of modules, we will need some combinatorics of pairs of maps from $\Pi$ and shuffles. We decided to make a separate section about this combinatorics. 

\subsection{Kernel and image of an order-preserving map}

For an order preserving map  $f:[n]\to X$ to a poset $X$ we set \begin{equation}
    \Im(f):=f([n]), \hspace{1cm} \Ker(f):=\{i\in [n-1]\mid f(i)=f(i+1) \}.
\end{equation}
If $f$ is an order-preserving map $f:[n]\to [m]$ is decomposed as
\begin{equation}
f=d^{(i_0)}d^{(i_1)} \dots d^{(i_{k})} s^{(j_0)} s^{(j_2)} \dots s^{(j_l)}.
\end{equation}
(see \eqref{eq_decomp_delta}), then it is easy to check that 
$\Ker(f)=\{j_0,j_1,\dots,j_l\}.$
Any order preserving map 
$f:[n]\to X$ can be uniquely presented as 
\begin{equation}\label{eq_epi-mono}
f=f'\sigma,    
\end{equation}
where $\sigma:[n]\epi [n']$ is a surjective order-preserving map and $f':[n']\mono X $ is injective order preserving map so that $\Ker(f)=\Ker(\sigma)$.

\subsection{Pairs of connected maps}
Further we will need to consider various sets of pairs of maps from $\Pi.$ In this section we introduce notations for them and explain their meaning on the language of paths in quivers $\qq^k\times \qq^l$ and $\qq^k\square \qq^l.$

For any $n,k,l\geq 0$ we set
\begin{equation}
\PPi(n;k,l)=\Pi([n],[k])\times \Pi([n],[l])    
\end{equation}
Since $\Pi([n],[m])\cong \Quiv(\qq^n,\qq^m),$ we see that 
\begin{equation}\label{eq_PPi_iso}
\PPi(n;k,l) \cong 
{\sf P}_n(\qq^k\times \qq^l).  \end{equation}

Further we set
\begin{equation}\label{eq:PPi_square}
\PPi_\square(n;k, l)= \{(f,g)\in \PPi(n;k,l) \mid \Ker(f) \cup \Ker(g)=[n-1] \}.    
\end{equation}

\begin{proposition}\label{proposition:square_nerve} The isomorphism \eqref{eq_PPi_iso} induces an isomorphism 
\[ \PPi_\square(n;k, l)\cong 
{\sf P}_n(\qq^k\square \qq^l). \]
\end{proposition}
\begin{proof}
Any pair 
$(f:[n]\to [k] , g:[n]\to [l])$ of morphisms from $\Pi$ 
defines an $n$-path in $\qq^k\times \qq^l$ given by \[( ((f(0),f(1)),(g(0),g(1))), \dots ((f(n-1),f(n)),(g(n-1),g(n))) ).\] This path lies in $\qq^k\square \qq^l$ iff for each $0\leq i\leq n-1$ either $f(i)=f(i+1)$ or $g(i)=g(i+1).$ In other words, this path in $\qq^k\square \qq^l$ iff for any $i\in [n-1] $ either $i\in\Ker(f)$ or $i\in \Ker(g).$
\end{proof}

The Proposition \ref{proposition:square_nerve} implies that $\PPi_\square$ defines a functor
\begin{equation}
 \PPi_\square : \Pi^{op} \times \Pi \times \Pi \longrightarrow {\sf Set},
\end{equation}
that sends $([n],[k],[l])$ to $\PPi_\square(n;k,l)$. 

\begin{remark} In Appendix (Section \ref{Appendix}) we show that the functor  $\PPi_\square$ has some categorical meaning. Namely it defines a structure of promonoidal category on $\Pi.$
\end{remark}

Let $n=k+l,$ where $k,l\geq 0.$ We define an $(k,l)$-shuffle as a pair $(\mu,\nu),$ where $\mu=(\mu_0,\dots,\mu_{k-1})$ and $\nu=(\nu_0,\dots,\nu_{l-1})$ are strictly increasing sequences of numbers from $\{0,\dots,n-1\}$ such that $\{\mu_0,\dots,\mu_{k-1}\}\cup \{\nu_0,\dots,\nu_{l-1}\}=\{0,\dots,n-1\}.$ The set of $(k,l)$-shuffles is denoted by ${\sf Sh}(k,l).$ For any $(\mu,\nu)\in {\sf Sh}(k,l)$ we consider a couple  $(s^\nu:[n]\to [k],s^\mu:[n]\to [l])$ given by 
\begin{equation}
    s^\nu = s^{\nu_0} \dots s^{\nu_{l-1}}, \hspace{1cm} s^\mu = s^{\mu_0}\dots s^{\mu_{k-1}}.
\end{equation}
Note that $\Ker(s^{\nu})=\{\nu_0,\dots,\nu_{l-1}\}$ and $\Ker(s^\mu)=\{\mu_0,\dots, \mu_{k-1}\}.$ Hence $(s^\nu,s^\mu)\in \PPi_\square^{\sf N}(n;k,l).$

\begin{lemma}
For any $(f,g)\in \PPi_\square(n;k,l)$ there exists a unique data set consisting of 
\begin{itemize}
    \item natural numbers $k',l',k'';$
    \item a shuffle $(\mu,\nu)\in {\sf Sh}(k',l');$
    \item a surjective order preserving map $\sigma:[k']\epi [k''];$
    \item injective order preserving maps $\alpha:[k'']\mono [k] $ and $\beta:[l']\mono [l]$
\end{itemize}
such that $n=k'+l'$ and 
\begin{equation}\label{eq:standard_decomposition}
    (f,g)=(\alpha \sigma s^\nu, \beta s^\mu).
\end{equation}
\begin{equation}
\begin{tikzcd}
& & & {[n]}\ar[ld,"s^\nu"',twoheadrightarrow] \ar[rd,"s^\mu",twoheadrightarrow]  \ar[llld,bend right=5mm,"f"'] \ar[rrrd,bend left=5mm,"g"] & & & \\
{[k]} & {[k'']}\ar[l,"\alpha",rightarrowtail] & {[k']} \ar[l,"\sigma",twoheadrightarrow] &  & {[l']}\ar[rr,"\beta"',rightarrowtail] & \phantom{[k'']} & {[l]}
\end{tikzcd}    
\end{equation}
The decomposition \eqref{eq:standard_decomposition} will be called the standard decomposition of $(f,g).$
\end{lemma}
\begin{proof}
Take the epi-mono decomposition of $g=\beta s^\mu,$ where $\mu=(\mu_0,\dots,\mu_{k'-1})$ is a strictly increasing sequence and $s^\mu = s^{\mu_0}\dots s^{\mu_{k'-1}}.$ Then there is a unique shuffle $(\mu,\nu)\in {\sf Sh}(k',l'),$ where $l'=n-k'.$ Since $ \Ker(g)=\Ker(s^\mu)=\{\mu_0,\dots,\mu_{k-1}\}$ and $\Ker(f)\cup \Ker(g)=[n-1],$ we obtain $\{\nu_0,\dots,\nu_{l'-1}\}\subseteq \Ker(f).$ It follows that $f=f' s^\nu$ for some order-preserving map $f'$. Then we take the epi-mono decomposition of $f'$ and obtain $f'= \alpha \sigma.$  Note that $k',$ $\mu$ and $\beta$ are uniquely defined by $g;$ $f=\alpha (\sigma s^\nu)$ is the epi-mono decomposition of $f,$ so $\alpha$ and the composition $\sigma s^\nu$ are uniquely defined by $f;$ $ \nu$ is uniquely defined by $\mu;$ $\sigma$ is uniquely defined by $\nu$ and the composition $\sigma s^\nu.$   
\end{proof}

We will also need not only pairs of maps but also pairs of surjections. Denote by $\Pi^-([n],[k])$ the set of surjective order preseving maps $[n]\epi [k]$. Then we set  
\[\PS(n;k, l)= \Pi^-([n],[k])\times \Pi^-([n],[l]).\]
It is easy to see that the elements of this set correspond to $n$-paths in $ \qq^k\times \qq^l$ starting in $(0,0)$ and ending in $(k,l).$
We also consider the following set 
\begin{equation}
\PS_\square(n,k,l)=\PS(n;k,l) \cap \PPi_\square(n;k,l),
\end{equation}
that corresponds to the set of $n$-paths in $\qq^k\square \qq^l$ starting in starting in $(0,0)$ and ending in $(k,l).$ Since all paths can be degenerated and non-degenerated, for all these sets we can also consider the corresponding degenerated and non-degenerated versions. For example:
\begin{equation}
\begin{split}
\PS^{\sf D}(n;k,l)&=\{(\sigma,\tau)\in \PS(n;k,l) \mid \Ker(\sigma)\cap \Ker(\tau)\ne \emptyset \}, \\ 
 \PS^{\sf N}(n;k,l)&=\{(\sigma,\tau)\in \PS(n;k,l) \mid \Ker(\sigma)\cap \Ker(\tau)= \emptyset \}, \\ 
 \PS_\square^{\sf D}(n;k,l)&=\PS_{\square}(n;k,l) \cap \PS^{\sf D}(n;k,l),
 \\
 \PS_\square^{\sf N}(n;k,l)&=\PS_\square(n;k,l) \cap \PS^{\sf N}(n;k,l).\\
\end{split}    
\end{equation}

We will also need notations for the unions of all these sets by $(k,l).$ For example 
\begin{equation}
\PPi_\square(n)=\coprod_{k,l} \PPi_\square(n;k,l), \hspace{1cm} \PS_\square(n)=\coprod_{k,l} \PS_\square(n;k,l).
\end{equation}

\begin{lemma}\label{lemma_shuffles} Let $(f,g)\in \PPi_\square(n)$ and let $(f,g)=(\alpha \sigma s^\nu, \beta s^\mu )$ be its standard decomposition \eqref{eq:standard_decomposition}. Then
\begin{itemize}
    \item $(f,g) \in \PS_\square(n)$ if and only if $\alpha={\sf id}$ and $\beta={\sf id};$
    \item $(f,g)\in \PS_\square^{\sf N}(n)$ if and only if $\alpha={\sf id},$  $\beta={\sf id}$ and $\sigma={\sf id}.$
\end{itemize}
In particular, $\PS_\square^{\sf N}(k+l;k,l)=\{(s^\nu,s^\mu)\mid (\mu,\nu)\in {\sf Sh}(k,l)\};$  and $\PS_\square^{\sf N}(n;k,l)=\emptyset,$ if $n\neq k+l.$
\end{lemma}
\begin{proof}
The proof is straightforward. 
\end{proof}

\subsection{Paths in the box product}

Since $Q\times R$ is the  product of $Q$ and $R$ in the sense of category theory we have a bijection for the sets of paths
\begin{equation}\label{eq:path_product}
    {\sf P}_n(Q\times R)\cong {\sf P}_nQ\times {\sf P}_nR.
\end{equation}

\begin{proposition}\label{prop:paths_in_product}  The bijection \eqref{eq:path_product}  induces a bijection
\begin{equation} \label{eq:X1}
{\sf P}_n(Q\square R) \cong \bigcup_{k+l=n}   \bigcup_{(\mu,\nu)\in {\sf Sh}(k,l)} s_\nu({\sf P}_kQ) \times s_\mu({\sf P}_lR).
\end{equation}
The right hand part can be also rewritten as 
\begin{equation} \label{eq:X2}
    {\sf P}_n(Q\square R) \cong   \bigcup_{(\sigma,\tau)\in \PS_\square(n)} \sigma^*({\sf P}_{|\sigma|}Q) \times \tau^*({\sf P}_{|\tau|}R).
\end{equation}
and as
\begin{equation}\label{eq:X3}
    {\sf P}(Q\square R)_n \cong   \bigcup_{(f,g)\in \PPi_\square(n)} f^*({\sf P}_{|f|}Q) \times g^*({\sf P}_{|g|}R).
\end{equation} 
\end{proposition}
\begin{proof} Denote by ${\sf P}'_n(Q\square R)$ the image of $ {\sf P}_n(Q\square R)$ in ${\sf P}_nQ\times {\sf P}_nR.$ Denote by $X_1,X_2,X_3\subseteq {\sf P}_nQ\times {\sf P}_nR$ the right hand parts of the equations \eqref{eq:X1}, \eqref{eq:X2}, \eqref{eq:X3}. So we need to prove that ${\sf P}'_n(Q\square R)=X_1=X_2=X_3.$ The inclusions $X_1\subseteq X_2\subseteq X_3 $ are obvious. Hence, it is sufficient to prove that $ X_3\subseteq {\sf P}'_n(Q\square R)$  and ${\sf P}'_n(Q\square R)\subseteq X_1.$

Let $(\alpha,\beta)\in {\sf P}_nQ\times {\sf P}_nR.$ Then $(\alpha,\beta)\in {\sf P}'_n(Q\square R)$ if and only if for each $i$ we have either $\alpha_i\in Q_1^{\sf D}$ or $\beta_i\in R_1^{\sf D}.$ 

If $(f,g)\in \PPi_\square(n),$ then we take the standard decomposition $(f,g)=(\alpha \sigma s^\mu,\beta s^\nu)$ (see \eqref{eq:standard_decomposition}).  Therefore, by Lemma \ref{lemma:degenerated_path}, we obtain 
$X_3\subseteq {\sf P}'_n(Q\square R).$ Let $(\alpha,\beta)\in {\sf P}'_n(Q\square R).$ 
Then there exists a shuffle $(\mu,\nu)\in {\sf Sh}(k,l),$ for some $k+l=n,$ such that $\alpha_{\nu_i}\in Q_1^{\sf D}$ for any $i$ and $\beta_{\mu_j}\in R_1^{\sf D}$ for any $j.$  By Lemma \ref{lemma:degenerated_path} it follows that $\alpha=s_\nu(\alpha')$ and $\beta=s_\mu(\beta')$ for some $\alpha'\in {\sf P}_kQ$ and $\beta'\in {\sf P}_lR.$ It follows that 
${\sf P}'_n(Q\square R)
\subseteq X_1.$
\end{proof}

\subsection{Graph of shuffles}
Now we are going to define a structure of weighted digraph on the set of shuffles ${\sf Sh}(k,l).$ Recall that we define an $(k,l)$-shuffle as a pair $(\mu,\nu),$ where $\mu=(\mu_0,\dots,\mu_{k-1})$ and $\nu=(\nu_0,\dots,\nu_{l-1})$ are strictly increasing sequences of numbers from $\{0,\dots,k+l-1\}$ such that $\{\mu_0,\dots,\mu_{k-1}\}\cup \{\nu_0,\dots,\nu_{l-1}\}=\{0,\dots,k+l-1\}.$

A good intuitive treatment of a shuffle is a path on a lattice.  For any $(k,l)$-shuffle $(\mu,\nu)$ we can consider the couple $(s^\nu,s^\mu)\in \PS_\square(k+l,k,l).$ Since, elements of  $\PS_\square(k+l,k,l)$  correspond to paths in $\qq^k\square \qq^l$ starting in $(0,0)$ end ending in $(k,l),$ each $(k,l)$-shuffle corresponds to a path, whose $i$-th point is $( \nu^{<i}, \mu^{<i}),$ where $\nu^{<i}$ is the number of indexes $j$ such that   $\nu_j<i$ and $\nu^{<i}$ is the number of indexes $j$ such that $\mu_j<i:$ 
\begin{equation}
\nu^{<i}= |\{0\leq j \leq k-1 \mid \nu_j<i\}|, \hspace{1cm} \mu^{<i}=|\{0 \leq j\leq l-1 \mid \mu_j<i\}|. 
\end{equation}

For example, the $(4,3)$-shuffle $((0,2,3,5),(1,4,6))$ corresponds to the following path in $\qq^3\times \qq^4$:

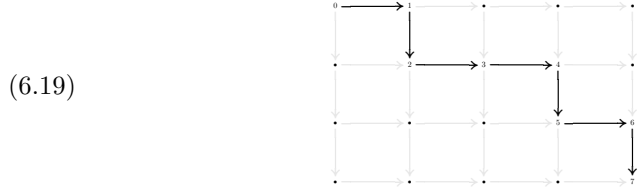
\begin{equation}
\begin{tikzcd}[scale cd=0.3] 
 0 \arrow{r}\arrow[color1]{d}  & 1 \arrow[color1]{r}\arrow{d}  & \bullet \arrow[color1]{r}\arrow[color1]{d} & \bullet\arrow[color1]{r}\arrow[color1]{d} & \bullet \arrow[color1]{d} 
\\
 \bullet \arrow[color1]{r}\arrow[color1]{d} & 
2 \arrow{r}\arrow[color1]{d} & 3 \arrow{r}\arrow[color1]{d} & 4 \arrow[color1]{r}\arrow{d} & \bullet\arrow[color1]{d} 
\\
 \bullet\arrow[color1]{r}\arrow[color1]{d} & \bullet\arrow[color1]{r}\arrow[color1]{d} & \bullet\arrow[color1]{r}\arrow[color1]{d} & 5\arrow{r}\arrow[color1]{d} & 6\arrow{d} 
\\
 \bullet\arrow[color1]{r} & \bullet\arrow[color1]{r} & \bullet\arrow[color1]{r} & \bullet\arrow[color1]{r} & 7
\end{tikzcd}    
\end{equation}

\

An {\it elementary inversion} of a $(k,l)$-shuffle $(\mu,\nu)$ is a element $1\leq i\leq n-1$ such that $i-1\in \{ \nu_0,\dots,\nu_{l-1}\}$ and $i\in  \{\mu_0,\dots,\mu_{k-1}\}.$ In other words, $i$ is an elementary inversion of $(\mu,\nu)$ if it has the form 
\begin{equation}
(\mu,\nu)=((\mu_0,\dots,\mu_{r-1}, i,\mu_{r+1}, \dots,\mu_{k-1}),(\nu_0,\dots,\nu_{t-1},i-1,\nu_{t+1},\dots,\nu_{l-1})).
\end{equation}
For example, $5$ is an elementary inversion of of the shuffle $((0,2,3,5),(1,4,6)).$

Let $(\mu,\nu)$ and $(\mu',\nu')$ be two $(k,l)$-shuffles. We say that there is an edge
\begin{equation}
(\mu,\nu)\overset{i}\longrightarrow (\mu',\nu') 
\end{equation}
of weight $1\leq i\leq k+l-1$ if $i$ is an elementary inversion of $(\mu,\nu)$ and 
\begin{align}
(\mu,\nu)&=((\mu_0,\dots,\mu_{r-1},\ \ i\ \ ,\mu_{r+1}, \dots,\mu_{k-1}),(\nu_0,\dots,\nu_{t-1},i-1,\nu_{t+1},\dots,\nu_{l-1})),\\
(\mu',\nu')&=((\mu_0,\dots,\mu_{r-1}, i-1,\mu_{r+1}, \dots,\mu_{k-1}),(\nu_0,\dots,\nu_{t-1},\ \ i\ \  ,\nu_{t+1},\dots,\nu_{l-1})).
\end{align}
Note that in this case $i$ is not an elementary inversion of $(\mu',\nu').$ It is easy to see that the paths corresponding to the shuffles $(\mu,\nu)$ and $(\mu',\nu')$ deffer only in one vertex: in the $i$-th vertex. 

For example, we have the edge
\begin{equation}\label{eq:example_arrow}
((0,2,3,5),(1,4,6)) \overset{5}\longrightarrow  ((0,2,3,4),(1,5,6)).  
\end{equation}
On the level of paths it looks as follows. 
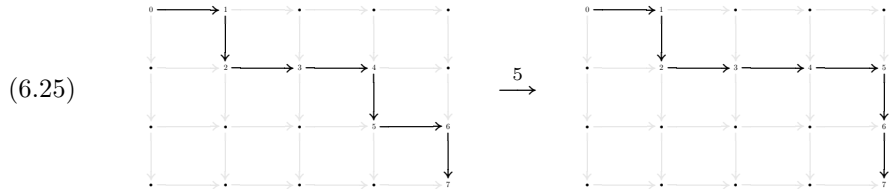
\begin{equation}
\begin{tikzcd}[scale cd=0.3]
 0 \arrow{r}\arrow[color1]{d}  & 1 \arrow[color1]{r}\arrow{d}  & \bullet \arrow[color1]{r}\arrow[color1]{d} & \bullet\arrow[color1]{r}\arrow[color1]{d} & \bullet \arrow[color1]{d} 
\\
 \bullet \arrow[color1]{r}\arrow[color1]{d} & 
2 \arrow{r}\arrow[color1]{d} & 3 \arrow{r}\arrow[color1]{d} & 4 \arrow[color1]{r}\arrow{d} & \bullet\arrow[color1]{d} 
\\
 \bullet\arrow[color1]{r}\arrow[color1]{d} & \bullet\arrow[color1]{r}\arrow[color1]{d} & \bullet\arrow[color1]{r}\arrow[color1]{d} & 5\arrow{r}\arrow[color1]{d} & 6\arrow{d} 
\\
 \bullet\arrow[color1]{r} & \bullet\arrow[color1]{r} & \bullet\arrow[color1]{r} & \bullet\arrow[color1]{r} & 7
\end{tikzcd}  
\hspace{5mm}
\overset{5}\longrightarrow
\hspace{5mm}
\begin{tikzcd}[scale cd=0.3]
 0 \arrow{r}\arrow[color1]{d}  & 1 \arrow[color1]{r}\arrow{d}  & \bullet \arrow[color1]{r}\arrow[color1]{d} & \bullet\arrow[color1]{r}\arrow[color1]{d} & \bullet \arrow[color1]{d} 
\\
 \bullet \arrow[color1]{r}\arrow[color1]{d} & 
2 \arrow{r}\arrow[color1]{d} & 3 \arrow{r}\arrow[color1]{d} & 4 \arrow{r}\arrow[color1]{d} & 5\arrow{d} 
\\
 \bullet\arrow[color1]{r}\arrow[color1]{d} & \bullet\arrow[color1]{r}\arrow[color1]{d} & \bullet\arrow[color1]{r}\arrow[color1]{d} & \bullet\arrow[color1]{r}\arrow[color1]{d} & 6\arrow{d} 
\\
 \bullet\arrow[color1]{r} & \bullet\arrow[color1]{r} & \bullet\arrow[color1]{r} & \bullet\arrow[color1]{r} & 7
\end{tikzcd}    
\end{equation}

The obtained weighted digraph is denoted by ${\bf Sh}(k,l).$ Note that for any edge $(\mu,\nu)\to (\mu',\nu')$ we have 
\begin{equation}
    {\sf sgn}(\mu,\nu)= - {\sf sgn}(\mu',\nu').
\end{equation}

\begin{lemma}\label{lemma_graph_sh}
The digraph ${\bf Sh}(k,l)$ is weakly connected. 
\end{lemma}
\begin{proof}
Note that the only $(k,l)$-shuffle that has no elementary inversions is $(\mu_0,\nu_0)=(\{0,\dots,k-1\},\{k,\dots,l+k-1\}).$ 
We prove that for any $(k,l)$-shuffle $(\mu,\nu)$ there is a path to this particular shuffle $(\mu_0,\nu_0)$. 
Denote by $i_{\sf max}(\mu,\nu)$ the maximal elementary inversion of $(\mu,\nu).$ 
Then there is an edge $(\mu,\nu)\to (\mu',\nu')$ such that either $(\mu',\nu')=(\mu_0,\nu_0)$ or $i_{\sf max}(\mu',\nu')<i_{\sf max}(\mu,\nu).$ 
The by induction on $i_{\sf max}$ we prove the assertion. 
\end{proof}

For our example \eqref{eq:example_arrow} of an arrow in ${\bf Sh}(4,3)$ it is easy to check that  $(s^\nu d^5,s^\mu d^5)=(s^{\nu'} d^5,s^{\mu'} d^5)\in \PS(6;3,4)$ and the corresponding path in the product $\qq^3\times \qq^4$  is the following path with the diagonal arrow:
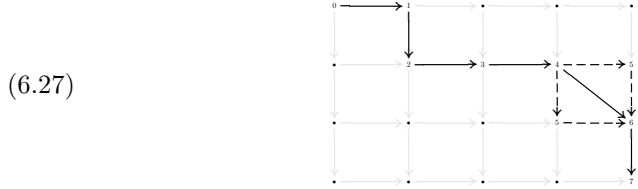
\begin{equation}
\begin{tikzcd}[scale cd=0.3] 
 0 \arrow{r}\arrow[color1]{d}  & 1 \arrow[color1]{r}\arrow{d}  & \bullet \arrow[color1]{r}\arrow[color1]{d} & \bullet\arrow[color1]{r}\arrow[color1]{d} & \bullet \arrow[color1]{d} 
\\
 \bullet \arrow[color1]{r}\arrow[color1]{d} & 
2 \arrow{r}\arrow[color1]{d} & 3 \arrow{r}\arrow[color1]{d} & 4 \arrow{rd} \arrow[dashed]{r}\arrow[dashed]{d} & 5\arrow[dashed]{d} 
\\
 \bullet\arrow[color1]{r}\arrow[color1]{d} & \bullet\arrow[color1]{r}\arrow[color1]{d} & \bullet\arrow[color1]{r}\arrow[color1]{d} & 5\arrow[dashed]{r}\arrow[color1]{d} & 6\arrow{d} 
\\
 \bullet\arrow[color1]{r} & \bullet\arrow[color1]{r} & \bullet\arrow[color1]{r} & \bullet\arrow[color1]{r} & 7
\end{tikzcd}    
\end{equation}
This gives a geometric intuition for the following lemma.

\begin{lemma} \label{lemma_mu_nu_edge}
Let $(\mu,\nu)\overset{i}\to (\mu',\nu')$ be an edge of weight $1\leq i\leq n-1$ in ${\bf Sh}(k,l),$ where $n=k+l.$ Set
\begin{equation}
(\sigma,\tau)=(s^\nu d^i, s^\mu d^i),\hspace{1cm}  (\sigma',\tau')=(s^{\nu'}d^i, s^{\mu'}d^i).  
\end{equation}
Then
\begin{itemize}    
\item $(\sigma,\tau) = (\sigma',\tau');$

\item $\sigma,\tau$ are surjections;

\item $\Ker(\sigma) \cap \Ker(\tau)=\emptyset;$
\item $ \Ker(\sigma) \cup \Ker(\tau)= [n-2]\setminus \{ i-1\};$
\item In particular, $ (\sigma,\tau) \in  \PS^{\sf N}(n-1,k,l)\setminus \PS^{\sf N}_\square(n-1;k,l);$
\item if $(s^{\nu''} d^j,s^{\mu''} d^j)=(\sigma,\tau)$ for some $(\mu'',\nu'')\in {\sf Sh}(k,l),$ and $0\leq j\leq n$ then $i=j$ and either $(\nu'',\mu'')=(\nu,\mu)$ or $( \nu'',\mu'')=(\nu',\mu').$
\end{itemize}
\end{lemma}
\begin{proof}  Assume that $ \nu_r=i.$ Then $\nu'_r=i-1,$  $\nu'_s=\nu_s$ for $s\ne r$ and 
\begin{equation}
    s^\nu = s^{\nu_{0}} \dots  s^{\nu_{r-1}} s^{i} s^{\nu_{r+1}} \dots s^{\nu_l}, \hspace{1cm} s^{\nu'}=s^{\nu_{0}} \dots  s^{\nu_{r-1}} s^{i-1} s^{\nu_{r+1}} \dots s^{\nu_l}
\end{equation}
Therefore  by the formula $s^{i-1}d^i={\sf id}=s^id^i$ and the formula $s^jd^i=d^is^{j-1}$ for $j>i$ we obtain
\begin{equation}\label{eq_smudi}
    s^\nu d^i = s^{\nu_{0}} \dots  s^{\nu_{r-1}} s^{\nu_{r+1} -1} \dots s^{\nu_l-1} = s^{\nu'}d^i.
\end{equation}
This formula also implies that $ s^\nu d^i$ is surjective. 
Similarly we prove that $s^\mu d^i = s^{\mu'}d^i$ and that $s^\mu d^i $ is surjective. 
It is easy to see that 
\begin{equation}
\begin{split}
&\Ker(s^\nu d^i)\cup \Ker(s^\mu d^i)= \\
&\{ \nu_0,\dots, \nu_{r-1}\} \cup  \{\nu_{r+1}-1,\dots,\nu_{l}-1 \} \cup \{ \mu_0,\dots, \mu_{t-1}\}\cup \{\mu_{t+1}-1,\dots,\mu_k-1\}=\\
&(\{ \nu_0,\dots, \nu_{r-1}\}  \cup \{ \mu_0,\dots, \mu_{t-1}\}) \cup  (\{\nu_{r+1}-1,\dots,\nu_{l}-1 \} \cup \{\mu_{t+1}-1,\dots,\mu_k-1\})=\\
&\{0,\dots, i-2\}\cup \{i,\dots,n-2\}.    
\end{split}
\end{equation}
Assume $(s^{ \nu''}d^j,s^{ \mu''}d^j)=(s^{\nu}d^i,s^{\mu}d^i).$ If $j=0$ or $j=n,$ then either $s^{\nu''}d^j$ is not a surjection or $s^{\mu''}d^j$ is not a surjection. So we can assume $1\leq j\leq n-1.$  Therefore, as we already proved, we have $\Ker(s^{ \nu''}d^j) \cup \Ker( s^{ \nu''}d^j )=[n-2]\setminus\{j-1\}$ and $\Ker(\sigma) \cup \Ker( \tau )=[n-2]\setminus\{i-1\}.$  Then $i=j.$ Since $s^{\nu''}d^i$ is a surjection and its image equals to $\{s^{\nu''}(0), \dots,s^{\nu''}(i-1),s^{\nu''}(i+1),\dots, s^{\nu''}(n)  \}$, we obtain either $s^{\nu''}(i-1)=s^{\nu''}(i)$ or $s^{\nu''}(i)=s^{\nu''}(i+1).$ It follows that there is $r''$ such that either $\nu''_{r''}=i-1,$ or $\nu''=i.$ Hence 
\begin{equation}
    s^\nu d^i=s^{\nu''}d^i = s^{\nu''_0} \dots s^{\nu''_{r''-1}}s^{\nu''_{r''+1}-1} \dots s^{\nu''_{n}-1}
\end{equation}
Combining this with equation \eqref{eq_smudi}, we obtain that either $\nu''=\nu$ or $\nu''=\nu'.$
\end{proof}

\begin{lemma}\label{lemma_PS_deg_square}
Let $(\sigma,\tau)\in \PS_\square^{\sf D}(n)$ and $0\leq i\leq n.$ Then either  $\Ker(\sigma d^i)\cap \Ker(\tau d^i)\ne \emptyset $ or $\Ker(\sigma d^i)\cup \Ker(\tau d^i)=[n-2].$ 
\end{lemma}
\begin{proof} 
Let 
\[\sigma=s^{\nu_0} s^{\nu_1}\dots s^{\nu_{n-k}}, \hspace{1cm}   
\tau=s^{\mu_0} s^{\mu_1} \dots s^{\mu_{n-l}},\]
where $0\leq \nu_0<\dots <\nu_{n-l}\leq n-1$ and $0\leq \mu_0<\dots <\mu_{n-k}\leq n-1.$ 
If $(\Ker(\sigma)\cap \Ker(\tau))\setminus \{i-1,i\} \ne \emptyset,$ then $\Ker(\sigma d^i)\cap \Ker(\tau d^i)\ne \emptyset.$ So we can assume   $\emptyset\ne \Ker(\sigma)\cap \Ker(\tau)\subseteq \{i-1,i\}.$ 
If $i=0$ we have $0\in \Ker(\sigma)\cap \Ker(\tau),$ and hence 
\[
\sigma=s^0 s^{\nu_1}\dots s^{\nu_l}, \hspace{1cm}   
\tau=s^0 s^{\mu_1} \dots s^{\mu_k}.
\]
Then 
\[\sigma d^0 = s^{\nu_1-1}\dots s^{\nu_l-1}, \hspace{1cm} \tau d^0=s^{\mu_1-1} \dots s^{\mu_k-1}.\]
Therefore $\Ker(\sigma d^0)\cup \Ker(\tau d^0)=[n-2].$
If $i=n,$ then $\nu_l=\mu_k=n-1 \in \Ker(\sigma)\cap \Ker(\tau)$ and we can prove this similarly.

Now we assume that $1\leq i\leq n-1.$
Since $\Ker(\sigma)\cup \Ker(\tau)=[n-1],$ we have either  $i-1,i\in \Ker(\sigma)$ or $i-1,i\in \Ker(\tau).$ Without loss of generality we can assume that 
$i-1,i\in \Ker(\sigma).$ Then we have 
\[ 
\sigma = s^{\nu_0} \dots s^{\nu_{r-1}} s^{i-1}s^i s^{\nu_{r+2}} \dots s^{\nu_l}
\]
and
\[\tau = s^{\mu_0} \dots s^{\mu_{t-1}} s^{\mu_t} s^{\mu_{t+1}}  \dots s^{\mu_k}, \]
where either $\mu_t=i-1$ or $\nu_t=i.$ If $i-1,i\in \Ker(\tau)$ we assume $\mu_t=i$.
Then 
\[ 
\sigma d^i = s^{\nu_0} \dots s^{\nu_{r-1}} s^{i-1} s^{\nu_{r+2}-1} \dots s^{\nu_l-1},
\]
\[\tau d^i=s^{\mu_0} \dots s^{\mu_{t-1}}  s^{\mu_{t+1}-1}  \dots s^{\mu_k-1}.\]
Therefore
\begin{equation}
\begin{split}
&\Ker(\sigma d^i)\cup \Ker(\tau d^i)= \\
&=\{\nu_0,\dots,\nu_{r-1}\} \cup \{\mu_0,\dots, \mu_{t-1} \} \cup \{i-1\} \cup \{ \nu_{r+2}-1,\dots \nu_{l}-1\} \cup \{\mu_{t+1}-1,\dots,\mu_k-1\}\\
&=\{0,\dots,i-2\}\cup \{i-1\} \cup \{i,\dots,n-2\}=\\
&=[n-2].
\end{split}    
\end{equation}
\end{proof}

\section{Simplicial modules}

We denote by $\mathbb{K}$ a commutative ring. All modules, algebras and tensor products are assumed to be over $\mathbb{K}.$

\subsection{Dold-Kan decomposition}

Here we recall some aspects of the theory of simplicial modules that can be found in \cite[Ch.8]{weibel1995introduction}, \cite[\S 22]{may1992simplicial}.

Recall that a simplicial module is a functor $A: \Delta^{\rm op}\to {\sf Mod},$ where $\Delta$ is the simplicial indexing category.  Equivalently it can be defined as a sequence of modules $A_0,A_1,\dots$ together with two collections of maps $d_i:A_n \to A_{n-1}$ and $s_i:A_n\to A_{n+1}$ for $0\leq i\leq n,$ called face maps and degeneracy maps, satisfying the simplicial identities. For an order preserving map $f:[m]\to [n]$ we set $f^*=A(f):A_n\to A_m.$

For a simplicial module $A$ one considers tree chain complexes $\CC A, \NN A, \DD A.$ The $n$th component of $\CC A$ is $A_n$ and the differential is defined by the formula $\partial_n=\sum_{i=0}^n (-1)^i d_i.$ The complex $\DD A$ is a subcomplex of $\CC A,$ whose $n$-th component is $\DD_n A=\sum_{i=0}^{n-1} s_i(A_{n-1}).$ Finally, the Moore complex  $\NN A$  is a complex whose components are
\begin{equation}
\NN_nA=\bigcap_{i\neq 0} {\sf Ker}(d_i:A_n\to A_{n-1})
\end{equation}
and the differential is induced by $d_0.$ Then $\NN A$ and $\DD A$ are subcomplexes of $\CC A$ and it is well known that $\CC A$ can be naturally decomposed as a direct sum of these subcomplexes
\begin{equation}
    \CC A=\NN A\oplus \DD A.
\end{equation}
Moreover, $\DD A$ is contractible and $\NN A$ is homotopy equivalent to $\CC A.$ The projection from $\CC A$ to $\NN A$ is denoted by 
\begin{equation}
    \rho : \CC A\epi \NN A.
\end{equation}
Then $\rho$ induces the isomorphism 
\begin{equation}
    \NN A \cong \CC A/\DD A.
\end{equation}
We will often identify $\NN A$ with $\CC A/\DD A.$

For any order preserving map $f:[n]\to [m]$ we set 
\begin{equation}
|f|:=m.    
\end{equation}
Then for any simplicial module $A$ any its component can be decomposed as
\begin{equation}\label{eq:DK-decomposition}
A_n = \bigoplus_{\sigma} \sigma^*(\NN_{|\sigma|}A),
\end{equation}
where the summation is taken by all surjective order-preserving maps $\sigma:[n]\epi [k],$ where $0\leq k\leq n$. This decomposition follows from the Dold-Kan correspondence and we call it the Dold-Kan decomposition. 
Since $\sigma^*:A_k\to A_n$ is injective, this decomposition implies that any element $a\in A_n$ can be uniquely presented as 
\begin{equation}
    a=\sum_{\sigma} \sigma^*(a_\sigma), \hspace{1cm} a_\sigma\in \NN_{|\sigma|} A. 
\end{equation}
Then the projection $\rho_n:A_n\epi  \NN_nA$  can be defined as 
\begin{equation}
    \rho_n(a)=a_{\sf id}.
\end{equation}

If $f:[m]\to [n]$ is an order preserving map, then the restriction of the map $f^*:A_n\to A_m$ on the summand 
$\sigma^*(\NN_{|\sigma|}A)$ is defined by the map to the summand
\begin{equation}\label{eq:f_on_summand}
\sigma^*(\NN_{|\sigma|}A) \longrightarrow \tau^*(\NN_{|\tau|} A), \hspace{1cm}  \sigma^*(a)\mapsto \tau^*( \alpha^*(a) ),
\end{equation}
where $\sigma f= \alpha \tau$ is the epi-mono decomposition of $\sigma f.$ Note that if $\alpha\notin\{{\sf id},d^0\},$ the element $\alpha^*(a)$ is trivial.

\subsection{Tensor product of simplicial modules}
For two simplicial modules $A$ and $A'$ their tensor product $A\otimes A'$ is defined dimension-wise $(A\otimes A')_n=A_n\otimes A_n'$ so that $(A\otimes A')(f)=A(f)\otimes A'(f).$ Then the Dold-Kan decomposition implies that 
\begin{equation}\label{eq:aotimesa'}
    (A\otimes A')_n =  \bigoplus_{(\sigma,\tau)\in \PS(n)} \sigma^*(\NN_{|\sigma|} A) \otimes \tau^*(\NN_{|\tau|} A'),   
\end{equation}
where the summation runs over couples of surjective order-preserving maps. 
Hence any element of $x\in (A\otimes A')_n$ can be uniquely presented as
\begin{equation}
x=\sum_{(\sigma,\tau)\in \PS(n)} (\sigma^*\otimes \tau^*)(x_{\sigma,\tau}), \hspace{1cm} x_{\sigma,\tau}\in \NN_{|\sigma|} A \otimes \NN_{|\tau|} A'.  
\end{equation}

\begin{lemma}
\label{lemma_deg_tensor}
For any simplicial modules $A$ and $A'$ we have
\begin{equation}
\DD_n(A\otimes A') =  \bigoplus_{(\sigma,\tau)\in \PS^\DD(n)} \sigma^*(\NN_{|\sigma|} A) \otimes \tau^*(\NN_{|\tau|} A').
\end{equation}
In other words, an element $x\in A_n\otimes A_n'$ is in $\DD_n(A\otimes A') $ if and only if $ x_{\sigma,\tau}=0$ for $(\sigma,\tau)\in \PS^\NN(n).$
\end{lemma}
\begin{proof}  Note that $\DD_n(A\otimes A')$ is the sum of submodules $s_i( \tilde \sigma^*(\NN_{|\tilde \sigma|}A) \otimes \tilde\tau^*(\NN_{|\tilde \tau|}A) )$ over all indexes $0\leq i\leq n-1,$ $(\tilde \sigma,\tilde \tau )\in \PS(n-1).$ For any $\sigma:[n]\epi [k]$ we have $i\in \Ker(\sigma:[n]\epi [k])$ iff $\sigma=\tilde \sigma s^i$ for some $\tilde \sigma\in [n-1]\to [k].$  Hence, $ \Ker(\sigma) \cap \Ker(\tau)\neq \emptyset$ iff $\sigma = \tilde \sigma s^i$ and $\tau=\tilde\tau s^i$ for some $i$ and some $\tilde \sigma:[n-1]\to [k]$ and $\tilde \tau:[n-1]\to [l],$ and in this case we have $\sigma^*(\NN_kA) \otimes \tau^*(\NN_lA') = s_i(\tilde \sigma^*(\NN_kA) \otimes \tilde \tau^*(\NN_lA')).$ The  equation follows. 
\end{proof}

\begin{corollary}\label{cor_N_tensor} For any simplicial modules $A$ and $A'$ there is an isomorphism
\begin{equation}\label{eq_N_iso_sum}
 \NN_n(A\otimes A') \cong     \bigoplus_{(\sigma,\tau)\in \PS^\NN(n)} \NN_{|\sigma|} A \otimes \NN_{|\tau|} A', \hspace{1cm} x\mapsto (x_{\sigma,\tau})_{(\sigma,\tau)\in \PS^\NN(n)}.
\end{equation}

\end{corollary}

\begin{remark}
Note that we do not claim that $\NN_n(A\otimes A')$ equals to the sum of the modules $ \sigma^*(\NN_{|\sigma|} A)\otimes \tau_*(\NN_{|\tau|} A)$ for $(\sigma,\tau)\in \PS^{\sf N}(n)$ as submodule of $A_n\otimes A_n'.$ There is only an isomorphism, not equation. Generally for $x\in \NN_n(A\otimes A')$ the coordinate $x_{\sigma,\tau}$ can be nontrivial for a degenerated pair of surjections $(\sigma,\tau)\in \PS^\DD(n).$
\end{remark}
\begin{corollary}\label{cor_rho}
Let $x,y\in A_n\otimes A'_n.$ Then $\rho(x)=\rho(y)$ if and only if $x_{\sigma,\tau}=y_{\sigma,\tau}$ for all non-degenerated pairs of surjections $(\sigma,\tau)\in \PS^\NN(n)$.  In particular, $\rho(x)_{\sigma,\tau}=x_{\sigma,\tau}$ for $\sigma,\tau\in \PS^\NN(n).$
\end{corollary}

\subsection{Eilenberg-Zilber and Alexander-Whitney maps}
Here we remind some information about the Eilenberg-Zilber theorem that can be found in \cite[\S 8]{weibel1995introduction}, \cite[\S 29]{may1992simplicial}.

For two simplicial modules $A,A'$
the Eilenberg-Zilber map is a morphism of chain complexes 
\begin{equation}
    \mathcal{E} :\CC A\otimes \CC A' \longrightarrow \CC (A\otimes A')
\end{equation}
given by 
\begin{equation}
  \mathcal{E}(a\otimes a') = \sum_{(\mu,\nu)\in {\sf Sh}(k,l)} {\sf sgn}(\mu,\nu)  s_\nu a \otimes s_\mu a' 
\end{equation}
for $a\in A_k$ and $a'\in A_l'.$ 
The Alexander-Whitney map is the morphism of complexes 
\begin{equation}
\mathcal{A} : \CC (A\otimes A') \longrightarrow \CC A\otimes \CC A' \end{equation}
defined by 
\begin{equation}
\mathcal{A}( a\otimes a' ) = \sum_{k+l=n} h^l a \otimes t^k a'.
\end{equation}
The Eilenberg-Zilber theorem says that they satisfy $\mathcal{E}\mathcal{A}\sim {\sf id}$ and $\mathcal{A}\mathcal{E}\sim {\sf id},$ and hence, $\mathcal{A},\mathcal{E}$ are homotopy equivalences.

It is well-known that these maps send degenerated elements to degenerated elements in the following sense $\mathcal{E}(\DD A\otimes \CC A' + \CC A \otimes \DD A') \subseteq \DD (A\otimes A')$ and $\mathcal{A}( \DD (A\otimes A') ) \subseteq \DD A \otimes \CC A' + \CC A \otimes \DD A';$ and induce maps
\begin{equation} \label{eq_eilenberg-zilber}
\varepsilon : \NN A\otimes \NN A' \longrightarrow \NN (A\otimes A'), \hspace{1cm} \alpha :   \NN (A\otimes A') \longrightarrow  \NN A\otimes \NN A'
\end{equation} 
defined by the formulas
\begin{equation}
    \varepsilon(x) = \rho \mathcal{E}(x), \hspace{1cm} \alpha(x)=(\rho\otimes \rho)\mathcal{A}(x),
\end{equation}
such that the diagram  
\begin{equation}\label{diag:zeta}
\begin{tikzcd}
\CC A \otimes \CC A' \ar[r,"\mathcal{E}"] \ar[d,twoheadrightarrow,"\rho\otimes \rho"] & \CC (A\otimes A') \ar[d,twoheadrightarrow,"\rho"] \ar[r,"\mathcal{A}"] & \CC A\otimes \CC A'\ar[d,twoheadrightarrow,"\rho\otimes \rho"] \\ 
 \NN A \otimes \NN A' \ar[r,"\varepsilon"]
\arrow[bend right=5mm,"{\sf id}"']{rr} 
&  \NN (A\otimes A') \ar[r,"\alpha"] & \NN A\otimes \NN A'
\end{tikzcd}
\end{equation}
is commutative. In particular, $\alpha \varepsilon={\sf id}.$ 

\begin{lemma}\label{lemma:Imepsilon}
Let $x\in {\sf N}_n(A\otimes A').$  Then $x\in \Im(\varepsilon : \NN A\otimes \NN A' \longrightarrow \NN (A\otimes A'))$ if and only if the following conditions are satisfied
\begin{enumerate}
    \item for any $0\leq k,l\leq n$ such that $k+l=n$ and any two shuffles $(\mu,\nu),(\mu',\nu')\in {\sf Sh}(k,l)$ there is an equation
\[
{\sf sgn}(\mu,\nu)x_{s^\nu,s^\mu}= {\sf sgn}(\mu',\nu')x_{s^{\nu'},s^{\mu'}};
\]
\item $x_{\sigma,\tau}=0$ for any $(\sigma,\tau)\in \PS^{\sf N}(n)\setminus \PS^{\sf N}_\square(n).$
\end{enumerate}
\end{lemma}
\begin{proof}
If $a\in \NN_kA$ and $a'\in \NN_lA'$ we have ${\sf sgn}(\mu,\nu) s_\nu(a)\otimes s_\mu(a') \in s_\nu(\NN_kA)\otimes s_\mu(\NN_lA'),$ and hence $ \varepsilon(a\otimes a')_{s^\nu,s^\mu}= \rho \mathcal{E}(a\otimes a')_{s^\nu,s^\mu} ={\sf sgn}(\mu,\nu) a\otimes a'$ (see Corollary \ref{cor_rho}). This equation shows that for any $y\in \NN A \otimes \NN A'$ we have
\[
\varepsilon(y)_{s^\nu,s^\mu}={\sf sgn}(\mu,\nu) y.
\]
This follows that the properties (1),(2) are satisfied for elements from $\Im(\varepsilon).$

Assume (1) and (2) are satisfied. For any fixed $k,l$ such that $k+l=n$ we consider the $(k,l)$-shuffle $(\mu_0,\nu_0)=((0,\dots,k-1),(k,\dots,n-1))$ and set \[y_{k,l}=x_{s^{\nu_0},s^{\mu_0}}\in \NN_kA\otimes \NN_lA'.\]
Then 
\[
\varepsilon(y_{k,l})_{s^\nu,s^\mu}= {\sf sgn}(\mu,\nu) y_{k,l}=x_{s^{\nu},s^\mu}
.\] 
If we take $y=\sum_{k,l} y_{k,l},$ we obtain that $\varepsilon(y)_{\sigma,\tau}=x_{\sigma,\tau}$ for all $(\sigma,\tau)\in \PS^{\sf N}(n)$.
This implies that $\varepsilon(y)=x.$
\end{proof}

\section{Path pairs of modules}

\subsection{Definition} In this section we denote by $\mathbb{K}$ a commutative ring. A {\it path pair (of modules)} is a pair $(A,B),$ where $A$ is a simplicial module and $B$ its path submodule. In other words $B$ is a graded submodule of a simplicial module $A$ closed with respect to degenerasy maps $s_i(B_n)\subseteq B_{n+1}$ and exterior face maps $d_0(B_n)\subseteq B_{n-1}$ and $d_n(B_n)\subseteq B_{n-1}.$  A morphism of path pairs  $f:(A,B)\to (A',B')$ is a morphism of simplicial modules $f:A\to A'$ such that $f(B)\subseteq B'.$

Generally on can define a path pair of objects in a category as a simplicial object together with its path ``subobject'' with an appropriate definition of a subobject. But in this section by a path pair we will always mean a path pair of modules. 

\subsection{Homology of a path pair} 

For a path pair $\PPP=(A,B)$ we denote by 
\begin{equation}
 \overline{B}_n := \rho(B_n) \subseteq \NN_nA    
\end{equation}
the image of $B_n$ in $\NN A$ with respect to the projection $\rho: \CC A\epi \NN A.$ Then $\overline{B}$ is a graded submodule of $\NN A$ which is not necessarily a subcomplex. We also set 
\begin{equation}
\NN \PPP = (\NN A, \overline{B}).
\end{equation}
and 
\begin{equation}
    \Omega  \PPP=\omega(\NN \PPP), \hspace{1cm} \Psi\PPP = \psi(\NN \PPP ).
\end{equation}

The homology of these complexes are called the path homology and copath homology of the path pair
\begin{equation}
    H_n\PPP=H_n(\Omega  \PPP), \hspace{1cm} H^{\sf c}_n \PPP=H_n(\Psi \PPP).
\end{equation}

Corollary \ref{cor:long} implies that there is a long exact sequence 
\begin{equation}
{\dots} \to H_n\PPP \to H_n(\NN A) \to H_n^{\sf c} \PPP \to H_{n-1}\PPP \to {\dots}.
\end{equation}

\subsection{Box product of path pairs}
Let $(A,B)$ and $(A',B')$ are pairs of modules. Motivated by Proposition \ref{prop:paths_in_product} their box product is defined as
\begin{equation}\label{eq:def_box-product}
    (A,B)\square (A',B') = (A\otimes A',B\diamond B'),
\end{equation}
where
\begin{equation}
(B\diamond B')_n= \sum_{k+l=n}\:  \sum_{(\mu,\nu)\in {\sf Sh}(k,l)} s_\nu( B_k) \:\bar \otimes\: s_\mu(B'_l).
\end{equation}
Note that this formula is similar to the formula of the Eilenberg-Zilber map. In order to prove that $B\diamond B'$ is a path submodule of $A\otimes A',$ we need a lemma.

\begin{lemma}\label{lemma_box_eq}
The graded module $B\diamond B'$ can be defined as
\begin{equation}
(B\diamond B')_n= \sum_{(f,g)\in \PPi_\square(n)} f^*( B_{|f|})\:\bar\otimes\: g^*(B'_{|g|}).
\end{equation}
and as 
\begin{equation}
(B\diamond B')_n= \sum_{(\sigma,\tau)\in \PS_\square(n)} \sigma^*( B_{|\sigma|})\:\bar\otimes\: \tau^*(B'_{|\tau|}).
\end{equation}
\end{lemma}
\begin{proof}
Take $(f,g)\in \PPi_\square(n;k,l)$ and consider its standard decomposition $(f,g)=(\alpha \sigma s^\nu \sigma, \beta s^\mu),$ where $(\mu,\nu)\in {\sf Sh}(k',l')$ (see \eqref{eq:standard_decomposition}). Then $f^*(B_{|f|}) \: \bar \otimes \: g^*(B'_{|g|}) \subseteq s_\nu( B_{k'} )\: \bar \otimes \: s_\mu(B_{l'}).$
\end{proof}

\begin{corollary}
$B\diamond B'$ is a path submodule of $A\otimes A'.$
\end{corollary}
\begin{proof} It follows from Lemma \ref{lemma_box_eq} and the fact that $\PPi_\square$ is natural by $[n],[k],[l]\in \Pi.$ 
\end{proof}

\begin{lemma} The following inclusion holds.
\label{lemma_box_sub}
\begin{equation}
(B\diamond B')_n \subseteq \bigoplus_{(\sigma,\tau)\in \PS_\square(n)} \sigma^*(\NN_{|\sigma|} A)  \otimes \tau^*(\NN_{|\tau|} A').    
\end{equation}
In other words, for any $x\in (B\diamond B')_n$ and any $(\sigma,\tau)\in {\sf PS}(n)\setminus {\sf PS}_\square(n)$
 we have $ x_{\sigma,\tau}=0.$
\end{lemma}
\begin{proof} Note that $B_k\subseteq A_k=\bigoplus_{\psi: [k]\epi [k']} \psi^*(\NN_{k'}A).$  Then for any $(\sigma,\tau)\in\PS_\square(n;k,l)$ we have that $  \sigma^*(B_k)\otimes \tau^*(B'_l )$ is included into the sum of  modules 
$(\psi\sigma)^*(\NN_kA)\otimes (\phi\tau)^*(\NN_lA')$ over all pairs of 
surjections $\psi:[k]\epi [k']$ and $\phi:[l]\epi [l'].$ 
Since $(\sigma,\tau)\in \PS_\square(n;k,l),$ we have $ (\psi\sigma,\phi\tau)\in \PS_\square(n;k',l').$ 
Then the assertion follows from Lemma \ref{lemma_box_eq}.
\end{proof}

For two complexes with graded submodules we set 
\[(C,D)\otimes (C',D')=(C\otimes C',D\bar \otimes D').\]

\begin{lemma} \label{lemma:EZ-for-pairs}
Let $\PPP=(A,B)$ and $\PPP'=(A',B')$ be two path pairs. Then the Eilenberg-Zilber and Alexander-Whitney maps induce  morphisms of complexes with graded submodules 
\begin{equation}
\varepsilon: \NN \PPP \otimes \NN \PPP' 
\longrightarrow \NN(\PPP\square \PPP'),
\end{equation}
\begin{equation}
\alpha : \NN(\PPP\square \PPP') \longrightarrow  \NN \PPP \otimes \NN \PPP'.
\end{equation}
\end{lemma}
\begin{proof}
Since, for any shuffle $(\mu,\nu)$ we have $(s^\mu,s^\nu)\in \PPi_\square(n),$ we obtain $ \mathcal{E}( B \bar \otimes B' )\subseteq B\diamond B'.$ 
Using that the diagram \eqref{diag:zeta} is commutative, we get 
$
\varepsilon (\overline{B}\bar \otimes  \overline{B}')  \subseteq \overline{B\diamond B'}.
$ Then the morphism $\varepsilon: \NN \PPP \otimes \NN \PPP' 
\to \NN(\PPP\square \PPP')$ is well defined. Since $B$ and $B'$ are closed with respect to exterior faces, we obtain $\mathcal{A}( B\diamond B' )\subseteq B\bar \otimes B'.$ Using that the diagram \eqref{diag:zeta} is commutative, we get $\alpha( \overline{B\diamond B'} )\subseteq \overline{B} \bar \otimes \overline{B}'.$ Then the map $\alpha: \NN (\PPP \square \PPP') \to \NN \PPP \otimes \NN \PPP'$ is well defined.
\end{proof}

\subsection{Homotopy invariance}
\label{sec:homotopy}

We denote by $\Delta^n$ the standard $n$-simplex and by $d^0,d^1:\Delta^0\to \Delta^1$ the two embeddings of $0$-simplex to the $1$-simplex. Consider the path pair of modules given by 
\begin{equation}
    I^{\sf p}=(\mathbb{K}[\Delta^1], \mathbb{K}[\Delta^1]), \hspace{1cm} {\sf pt}^{\sf p}=(\mathbb{K},\mathbb{K})
\end{equation}
and two morphisms between them induces by $d^0,d^1$
\begin{equation}
    i_0,i_1:{\sf pt}^{\sf p}\longrightarrow I^{\sf p}.
\end{equation} 
Note that ${\sf pt}^{\sf p} \square \PPP\cong \PPP.$ Then we obtain a weak cylinder functor 
\begin{equation}
    {\sf cyl}(\PPP)=\PPP \square I^{\sf p}
\end{equation}
and define homotopic morphisms of path pairs via this weak cylinder functor. 

\begin{theorem}\label{th:homotopy}
Any homotopic morphisms of path pairs  $f\sim g:\PPP\to \PPP'$ induce homotopic maps 
\begin{equation}
\Omega  f  \sim \Omega  g : \Omega  \PPP \longrightarrow \Omega  \PPP', \hspace{1cm} 
\Psi f \sim \Psi g : \Psi \PPP \longrightarrow \Psi \PPP'.
\end{equation}
\end{theorem}
\begin{proof}By \eqref{eq:cyl_I} we have
$
{\sf cyl}( \NN \PPP )\cong (\NN A \otimes I^{\sf c}, \overline{B} \otimes I^{\sf g} ).
$ On the other hand 
$I^{\sf c}=\NN (\KK[\Delta^1]).$ Therefore, $\NN A \otimes I^{\sf c}= \NN A \otimes \NN(\KK[\Delta^1])$ and ${\sf cyl}(\NN \PPP) = \NN \PPP \otimes \NN I^{\sf p}.$ 
Then the we have the Eilenberg-Zilber map (Lemma \ref{lemma:EZ-for-pairs})
\begin{equation}
\varepsilon: {\sf Cyl}(\NN \PPP) \to \NN( {\sf cyl}( \PPP)).    
\end{equation}
We claim that the triangle 
\begin{equation}
\begin{tikzcd} 
& \NN \PPP \ar[dl,"i^n_{\NN \PPP}"'] \ar[dr," \NN(i^n_\PPP)"] & 
\\
{\sf cyl}(\NN \PPP) \ar[rr,"\varepsilon"]
 & &  \NN( {\sf cyl}(\PPP)) 
\end{tikzcd}    
\end{equation}
is commutative for $n=0,1$. Indeed, for any $a\in (\NN A)_m$ we have $i_n(a)=(-1)^n a\otimes d^n,$ where $d^0,d^1$ are corresponding elements from $(\Delta^1)_0.$ There is only one $(m,0)$-shuffle, and hence, $\varepsilon(i^n_{\NN \PPP}(a))= (-1)^n a \otimes d^n = \NN(i^n_\PPP)(a).$ So the triangle is commutative. The assertion follows from Proposition \ref{prop:cyl}.
\end{proof}

\subsection{Eilenberg-Zilber theorem for \texorpdfstring{$\Omega$}{Ω}}

\begin{theorem}[{cf. \cite[Th.7.6]{grigor2012homologies}}]\label{th:EZ} Let $\mathbb{K}$ be a principal ideal domain and let $\PPP=(A,B)$ and $\PPP'=(A',B')$ be two path pairs of modules such that:
\begin{itemize}
    \item $A_n$ and $A_n'$ are free modules for any $n\geq 0;$
    \item $\overline{B}_n$ and $\overline{B}'_n$ are direct summands of $\NN_nB$ and $\NN_nB'$ respectively.
\end{itemize}
Then the Eilenberg-Zilber and Alexander-Whitney maps   \eqref{eq_eilenberg-zilber} induce mutually inverse isomorphism of complexes
\begin{equation}\label{eq:EZ}
\Omega \PPP\otimes \Omega \PPP'\cong \Omega (\PPP\square \PPP').
\end{equation}
Moreover, there is a short exact sequence 
\begin{equation}\label{eq:Kunn}
0 \to
\bigoplus_{i+j=n} H_i\PPP \otimes H_j\PPP' \to    H_n( \PPP \square \PPP' ) \to \bigoplus_{i+j=n-1} {\sf Tor}^\KK_1( H_i\PPP,H_j\PPP') \to 0.
\end{equation}
\end{theorem}
\begin{remark}
Note that in the ordinary Eilenberg-Zilber theorem for simplicial modules, the complexes $\NN A\otimes \NN A'$ and $\NN (A\otimes A')$ are generally not isomorphic, the first one is just a homotopy retract in the second one. But in Theorem \ref{th:EZ}, following \cite[Th.7.6]{grigor2012homologies}, we obtain a stronger result, an isomorphism of complexes.
\end{remark}

Since $\mathbb{K}$ is a principal ideal domain and the modules $A_n,A'_n$ are free, then their submodules $\NN_nA,\NN_nA', \Omega_n\PPP,\Omega_n\PPP'$ are also free, and the map $\Omega \PPP\otimes \Omega \PPP'\to \NN A\otimes \NN A'$ is injective. So we can identify $\Omega \PPP\otimes \Omega \PPP'$ with a submodule of $\NN A\otimes \NN A'.$ By the same reason we identify $B_k\otimes B'_l$ with the submodule of $A_k\otimes A'_l$ and identify $\bar B_k\otimes \bar B'_l$ with the submodule of $\NN_kA\otimes \NN_lA'.$ In order to prove Theorem \ref{th:EZ} we need two lemmas.

\begin{lemma}\label{lemma_restriction}
Under the conditions of Theorem \ref{th:EZ} the Eilenberg-Zilber and Alexander-Whitney maps   can be restricted to the maps of subcomplexes \begin{equation} \varepsilon':
\Omega \PPP \otimes \Omega \PPP'  \rightleftarrows   \Omega(\PPP\square \PPP' ): \alpha'    
\end{equation} 
such that $ \alpha'\varepsilon'={\sf id}.$
More precisely, there are inclusions 
\begin{equation}
     \varepsilon  (\Omega \PPP\otimes  \Omega\PPP')  \subseteq \Omega(\PPP \square \PPP'), \hspace{1cm}
\alpha(\Omega(\PPP\square \PPP')) \subseteq \Omega\PPP\otimes \Omega\PPP'.
\end{equation}
\end{lemma}
\begin{proof}
Lemma \ref{lemma:EZ-for-pairs} implies that these maps can be restricted to the maps 
\begin{equation}
\omega(\NN A \otimes \NN A', \overline{B}\otimes \overline{B}' ) \leftrightarrows \Omega (\PPP \square \PPP').    
\end{equation}
The assertion follows from Proposition \ref{prop_omega_direct}. 
\end{proof}

\begin{lemma}[{cf. \cite[Prop.7.12]{grigor2012homologies}}]
\label{lemma_omega_in_image} Under the assumptions of Theorem \ref{th:EZ} there is an inclusion
\begin{equation}
  \Omega (\PPP\square \PPP' ) \subseteq \Im(\varepsilon : \NN A\otimes \NN A' \longrightarrow \NN (A\otimes A')).  
\end{equation}
\end{lemma}
\begin{proof} 
By Lemma \ref{lemma:Imepsilon} we need to prove that for any $x\in \Omega (\PPP\square \PPP')$ and any $0\leq k,l\leq n$ such that $k+l=n$ we have 
\begin{equation}\label{eq:sngmunu}
    {\sf sgn}(\mu,\nu)x_{s^\nu,s^\mu} = {\sf sgn}(\mu',\nu') x_{s^{\nu'},s^{\mu'}}
\end{equation}
and $x_{\sigma,\tau}=0$ for any   $(\sigma,\tau)\in \PS^{\sf N}(n)\setminus \PS^{\sf N}_\square(n).$

Consider a preimage $\tilde x\in B\diamond B'$ of $x$ i.e. an element such that $\rho(\tilde x)=x.$ Then by Lemma \ref{lemma_box_sub} we have $\tilde x_{\sigma,\tau}=0$ for $(\sigma,\tau)\in \PS(n)\setminus \PS_\square(n).$ Hence, 
\begin{equation}
    \tilde x= \sum_{(\sigma,\tau)\in \PS_\square(n)} (\sigma^* \otimes \tau^*) (\tilde x_{\sigma,\tau}), \hspace{1cm} x_{\sigma,\tau} \in \NN_{|\sigma|} A \otimes \NN_{|\tau|} A' .
\end{equation} 
(The element $\tilde x$ is ``better'' than $x$ because $\tilde x_{\sigma,\tau}=0$ for $(\sigma,\tau)\in \PS^{\sf D}(n) \setminus \PS_\square(n)$, while $x_{\sigma,\tau}$ can be nontrivial).
Note $\rho(\tilde x)=x$ implies  
\begin{equation}\label{eq:x=tildex}
    \tilde x_{\sigma,\tau} = x_{\sigma,\tau}, \hspace{1cm} \text{if} \  (\sigma,\tau)\in \PS^{\sf N}(n).
\end{equation}
In particular, we obtain  that $x_{\sigma,\tau}=0$ if $(\sigma,\tau)\in \PS^{\sf N}(n)\setminus \PS^{\sf N}_\square(n).$  So we only need to prove \eqref{eq:sngmunu}.

Fix some $k,l$ such that $k+l=n.$ 
Since the graph ${\bf Sh}(k,l)$ is connected (Lemma \ref{lemma_graph_sh}), it is enough to check  \eqref{eq:sngmunu} for two shuffles with an edge $(\mu,\nu)\to (\mu',\nu')$ in ${\bf Sh}(k,l).$ Assume that the weight of the edge $(\mu,\nu)\to (\mu',\nu')$ is $1\leq i\leq n-1.$ By Lemma \ref{lemma_mu_nu_edge} we have that  $(s^\nu d^i,s^\mu d^i)=(s^{\nu'} d^i,s^{\mu'} d^i)\in \PS^{\sf N}(n)\setminus \PS^{\sf N}_\square(n).$ Set  $\sigma_0 =s^{\nu} d^i$ and $\tau_0=s^\mu d^i.$

Since $x\in \Omega (\PPP\square \PPP')$ we see that $\rho(\partial^C(\tilde x))\in \overline{B\diamond B'}.$ 
Combining the fact that $(\sigma_0,\tau_0)\in {\sf PS}^{\sf N}(n) \setminus {\sf PS}^{\sf N}_\square(n),$ Lemma \ref{lemma_box_sub}, and Corollary \ref{cor_rho}
we obtain
\[\partial^C(\tilde x)_{\sigma_0,\tau_0}=\rho(\partial^C(\tilde x))_{\sigma_0,\tau_0}=0.\]
On the other hand 
\begin{equation}\label{eq:d^Ctildex}
\partial^C(\tilde x) = \sum_{j=0}^n \sum_{(\sigma,\tau)\in \PS_\square(n)}  (-1)^j d_j( (\sigma^* \otimes \tau^*) (\tilde x_{\sigma,\tau})).    
\end{equation}

We claim that only non-trivial summands of the sum \eqref{eq:d^Ctildex} that can lie in $\sigma_0^*(\NN_kA)\otimes \tau_0^*(\NN_lA')$ are $(-1)^id_i( (s_\nu \otimes s_\mu) (\tilde x_{s^\nu,s^\mu}))$ and $(-1)^i d_i((s_\nu \otimes s_\mu)(\tilde x_{s^\nu,s^\mu})).$ 
Let us prove it. Take $(\sigma,\tau)\in {\sf PS}_\square(n).$
By \eqref{eq:f_on_summand} the summand $d_j( (\sigma^*\otimes \tau^*) (\tilde x_{\sigma,\tau}))$ is in $\varphi^*(\NN_{|\varphi|}A)\otimes \psi^*(\NN_{|\psi|}A'),$ where $\sigma d^j=\alpha\varphi$ and $\tau d^j = \beta \psi$ are epi-mono decompositions of $\sigma d^j,\tau d^j.$ Assume that the summand $(-1)^jd_j( (\sigma \otimes \tau) (\tilde x_{\sigma,\tau}))$ is non-trivial an lies in $\sigma_0^*(\NN_kA)\otimes \tau_0^*(\NN_kA').$ Then $(\varphi,\psi)=(\sigma_0,\tau_0)$ and $|\varphi|=k,$  $|\psi|=l.$
We have two cases: $(\sigma,\tau)\in \PS_\square^{\sf D}(n)$ and $(\sigma,\tau)\in \PS_\square^{\sf N}(n).$ Consider them separately.  

First assume $(\sigma,\tau)\in \PS_\square^{\sf D}(n).$ By Lemma \ref{lemma_PS_deg_square}, using that $\Ker(\sigma d^j)=\Ker(\varphi)=\Ker(\sigma_0)$ and $\Ker(\tau d^j)=\Ker(\psi)=\Ker(\tau_0),$ we obtain either $\Ker(\sigma_0) \cap \Ker(\tau_0)\ne \emptyset$ or $\Ker(\sigma_0)\cup \Ker(\tau_0)=[n-2].$ However, this contradicts to Lemma \ref{lemma_mu_nu_edge}, because $\Ker(\sigma_0) \cap \Ker(\tau_0)= \emptyset$ and $\Ker(\sigma_0)\cup \Ker(\tau_0)=[n-2]\setminus \{i-1\}.$

Now assume that $(\sigma,\tau)\in {\sf PS}_\square^{\sf N}(n).$ Since $|\varphi|=k$ and $|\psi|=l,$ we have $\alpha:[k]\mono [|\sigma|]$ and $\beta:[l]\mono [|\tau|].$ 
Then $k\leq |\sigma|, l\leq |\tau|$ and $k+l=n.$ 
It follows $|\sigma|=k$ and $|\tau|=l.$ 
Therefore $\alpha={\sf id},$  $\beta={\sf id}$ and  $(\sigma,\tau)=(s^\lambda,s^\kappa)$ for some shuffle $(\lambda,\kappa)\in {\sf Sh}(k,l)$ (Lemma \ref{lemma_shuffles}).
By Lemma \ref{lemma_mu_nu_edge} we obtain that the equation $(\sigma d^j,\tau d^j)=(\sigma_0,\tau_0)$ implies that $i=j$ and either $(\sigma,\tau)=(s^\nu,s^\mu)$
or $(\sigma,\tau)=(s^{\nu'},s^{\mu'}).$ So, the only nontrivial summands lying in $\sigma_0^*(\NN_k A)\otimes \tau_0^*(\NN_l A')$ are $(-1)^id_i((s_\nu \otimes s_\mu)(\tilde x_{s^\nu,s^\mu}))$ and $ (-1)^id_i( (s_\nu \otimes s_\mu) (\tilde x_{s^{\nu'},s^{\mu'}}))$. 

Therefore by \eqref{eq:f_on_summand} we have
 \[0=\partial^C(\tilde x)_{\sigma_0,\tau_0}=(-1)^i(\tilde x_{s^\nu,s^{\mu}}+ \tilde x_{s^{\nu'},s^{\mu'}}).\]
 By \eqref{eq:x=tildex} we have $\tilde x_{s^\nu,s^\mu}=x_{s^\nu,s^\mu}$ and $\tilde x_{s^{\nu'},s^{\mu'}}=x_{s^{\nu'},s^{\mu'}}.$
 The assertion follows. 
\end{proof}

\begin{proof}[{Proof of Theorem \ref{th:EZ}}] Since $\Omega_n \PPP$ and $\Omega_n \PPP'$ are direct summands of $A_n$ and $A'_n,$ we obtain that they are also free modules. Then the short exact sequence \eqref{eq:Kunn} follows from the isomorphism \eqref{eq:EZ} and the K\"{u}nneth theorem for chain complexes. So we only need to prove the isomorphism \eqref{eq:EZ}.  Lemma \ref{lemma_restriction} implies that the restrictions  $\varepsilon'$ and $\alpha'$ are well defined and $\alpha'\varepsilon'={\sf id}.$ Lemma \ref{lemma_omega_in_image} implies that any element $x\in \Omega (\PPP\square\PPP')$ can be presented as $x=\varepsilon(y)$ for some $y\in \NN A\otimes \NN A'.$ Then $\varepsilon\alpha(x)=\varepsilon \alpha \varepsilon(y)= \varepsilon (y)=x$ because $ \alpha\varepsilon={\sf id}.$ Hence $ \varepsilon'\alpha'={\sf id}.$
\end{proof}

\begin{corollary}\label{cor:EZ}
If $\mathbb{K}$ is a field, for any path pairs of vector spaces $\PPP$ and $\PPP'$ there is an isomorphism 
\begin{equation}
\Omega \PPP\otimes \Omega  \PPP'\cong \Omega (\PPP\square \PPP').
\end{equation} 
\end{corollary}

\section{Path pairs of sets and path complexes}\label{sec:path_pairs_of_sets} 
\subsection{Path pairs of sets}

A path pair of sets is a pair $\SSS=(X,Y),$ where $X$ is a simplicial set and $Y$ is its path subset. The associated path pair of modules is given by $\KK[\SSS]=(\KK[X],\KK[Y]),$ where $\KK[-]:{\sf Set}\to {\sf Mod}$ is the functor of free module applied level-wise. The complexes $\Omega(\SSS,\KK)$ and $\Psi(\SSS,\KK)$ are defined as 
\begin{equation}
 \Omega(\mathpzc{S},\mathbb{K}) = \Omega(\mathbb{K}[\mathpzc{S}]), \hspace{1cm} \Psi(\mathpzc{S},\mathbb{K}) = \Psi(\mathbb{K}[\mathpzc{S}] ) 
\end{equation}
If $\mathbb{K}$ is fixed, we will omit it in the notation $\Omega\mathpzc{S}=\Omega(\mathpzc{S},\mathbb{K})$ and $\Psi \SSS=\Psi(\SSS,\KK).$ The path homology and copath homology of a path pair of sets are defined by $H_*\SSS=H_*(\Omega \SSS)$ and $H_*^{\sf c} \SSS.$ Note that there is a long exact sequence
\begin{equation}\label{eq:long_set}
{\dots} \to H_n \SSS \to H_n X \to H_n^{\sf c} \SSS \to H_{n-1} \SSS \to {\dots}.  
\end{equation}

As in the case of path pairs of modules, motivated by Proposition \ref{prop:paths_in_product}, we define the box product of two path pairs of sets  $\SSS=(X,Y)$ and $\SSS'=(X',Y')$ as
\begin{equation}
    \SSS \square \SSS' = (X\times X',Y\diamond Y'),
\end{equation}
where 
\begin{equation}
  (Y\diamond Y')_n = \bigcup_{k+l=n} \bigcup_{(\mu,\nu)\in {\sf Sh}(k,l)} s_\nu(Y_{k})\times s_\mu(Y_{l}').
\end{equation}

\begin{lemma} \label{lemma:diamond_for_sets}
The subset $(Y\diamond Y')_n$ can be defined as
\begin{equation}
(Y\diamond Y')_n= \bigcup_{(f,g)\in \PPi_\square(n)} f^*( Y_{|f|})\times g^*(Y'_{|g|}).
\end{equation}
and as 
\begin{equation}
(B\diamond B')_n= \bigcup_{(\sigma,\tau)\in \PS_\square(n)} \sigma^*( B_{|\sigma|})\times \tau^*(B'_{|\tau|}).
\end{equation}
\end{lemma}
\begin{proof}
The proof is similar to the proof of Lemma \ref{lemma_box_eq}. 
\end{proof}

\begin{remark}
A more conceptual and categorical definition of the box product can be given via Day convolution (see Section \ref{Appendix}).
\end{remark}

It is easy to see that 
\begin{equation}\label{eq:square}
    \KK[\SSS\square \SSS' ] \cong \KK[\SSS] \square \KK[\SSS'].
\end{equation}

Consider the path pair of sets given by 
\begin{equation}
    I^{\sf s}=(\Delta^1,\Delta^1), \hspace{1cm} {\sf pt}^{\sf s}=(\Delta^0,\Delta^0)
\end{equation}
and two morphisms between 
$i_0,i_1:{\sf pt}^{\sf s} \to I^{\sf s}$ them induced by $d^0,d^1.$ 
Since $\SSS \square {\sf pt}^{\sf s}\cong \SSS,$ we obtain that 
\begin{equation}
{\sf cyl}(\SSS) = \SSS \square I^{\sf s}
\end{equation}
is a weak cylinder functor, and we define homotopic morphisms of path sets via this weak cylinder functor.

\begin{proposition}\label{prop:homotopy_set} Any homotopic morphisms of path sets $f\sim g:\SSS\to \SSS'$ induce homotopic maps on $\Omega$ and $\Psi$
\begin{equation}
\Omega f \sim \Omega g: \Omega \SSS \longrightarrow \Omega \SSS', \hspace{1cm} \Psi f \sim \Psi g: \Psi \SSS \longrightarrow \Psi \SSS'. 
\end{equation}
\end{proposition}
\begin{proof}
It follows from \eqref{eq:square}, which implies the isomorphism $\KK[{\sf cyl}(\SSS)]\cong {\sf cyl}(\KK[\SSS]),$ and Theorem \ref{th:homotopy}.
\end{proof}

\begin{lemma}\label{lemma_path_pair_sets}
For any path pair of sets $\SSS=(X,Y),$ if we set $(A,B):=(\KK[X],\KK[Y]),$ then $A_n$ is a free module and $\overline{B}_n$ is a direct summand in $A_n.$ 
\end{lemma}
\begin{proof} The module $A_n=\KK[X_n]$ is free by the definition. Prove that $\overline{B}_n$ is a direct summand in $A_n.$ 
We set ${\sf D}_nX=\bigcup_{i=0}^{n-1} s_i(X_{n-1})$ and ${\sf N}_nX=X_n \setminus {\sf D}_nX.$ Then ${\sf N}_nA=\KK[{\sf N}_nX]$ and the map $\rho:A_n \to {\sf N}_nA$ is identical on elements of ${\sf N}_nX$ and trivial on elements of ${\sf D}_nX.$ Therefore, $\overline{B}_n=\KK[ ({\sf N}_nX) \cap Y_n].$ The assertion follows.
\end{proof}

\begin{theorem}\label{th:kunneth_sets} If $\KK$ is a principal ideal domain, for any two path pairs of sets $\SSS$ and $\SSS'$ there is an isomorphism 
\begin{equation}
    \Omega (\SSS \square \SSS) \cong \Omega \SSS \otimes \Omega \SSS'. 
\end{equation}
Moreover, there is a short exact sequence 
\begin{equation} 0 \to
\bigoplus_{i+j=n} H_i(\SSS) \otimes H_j(\SSS') \to    H_n( \SSS \square \SSS' ) \to \bigoplus_{i+j=n-1} {\sf Tor}^\KK_1( H_i(\SSS),H_j(\SSS')) \to 0.
\end{equation}
\end{theorem}
\begin{proof}
It follows from Lemma \ref{lemma_path_pair_sets}, Theorem \ref{th:EZ} and an isomorphism \eqref{eq:square}.
\end{proof}

\subsection{Regular path complexes}
In this subsection we remind the definition of a regular path complex $P$ and its regular complex of $\delta$-invariant paths $\Omega(P)$ given in \cite{grigor2012homologies}. Further we show that this complex can be defined on the language of path sets. 

For any set $V$ we denote by ${\sf cosk}_0(V)$ the simplicial set with components ${\sf cosk}_0(V)_n=V^{n+1},$ whose face and degeneracy maps are defined by formulas 
\begin{equation}
\begin{split}
d_i(v_0,\dots,v_n)&=(v_0,\dots, v_{i-1},v_{i+1},\dots,v_n),\\
s_i(v_0,\dots,v_n)&=(v_0,\dots, v_i,v_i,\dots,v_n).
\end{split}    
\end{equation}
It is easy to check that this simplicial set is the $0$-coskeleton of $V$ treated as a $0$-truncated simplicial set. Note that degenerated elements of ${\sf cosk}_0(V)$ are sequences with repetitions, which are called irregular paths in \cite{grigor2012homologies}. For a sequence $(v_0,\dots,v_n)\in V^{n+1}$ we set 
\begin{equation}
\Ker(v_0,\dots,v_n)= \{0\leq i\leq n-1\mid v_i=v_{i+1}\}.   
\end{equation}
Then $(v_0,\dots,v_n)$ is \emph{regular} if and only if $\Ker(v_0,\dots,v_n)=\emptyset.$ For arbitrary sequence $(v_0,\dots,v_n)$ we denote by $(v_0,\dots,v_n)_{\sf reg}$ the regular sequence with deleted repetitions. Then
\begin{equation}\label{eq:kerker2}
(v_0,\dots,v_n) = s_\mu((v_0,\dots,v_n)_{\sf reg}),   
\end{equation}
where $\Ker(v_0,\dots,v_n)=\{\mu_0 < \dots < \mu_{l-1}\}.$ On the other hand for any $f:[n]\to [k]$ and any $(u_0,\dots,u_k)\in V^{k+1}$ we have 
\begin{equation}\label{eq:kerker}
\Ker(f) \subseteq  \Ker(f^*(u_0,\dots,u_k)).     
\end{equation}

If $\KK$ is a commutative ring, then, following \cite{grigor2012homologies}, we set 
\begin{equation}
\Lambda(V)=\KK[{\sf cosk}_0(V)]
\end{equation}
Degenerated elements $D(\Lambda(V))$ of $\Lambda(V)$ are linear combinations of irregular paths. The Moore complex $\NN(\Lambda(V)) $ of this simplicial module  is denoted by
\begin{equation}
\mathcal{R}(V)=\NN(\Lambda(V).
\end{equation}
Here we identify $\NN(\Lambda (V))$ with $\CC(\Lambda (V))/\DD(\Lambda (V)).$ This complex is called the complex of regular paths. The set of all sequences  $(v_0,\dots,v_n)\in V^{n+1}, v_i\ne v_{i+1}$ forms a basis of $\mathcal{R}_n(V).$

A \emph{path complex} is a couple $P=(V,(P_n)_{n=0}^\infty)$ where $V$ is a set and $P_n\subseteq V^{n+1}$ such that if $(v_0,\dots,v_n)\in P_n,$ then $(v_0,\dots,v_{n-1}),$  $(v_1,\dots,v_n)\in P_{n-1}.$ A sequence $(v_0,\dots,v_n)$ is called \emph{regular}, if $v_i\ne v_{i+1}.$ A path complex is called {\it regular} if $P_n$ consists of regular sequences for any $n.$ Equivalently, $P$ is regular, if $(v,v)\notin P_1$ for any $v\in V.$ 

For a regular path complex $P$ we define $\mathcal{A}_n(P)\subseteq \mathcal{R}_n(V)$ as the submodule generated by the images of elements from $P_n.$ So $\mathcal{A}(P)$ is a graded submodule of the chain complex $\mathcal{R}(V).$ Then the complex of $\partial$-invariant forms $\Omega P$ is defined as 
\begin{equation}
\Omega P = \omega(\mathcal{R}(V), \mathcal{A}(P))  
\end{equation}
The homology of $P$ is defined as the homology of $\Omega P.$

\subsection{Complete path complexes}

A path complex $P$ is called \emph{complete}, if
\begin{equation}
 (v_0,\dots,v_n)\in P_n \hspace{5mm} \Rightarrow \hspace{5mm} (v_0,\dots, v_i,v_i,\dots, v_n)\in P_{n+1}.
\end{equation}
For any path complex $P$ we can consider the minimal complete path complex containing $P$ with the same vertex set and denote it by $\widehat{P}.$ The path complex $\widehat{P}$ is called the completion of $P.$

It is easy to see that a complete path complex can be defined as a path subset of ${\sf cosk}_0(V).$ So any complete path complex $P$ defines a path pair of sets 
\begin{equation}
    \SSS P=({\sf cosk}_0(V),P).
\end{equation}

Since ${\sf cosk}_0(V)$ is contractible, the long exact sequence \eqref{eq:long_set} implies that for $n\geq 1$ we have an isomorphism
\begin{equation}
H_n(\SSS P) \cong H_{n+1}^{\sf c}(\SSS P).
\end{equation}

\begin{proposition}
Let $P$ be a regular path complex and let $\widehat{P}$ be its completion. Then there is an isomorphism 
\begin{equation}
    \Omega P \cong \Omega (\SSS(\widehat{P})).
\end{equation}
\end{proposition}
\begin{proof}
Let $(A,B)=(\KK[{\sf cosk}_0(V)],\KK[ \widehat{P}]).$ Then 
$\NN A=\mathcal{R}(V)$ and it is easy to check that $\overline{B}_n= \mathcal{A}_n(P).$ Then  we have  $\Omega(\SSS(\widehat{P}))=  \omega(\mathcal{R}(V),\mathcal{A}( P))\cong \Omega P.$ 
\end{proof}

\subsection{Box product of path complexes} 

We denote by 
\begin{equation}
\theta:(V\times V')^{n+1} \to V^{n+1}\times (V')^{n+1}    
\end{equation}
the obvious bijection and  $\theta_1:(V\times V')^{n+1} \to V^{n+1}$ and $\theta_2:(V\times V')^{n+1} \to (V')^{n+1}$ are its components. For two complete path complexes $P=(V,(P_n)), P'=(V',(P_n')),$ we define the box product $P\square P'$ as a path complex on the set $V\times V'$ such that $(P\square P')_n$ consists of sequences $(w_0,\dots,w_n)$ such that \begin{equation}
\theta(w_0,\dots,w_n)= (f^*(v_0,\dots,v_k), g^*(v'_0,\dots,v'_k))    
\end{equation}
for some $(f,g)\in \PPi_\square(n).$ Then $\theta$ restricts to a bijection
\begin{equation}
(P \square P')_n \cong
\bigcup_{(f,g)\in \PPi_\square(n)}    f^*(P_{|f|}) \times g^*(P'_{|g|}). 
\end{equation}

By definition we obtain
\begin{equation}\label{eq:path_box}
\SSS(P\square P') \cong \SSS P\square \SSS P'.    
\end{equation}

A sequence $(w_0,\dots,w_n)$ of pairs $w_i=(v_i,v'_i)\in  W$  is called \emph{step-like} if for any $0\leq i\leq n-1$ either $v_i=v_{i+1}$ or $v'_i=v'_{i+1}$ (or both). In other words, a sequence $(w_0,\dots,w_n)$ is step-like, if $\Ker(v_0,\dots,v_n) \cup \Ker(v'_0,\dots,v'_n)=[n-1].$ 
Using \eqref{eq:kerker}, it is easy to see that  all sequences from $(P \square P')_n$ are step-like. 

For regular path complexes $P,P'$ we define their regular box product $P\square_{\sf reg} P'$ such that $(P\square_{\sf reg} P')_n$ consists of regular step-like sequences $(w_0,\dots,w_n)$ of pairs $w_i=(v_i,v'_i)$ such that $(v_0,\dots,v_n)_{\sf reg}\in P_k$ and $(v'_0,\dots,v'_n)_{\sf reg}\in P_l$ for some $k,l\leq n.$

\begin{proposition}\label{prop:box_reg} For any regular path complexes $P,P'$ we have
\begin{equation}
    (P\square_{\sf reg} P')^\wedge= \widehat{P} \square \widehat{P'}.
\end{equation}
\end{proposition}
\begin{proof} Let $(w_0,\dots,w_n)\in (P\square_{\sf reg} P')_n$ and $w_i=(v_i,v'_i).$ Then by  \eqref{eq:kerker2} we obtain that $\theta(w_0,\dots,w_n)= (s_\mu((v_0,\dots,v_n)_{\sf reg}), s_\nu((v'_0,\dots,v'_n)_{\sf reg})),$ where $\{\mu_0,\dots,\mu_{l-1}\} \cup \{\nu_0,\dots,\nu_{k-1} \}=[n-1].$ Therefore $(w_0,\dots,w_n)\in (\widehat{P} \square \widehat{P'})_n.$ So we proved $(P\square_{\sf reg} P')^\wedge \subseteq \widehat{P} \square \widehat{P'}.$ 

Now assume that $(w_0,\dots,w_n)\in (\widehat{P} \square \widehat{P'})_n.$
Then 
\begin{equation}
\theta(w_0,\dots,w_n)=(f^*(u_0,\dots,u_k), g^*(u'_0,\dots,u'_l)),    
\end{equation}
where $(u_0,\dots,u_k)\in \widehat{P}_k,$  $(u'_0,\dots,u'_l)\in \widehat{P'}_l$ and $(f,g)\in \PPi_\square(n).$
We need to prove that $(w_0,\dots,w_n)_{\sf reg}\in (P \square_{\sf reg} P')_n.$ Note that $\theta_i((w_0,\dots,w_n)_{\sf reg})_{\sf reg}=\theta_i(w_0,\dots,w_n)_{\sf reg}$ for $i=1,2.$ Also note that if $(w_0,\dots,w_n)$ is step-like, then $(w_0,\dots,w_n)_{\sf reg}$ is also step-like. Then we only need to prove that $(f^*(u_0,\dots,u_k))_{\sf reg} \in P_{k'}$ and $(g^*(u'_0,\dots,u'_l))_{\sf reg} \in P'_{l'}$ for some $k',l',$ which is obvious, because they are regular sequences of $\widehat{P}$ and  $\widehat{P}'$ respectively. Hence $ \widehat{P} \square \widehat{P'}\subseteq (P\square P')^\wedge.$
\end{proof}

It is proved in  \cite[Th.7.6]{grigor2012homologies} that for any regular path complexes $P,P'$ and any field $\KK$  there is an isomorphism
\begin{equation}
    \Omega( P\square_{\sf reg} P' ) \cong \Omega P \otimes \Omega P'.
\end{equation}
This isomorphism follows from Proposition \ref{prop:box_reg}, the isomorphism \eqref{eq:path_box} and Theorem \ref{th:kunneth_sets}, which is a corollary of Theorem \ref{th:EZ}. So Theorem \ref{th:EZ} can be regarded as a generalization of \cite[Th.7.6]{grigor2012homologies}.

\section{Marked categories}\label{sec:embedded_quivers}
 
\subsection{Marked categories}
Mimicking the definition of marked simplicial sets \cite[\S 3.1]{lurie2009higher} we define marked categories as couples $\MMM = (\CCC, M),$ where $\CCC$ is a (small) category and $M\subseteq {\sf Mor}(\CCC)$ is a set of morphisms, called marked morphisms, containing all the identity morphisms. A morphism of marked categories $(\CCC,M)\to (\CCC',M')$ is a functor $\CCC\to \CCC'$ that takes $M$ to $M'.$ We denote by $Q_\MMM$ the associated quiver, whose vertices are objects of $\CCC$ and arrows are morphisms from $M.$ 

It is easy to check that a marked category $\MMM=(\CCC,M)$ defines a path pair of sets 
\begin{equation}
\SSS \MMM =(\Nrv(\CCC),{\sf P}Q_\MMM),
\end{equation}
where $\Nrv(\CCC)$ is the nerve of $\CCC$ and  ${\sf P}Q_\MMM$ is the path set associated with the quiver $Q_\MMM,$ which consists of composable sequences of marked morphisms. We also set $\Omega\MMM = \Omega(\SSS\MMM),$ $\Psi \MMM = \Psi(\SSS\MMM).$ The path and copath homology of $\MMM$ are defined as
\begin{equation}
\PH_*(\MMM) = H_*(\Omega \MMM), \hspace{1cm} \PH_*^{\sf c}(\MMM) = H_*(\Psi\MMM).
\end{equation}
Then  \eqref{eq:long_set} implies that there is a long exact sequence
\begin{equation}\label{eq:long_emb}
\dots \to \PH_n(\MMM) \to H_n(\CCC) \to \PH_n^{\sf c}(\MMM) \to \PH_{n-1}(\MMM) \to \dots 
\end{equation}

\subsection{Detailed description and low dimensions.} 
The nerve of a small category $\CCC$ is a simplicial set whose $n$-th component $\Nrv(\CCC)_n$ consists of composable $n$-sequences of morphisms $(\gamma_1,\dots,\gamma_n)$ if $n\geq 1,$ and  $\Nrv(\CCC)_0={\sf ob}(\CCC).$ The Moore complex of the corresponding simplicial module is denoted by $\NN \CCC =\NN(\KK[\Nrv(\CCC)]).$ The image of $(\gamma_1,\dots,\gamma_n)$ in $\NN \CCC $ is denoted by $\langle \gamma_1,\dots,\gamma_n \rangle.$
Then $(\NN\CCC)_n$ is a free module freely generated by all elements $\langle \gamma_1,\dots,\gamma_n \rangle,$ where $\gamma_i\ne 1_{v}$ for any object $v.$  If $\gamma_i=1_v$ for some $i,$  then $\langle \gamma_1,\dots,\gamma_n
\rangle=0.$

We also consider the maps 
\begin{equation}
    \tilde d_i : (\NN \CCC)_n \longrightarrow (\NN \CCC)_{n-1}
\end{equation}
defined on the basis by the formulas
\begin{equation}
\begin{split}
\tilde d_0 \langle \gamma_1,\dots,\gamma_n\rangle, &=\langle \gamma_2,\dots,\gamma_n\rangle \\     
\tilde d_i \langle \gamma_1,\dots,\gamma_n\rangle &= \langle \gamma_1,\dots, \gamma_{i+1}\gamma_i, \dots, \gamma_n \rangle, \hspace{1cm} 1\leq i \leq n-1, \\ 
\tilde d_n \langle \gamma_1,\dots,\gamma_n \rangle & = 
\langle \gamma_1,\dots,\gamma_{n-1} \rangle,
\end{split}
\end{equation}
if $n\geq 2,$ and $\tilde d_0(\alpha)={\sf codom}(\alpha)$ and $\tilde d_1(\alpha)={\sf dom}(\alpha)$ for $n=1.$

By the definition the 
differential $\partial^{\NN \CCC}$ on $\NN \CCC$ is induced by the differential 
on $\KK[\Nrv(\CCC)],$ which is defined as $ \sum (-1)^i d_i.$ 
It is not difficult to check that the differential on $\NN \CCC$ satisfies 
\begin{equation}
\partial^{\NN \CCC} = \sum_{i=0}^n (-1)^i \tilde d_i.    
\end{equation}

For a marked category $\MMM=(\CCC,M)$ we set
\begin{equation}
\overline{\KK[{\sf P}Q_\MMM]} = \rho(\KK[{\sf P}Q_\MMM]).    
\end{equation}

Generalising the definition of $\DD A$ and $\NN A$ for a simplicial module $A$ to a path module $B,$ we consider a graded modules $\DD B$ and $\NN B$ defined by the formulas
\begin{equation}
(\DD B)_n = \sum_i s_i(B_{n-1}), \hspace{1cm} \NN B = B/\DD B.    
\end{equation}

\begin{lemma}\label{lemma:DB}
Let $\MMM=(\CCC,M)$ be a marked category and we set $A=\KK[\Nrv(\CCC)]$ and $B=\KK[{\sf P}Q_\MMM].$  Then 
\begin{equation}
(\DD B)_n = (\DD A)_n \cap B_n.
\end{equation} 
\end{lemma}
\begin{proof}  The inclusion  $\subseteq $ is obvious. Prove $\supseteq.$
The set of all $n$-sequences of 
composable morphisms is a basis of $A_n.$ The set of $n$-sequences containing an 
identity morphism is a basis of $(\DD A)_n.$ The set of all $n$-sequences of composable morphisms from $M$ is a basis of 
$B_n.$ So $(\DD A)_n$ and $B_n$ are generated by subsets of the basis of $A_n.$ 
So, the intersection of these subsets is a basis of the 
intersection $(\DD A)_n \cap B_n.$ Then the set of all $n$-sequences of composable marked
morphisms containing an identity morphism is a basis of the intersection $(\DD A)_n \cap B_n.$ All elements of this basis 
are in $(\DD B)_n,$ which implies the inclusion $\supseteq.$
\end{proof}

If we set $\NN M:=\NN(\KK[{\sf P}Q_\MMM]),$ we obtain that Lemma \ref{lemma:DB} implies that $\rho$ induces an isomorphism
\begin{equation}
\NN M \cong \overline{\KK[{\sf P}Q_\MMM]}.
\end{equation}
We will identify $\NN M$ with its image in $\NN \CCC.$ Then 
\begin{equation}
\Omega\MMM = \omega(\NN \CCC,\NN M), \hspace{1cm} \Psi\MMM = \psi(\NN \CCC,\NN M).
\end{equation}

\begin{proposition}[{cf. \cite[Prop. 4.2]{grigor2012homologies}}] \label{prop:low_dim}
For any marked category $\MMM$ we have 
$(\Omega \MMM)_n=(\NN M)_n$ for $n=0,1$ and 
\begin{equation}\label{eq:omega_2}
(\Omega \MMM)_2= (\NN M)_2\cap   \tilde d_1^{-1}((\NN M)_1).
\end{equation}
Moreover, $(\Omega\MMM)_2$ is generated by  differences of composable pairs $\langle \alpha_1,\beta_1\rangle - \langle \alpha_2,\beta_2 \rangle$ such that $\beta_1\alpha_1=\beta_2\alpha_2$ and $\alpha_i,\beta_i\in M.$
\end{proposition}
\begin{proof} 
The equation for $n=0,1$ are obvious. For $n=2$ we have $\partial_2^{\NN M}=\tilde{d}_0  - \tilde{d}_1 + \tilde{d}_2.$ For any $x\in (\NN M)_2$ we have $\tilde{d}_0(x), \tilde{d}_2(x)\in (\NN M)_1.$ Hence $\partial^{\NN M}_2(x) \in (\NN M)_1$ if and only if $ \tilde d_1(x)\in (\NN M)_1.$ The equation \eqref{eq:omega_2} follows.  

Denote by $D\subseteq (\NN M)_2$ the submodule generated by $\langle \alpha_1,\beta_1\rangle - \langle \alpha_2,\beta_2 \rangle$ for $\beta_1\alpha_1=\beta_2\alpha_2.$ It is easy to see that $D\subseteq (\Omega\MMM)_2.$ Prove that $(\Omega\MMM)_2\subseteq D.$ Let $x \sum_{i=1}^n a_i \langle \alpha_i,\beta_i \rangle \in (\Omega\MMM)_2,$ where $a_i\in \KK\setminus\{0\}$ and $\alpha_i,\beta_i\in M\setminus\{1_c|c\in {\sf ob}(\CCC)\}.$ We want to prove $x\in D$ by induction on $n.$ For $n=0$ this is obvious. Assume that $n\geq 1.$ Then $\tilde d_1(x)=\sum_{i=1}^n a_i \langle \beta_i \alpha_i \rangle=0.$ If  $ \beta_n\alpha_n=1_v$ for some object $v$ then $\langle \alpha_n,\beta_n\rangle = \langle \alpha_n,\beta_n\rangle - \langle 1_v,1_v \rangle \in D$ and by induction hypothesis $x- a_n\langle \alpha_n,\beta_n\rangle\in D.$ Hence $x\in D.$ So we can assume that $\beta_n\alpha_n\ne 1_v.$ Thus $a_n \langle \beta_n\alpha_n \rangle\ne 0,$ and hence, $n\geq 2.$  Then there exists $1\leq m<n$ such that $\beta_m \alpha_m=\beta_n \alpha_n.$ It follows that  $x-a_n(\langle \alpha_n, \beta_n \rangle - \langle \alpha_m, \beta_m \rangle) \in D$ by the induction hypothesis. Hence $x\in D.$
\end{proof}

\subsection{DG-coalgebra structure on  \texorpdfstring{$\Omega$}{Ω}} 

Consider the composition of the map induced by the diadonal $\NN \CCC\to \NN(\CCC\times \CCC)$ and the Alexander-Whitney map $\alpha:\NN(\CCC\times \CCC) \to \NN\CCC \otimes \NN \CCC$ 
\begin{equation}
    \nu : \NN \CCC \longrightarrow \NN \CCC \otimes \NN \CCC.
\end{equation}
This map $\nu_n: \NN_n \CCC \to (\NN \CCC\otimes \NN \CCC)_n$ can be written explicitly on the basis as $\nu_n=\sum_{k+l=n} \nu_{k,l},$ where 
\begin{equation}
\nu_{k,l}:\NN_n \CCC \to \NN_k\CCC \otimes \NN_l \CCC    
\end{equation}
is defined by the formulas
\begin{equation}
\begin{split}
\nu_{0,n}(\langle \alpha_1,\dots, \alpha_n \rangle)& = 1_{v_0}\otimes \langle \alpha_1,\dots, \alpha_n \rangle \\
\nu_{k,l}(\langle \alpha_1,\dots, \alpha_n \rangle)& = \langle \alpha_1,\dots, \alpha_i\rangle \otimes \langle \alpha_{i+1} ,\dots, \alpha_n \rangle, \hspace{1cm} k,l\geq 1 \\
\nu_{n,0}(\langle \alpha_1,\dots, \alpha_n \rangle)& = \langle \alpha_1,\dots, \alpha_n \rangle \otimes 1_{v_n}.
\end{split}    
\end{equation}
We also consider the map $e:\NN\CCC \to (\NN \CCC)_0 \to \KK$ given by $e(1_v)=1.$ It is easy to see that $\nu,e$ define a structure of dg-coalgebra on $\NN \CCC.$

\begin{proposition}\label{prop:coalg_embed}
Let $\KK$ be a principal ideal domain and $\MMM$ be a marked category. Then  $\Omega \MMM$ is a split sub-dg-coalgebra of $\NN \CCC$.
\end{proposition}
\begin{proof}
It is easy to check that $\NN 
M \subseteq \NN \CCC$ is a split subcoalgebra. Then the assertion follows from Proposition \ref{prop:coalg_omega}. 
\end{proof}

\begin{theorem}\label{th:embedding}
Let $\KK$ be a principal ideal domain and $k,l$ are natural numbers. Then for any marked category $\MMM$ the map $\nu_{k,l}$ induces a monomorphism 
\begin{equation}
\Omega_{k+l} \MMM \mono \Omega_k \MMM \otimes \Omega_l \MMM.
\end{equation}
\end{theorem}
\begin{proof}
The map $\nu_{k,l}: \NN_{k+l} \CCC \to \NN_k \CCC \otimes \NN_l \CCC $ is a monomorphism because it sends different elements of the basis to different elements of the basis. Proposition     \ref{prop:coalg_embed} implies that it can be restricted to the map $\Omega_{k+l} \MMM \mono \Omega_k \MMM \otimes \Omega_l \MMM,$ which is also a monomorphism. 
\end{proof}

\begin{corollary}[{cf. \cite[Prop. 3.23]{grigor2012homologies}}]\label{cor:dimensions}
If $\KK$ is a principal ideal domain, $\MMM$ is a marked category and $\Omega_n \MMM=0$ for some $n,$ then for any $m>n$ we also have $\Omega_m\MMM=0.$ 
\end{corollary}

\begin{corollary}\label{cor:dim}
If $\KK$ is a field and $\MMM$ is a marked category, then 
\begin{equation}
{\sf dim}(\Omega_{k+l}\MMM) \leq  {\sf dim}(\Omega_k\MMM)\cdot {\sf dim}(\Omega_l\MMM)    
\end{equation}
for any natural $k,l.$
\end{corollary}

\subsection{Cohomology of marked categories, cup product}\label{subsec:cohomology_of_embedded_quivers} For a $\KK$-module $V$ we set $M^\vee={\sf Hom}(V,\KK).$ For any category $\CCC$ we consider the cochain complex $(\NN \CCC)^\vee.$ The cochain complex has a natural structure of dg-algebra defined by the composition
\begin{equation}
(\NN \CCC)^\vee \otimes (\NN\CCC)^\vee \longrightarrow ( \NN\CCC \otimes \NN\CCC )^\vee \overset{\nu^\vee}\longrightarrow (\NN \CCC)^\vee. 
\end{equation}
Similarly we can define a graded algebra structure (without a differential) on $(\NN M)^\vee.$ For a marked category  $\MMM=(\CCC,M)$ we obtain a homomorphism of graded algebras $(\NN \CCC)^\vee \to (\NN M)^\vee,$ whose kernel is denoted by $K\MMM.$ Then $K\MMM$ is an graded ideal of $(\NN \CCC)^\vee.$  We set 
\begin{equation}
\Omega^\bullet \MMM = \psi( (\NN C)^\vee,K\MMM)   
\end{equation}
and define the path cohomology of $\MMM$ by the formula
\begin{equation}
\PH^*(\MMM) := H^*(\Omega^\bullet \MMM).
\end{equation}
By Proposition \ref{prop:psi-dg} $\Omega^\bullet \MMM$ inherits a natural structure of a dg-algebra. Hence $\PH^*(\MMM)$ has a natural structure of a graded algebra. Note that this structure is defined for any commutative ring $\KK.$

The construction is natural and any morphism of marked categories $\MMM\to \MMM'$ defines a homomorphism of graded algebras
\begin{equation}
    \PH^*(\MMM') \longrightarrow \PH^*(\MMM).
\end{equation}
In particular, for any marked category $\MMM=(\CCC,M)$ we have a homomorphism from the cohomology algebra of the category to the cohomology of the marked category $H^*(\CCC) \to \PH^*(\MMM).$

Similarly to Theorem \ref{th:embedding} we obtain that the product on $\Omega^\bullet\MMM$ defines an epimorphism 
\begin{equation}
    \Omega^k \MMM \otimes \Omega^l \MMM \epi \Omega^{k+l} \MMM
\end{equation}
for any natural $k,l.$ If $\KK$ is a field, then by Proposition \ref{prop:duality} we obtain that 
\begin{equation}
    \Omega^\bullet \MMM \cong (\Omega \MMM)^\vee, \hspace{1cm} \PH^*(\MMM) \cong (\PH_*(\MMM))^\vee. 
\end{equation}

\subsection{Box product of marked categories}

We define the box product of marked categories $\MMM=(\CCC,M)$ and $\MMM'=(\CCC',M')$  by the formula 
\begin{equation}
\MMM \square \MMM' = (\CCC\times \CCC', \{(\mu,1_{c'})\} \cup \{(1_c,\mu')\} ),
\end{equation}
where $\{(\mu,1_{c'})\}$ stands for 
$\{(\mu,1_{c'})\mid \mu\in M, c\in {\sf ob}(\CCC') \}$ and $\{(1_c,\mu')\}$ stands for $\{(1_c,\mu')\} = \{ (1_c,\mu') \mid c\in {\sf ob}(\CCC), \mu'\in M' \}.$ It is easy to see that 
\begin{equation}
Q_{\MMM\square \MMM'} \cong Q_{\MMM} \Box Q_{\MMM'}.     
\end{equation}
Proposition \ref{prop:paths_in_product} implies that the box product is compatible with the box product
\begin{equation}\label{eq:SE_box}
\SSS(\MMM\square \MMM') \cong \SSS\MMM \square \SSS\MMM'.     
\end{equation}

\begin{proposition}\label{prop:EZ_embedded}
If $\KK$ is a principal ideal domain, for any marked categories $\MMM,\MMM'$ we have
\begin{equation}
\Omega(\MMM\square \MMM')\cong \Omega\MMM \otimes \Omega\MMM'.    
\end{equation}
Moreover, if we set $H_*=\PH_*(\MMM),$ $H'_*=\PH_*(\MMM')$ and $H''_*=\PH(\MMM\times \MMM'),$ we get a K\"unneth-like short exact sequence
\begin{equation}  0 \longrightarrow 
\bigoplus_{i+j=n} H_i\otimes H'_j \longrightarrow H''_n \longrightarrow \bigoplus_{i+j=n-1} {\sf Tor}^\KK_1(H_i,H'_j) \longrightarrow 0.
\end{equation}
\end{proposition}
\begin{proof}
It follows from Theorem \ref{th:kunneth_sets} and \eqref{eq:SE_box}. 
\end{proof}

\subsection{Homotopy invariance for marked categories}

Consider a marked category that models the interval $I^{\sf m}=(\mathpzc{F}(\qq^1),\qq^1_1),$ where $\qq^1$ is the quiver with two vertices $(0\to 1)$ and $\mathpzc{F}(\qq^1)$ is the free category defined by the quiver with only one non-identical morphism. We also consider the marked category that models a point ${\sf pt}^{\sf m}=(\mathpzc{F}(\qq^0),\qq^0_1).$ Then for any marked category $\MMM$ we have $\MMM\square{\sf pt}^{\sf m}\cong \MMM.$ 
Hence two morphisms $i^0,i^1:{\sf pt}^{\sf m}\rightrightarrows I^{\sf m}$
define a weak cylinder functor 
${\sf cyl}(\MMM) = \MMM \square I^{\sf m},$
and we define homotopic morphisms of marked categories via this weak cylinder functor. 

\begin{proposition}\label{prop:nat_tranf}
Two morphisms of marked categories $f,g:(\CCC,M)\to (\CCC',M')$ are one-step homotopic if and only if there is a natural transformation $\varphi:f\to g$ such that $\varphi_c\in M'$ for any object $c$ of $\CCC.$
\end{proposition}
\begin{proof}
The set of natural transformations $\varphi:f\to g$ is in bijection with the set of functors $h:\CCC \times \mathpzc{F}(\qq^1) \to \CCC'$ such that $h(\alpha,{\sf id}_0)=f(\alpha),$ $h(\alpha,{\sf id}_1)=g(\alpha)$ for any morphism $\alpha$ of $\CCC.$ The functor $h$ corresponding to $\varphi$ is defined by the formula  $h({\sf id}_c,(0,1))= \varphi_c.$  The assumptions $f(M)\subseteq M', g(M)\subseteq M'$ and $\varphi_c \in M'$ are equivalent to the fact that this morphism defines a morphism of marked categories $h: \MMM \square I^{\sf m} \to \MMM'$ such that $hi^0_\MMM = f$ and $hi^1_\MMM =g.$ The assertion follows. 
\end{proof}

\begin{proposition} \label{prop:homotopy_emb}
Any two homotopic morphisms of marked categories $f\sim g:\MMM\to \MMM'$ induce homotopic morphisms of complexes 
\begin{equation}
\Omega f \sim \Omega g : \Omega \MMM \to \Omega \MMM' \hspace{1cm} \Psi f \sim \Psi g : \Psi \MMM \to \Psi \MMM'.
\end{equation}
\end{proposition}
\begin{proof}

Then the assertion follows from Proposition \ref{prop:homotopy_set},  \eqref{eq:SE_box} and the formulas $ \SSS(I^{\sf m})=I^{\sf s}$ and $\SSS({\sf pt}^{\sf m})= {\sf pt}^{\sf s}.$
\end{proof}

\subsection{Isomorphism-lemma for marked categories}

For a marked category $\MMM=(\CCC,M)$ we denote by $M^2$ the set of compositions $\mu\mu',$ where $\mu$ and $\mu'$ are composable morphisms from $M.$ Since all identity morphism are in $M,$ we obtain $M\subseteq M^2.$

\begin{proposition}[Isomorphism-lemma]\label{prop:isomorphism-lemma}
Let $\MMM=(\CCC,M)$ and $\MMM'=(\CCC',M')$ be marked categories and $f:\MMM \to \MMM'$ be a morphism of marked categories, which induces a bijection on objects and bijections $M\cong M'$ and $M^2\cong (M')^2.$   Then $f$ induces an isomorphism
\begin{equation}
\Omega\MMM\cong \Omega\MMM'.
\end{equation}
\end{proposition}
\begin{proof} 
Set $A=\KK[\Nrv(\CCC)],$ $B=\KK[ {\sf P}Q_\MMM ]$ and $E=\KK[{\sf P} Q_{\MMM^{[2]}}],$ where $\MMM^{[2]}=(\CCC,M^2).$ As usual, we also denote by $\overline{B}$ the image of $B\to \NN A$ and by $\overline{E}$ the image of $E\to \NN A.$ We use similar notation for $(\CCC',M'):$ $A',B', E'$. The fact that $f$ induces isomorphisms $M\cong M'$ and $M^2\cong (M')^2$ implies that $f$ induces isomorphisms $B\cong B'$ and $E\cong E'.$ Therefore, $f$ induces isomorphisms $\NN B \cong \NN B'$ and $\NN E \cong \NN E'.$  By Lemma  \ref{lemma:DB} we see that $\NN B\cong \overline{B},$ $\NN E \cong \overline{E}$ and  $\NN B'\cong \overline{B}',$ $\NN E' \cong \overline{E}'.$ Then $f$ induces isomorphisms $\overline{B}\cong \overline{B}'$ and $\overline{E}\cong \overline{E}'.$ By the definitions of $M^2$ we see that $\partial(\overline{B})\subseteq \overline{E}.$ Similarly $\partial(\overline{B}')\subseteq \overline{E}'.$ Then the assertion follows from Proposition  \ref{prop:E-iso}.
\end{proof}

\begin{corollary}\label{cor:embedding_categories} 
If $\CCC$ is a subcategory of $\CCC',$  then 
\begin{equation}
\Omega(\CCC,M)\cong \Omega(\CCC',M).    
\end{equation}
\end{corollary}

\subsection{Categories with ideals and free categories}

If $\CCC$ is a small category, a set $I\subseteq {\sf mor}(\CCC)$ is called ideal, if a composition of a morphism from $I$ with any other morphism, from any side, is also from $I.$ 

\begin{proposition}\label{prop:ideal}
Let $\MMM=(\CCC,M)$ be a marked category and let $I$ be an ideal of $\CCC$ such that $M\cap I=\emptyset$ and the composition of any two non-identical arrows from $M$ are in $I.$ Then 
\begin{equation}
\Omega_n\MMM=\{x\in (\NN M)_n \mid \tilde d_i(x)=0,\  \text{for}\ 1\leq i\leq n-1 \}.
\end{equation}
The differential $\partial:\Omega_n\MMM\to \Omega_{n-1}\MMM$ is given by $\partial=\tilde d_0+(-1)^n\tilde d_n.$
\end{proposition}
\begin{proof} 
Denote by $E_{n-1,i}$ a submodule of $(\NN \CCC)_{n-1}$ which is generated by composable $(n-1)$-sequences of non-identical morphisms   $ \langle \gamma_1,\dots,\gamma_{n-1} \rangle$  such that $\gamma_i\in I$ and $\gamma_j\notin I$ for $j\ne i.$ 
We also denote by $E_{n-1}$ the sum of the modules $E_{n-1,i}.$
It is easy to see that $E_{n-1}=\bigoplus_{i=1}^{n-1} E_{n-1,i}.$ Also note that $\tilde d_i((\NN M)_n )\subseteq E_{n-1,i}$ for $1\leq i\leq n$ and $\tilde d_0((\NN M)_n) , \tilde d_n((\NN M)_n)\subseteq (\NN M)_{n-1}.$ Moreover, note that $E_{n-1}\cap (\NN M)_{n-1}=0.$ It is easy to see that $\tilde d_0( (\NN M)_n ), \tilde d_n((\NN M)_n) \subseteq (\NN M)_{n-1}.$ So for $b\in (\NN M)_n$ 
\begin{equation}
\partial^{\NN M}(b) = \big(\tilde d_0(b)+(-1)^n \tilde d_n(b)\big)  + \sum_{i=1}^{n-1}(-1)^i \tilde d_i(b).
\end{equation}
Therefore $\partial(b)\in (\NN M)_{n-1}$ if and only if $\tilde d_i(b)=0$ for $1\leq i\leq n-1.$
\end{proof}

\begin{proposition} \label{prop:free}
Let $Q$ be a quiver and $\mathpzc{F}(Q)$ be the free category (the category of paths) generated by $Q.$ Then
\begin{equation}
\Omega_n(\mathpzc{F}(Q),Q_1)=0, \hspace{1cm} n\geq 2   
\end{equation}
and  $\Omega_n(\mathpzc{F}(Q),Q_1)\cong \KK^{Q_n}$ for $n=0,1.$ Moreover, $\PH_*(\mathpzc{F}(Q),Q_1)$ is isomorphic to the homology of the quiver $Q$ treated as $1$-dimensional space. 
\end{proposition}
\begin{proof} We denote by $I$ the ideal of paths of length $\geq 2.$ Then it follows from Proposition \ref{prop:ideal} and the fact that for the free category the maps $\tilde d_i: \NN(\mathpzc{F}(Q))_n \to \NN(\mathpzc{F}(Q))_{n-1} $ are injective for $1\leq i\leq n-1.$
\end{proof}

\subsection{Digraphs as marked categories}
By a digraph $G$ we mean a couple of sets $G=(V,E),$ where $E\subseteq V^2$ such that $(v,v)\in E$ for any $v\in V.$ 
The edges of the form $(v,v)$ are called degenerated. 
A digraph $G$ defines a quiver $Q(G)$ such that $Q(G)_0=V,$  $Q(G)_1=E$ and $t(v,v')=v, h(v,v')=v',s(v)=(v,v).$ We also consider a category ${\sf c}(V)$ such that ${\sf ob}({\sf c}(V))=V$ and ${\sf c}(V)(v,v')=\{(v,v')\}.$ So a digraph $G$ defines a marked category $\MMM G=({\sf c}(V),E),$ which defines a chain complex $\Omega G=\Omega(\MMM G)$ and the homology $\PH_*(G)=H_*(\Omega G).$ It is easy to see that this chain complex coincides with the chain complex defined in \cite{grigor2012homologies}. Proposition \ref{prop:EZ_embedded} implies that $\Omega(G\square G')\cong \Omega G\otimes  \Omega G',$ if $\KK$ is a principal ideal domain. 
Proposition    \ref{prop:homotopy_emb} implies that homotopic morphisms of graphs $f\sim g:G\to G'$ induce homotopic morphisms of complexes $\Omega f\sim \Omega g: \Omega G \to \Omega G'$ and $\Psi f \sim \Psi g: \Psi G \to \Psi G'.$
Since, the category ${\sf c}(V)$ is contractible, we see that $\PH_n(G)\cong \PH^{\sf c}_{n+1}(G)$ for $n\geq 1.$ Also note that if $\KK$ is a field, then Corollary \ref{cor:dim} implies that 
\begin{equation}
{\sf dim}(\Omega_{k+l} G) \leq {\sf dim}( \Omega_k G )\cdot {\sf dim}(\Omega_l G).
\end{equation}

\section{A generalization: marked linear categories}
\label{sec:linearly_embedded_quiver}

A ($\KK$-)\emph{linear category} is a category $\AAA$ enriched over the category of $\KK$-modules i.e. $\AAA$ is a category together with a structure of $\KK$-module on the hom-set $\AAA(a,a')$ for any objects $a,a'$ such that the composition is bilinear. A linear functor between linear categories is a functor which is a homomorphism on any hom-set. 

We denote by $\KK^c$ the linear category with one object whose hom-set is equal to $\KK$ and the composition is defined by multiplication. An \emph{augmented linear category} $\AAA$ is a linear category together with a linear functor $\varepsilon: \AAA \to \KK^c.$  We denote by $\AAA^1$ the wide subcategory of $\AAA$ whose morphisms are all morphisms $\alpha$ satisfying $\varepsilon(\alpha)=1.$

Any category $\CCC$ defines an augmented linear category $\KK[\CCC]$. The category $\KK[\CCC]$ has the same objects and its hom-sets are free modules generated by the hom-sets of $\CCC:$ $\KK[\CCC](c,c')=\KK[\CCC(c,c')].$ The composition on $\KK[\CCC]$ is defined as the bilinear extension of the composition on $\CCC.$ The augmentation $\varepsilon:\KK[\CCC]\to \KK^c$ is defined so that for any morphism $\alpha$ of $\CCC$ we have $\varepsilon(\alpha)=1.$ In particular, we have $\CCC\subseteq \KK[\CCC]^1.$ 

A \emph{linear nerve} of an augmented linear category is a simplicial module ${\sf Lnerve}(\AAA)$ such that 
\begin{equation}
{\sf Lnerve}(\AAA)_0=\KK[{\sf ob}(\AAA)]    
\end{equation}
 and 
\begin{equation}
{\sf Lnerve}(\AAA)_n=\bigoplus_{a_0,\dots,a_n\in {\sf ob}(\AAA)} \AAA(a_0,a_1) \otimes \dots \otimes \AAA(a_{n-1},a_n)    
\end{equation}
for $n\geq 1.$ The face maps $d_i:{\sf Lnerve}(\AAA)_n \to {\sf Lnerve}(\AAA)_{n-1}$ and degeneracy maps $s_i:{\sf Lnerve}(\AAA)_{n-1} \to {\sf Lnerve}(\AAA)_{n}$ for $n\geq 2$ are defined by 
\begin{equation}
\begin{split}
d_0(\alpha_1\otimes \dots \otimes \alpha_n) &= \varepsilon(\alpha_1) \alpha_2 \otimes \dots \otimes \alpha_n, \\ 
d_i(\alpha_1\otimes \dots \otimes \alpha_n) &= \alpha_1 \otimes \dots \otimes \alpha_i\cdot \alpha_{i+1} \otimes \dots \otimes  \alpha_n, \hspace{1cm} 1\leq i\leq n-1, \\
d_n(\alpha_1\otimes \dots \otimes \alpha_n) &= \alpha_1 \otimes \dots \otimes \alpha_{n-1} \varepsilon(\alpha_n),\\
s_i(\alpha_1 \otimes \dots \otimes \alpha_{n-1}) &= \alpha_1 \otimes \dots \otimes \dots \otimes 1_{a_i} \otimes \dots \otimes \alpha_{n-1}, \hspace{1cm} 0\leq i\leq n-1, 
\end{split}
\end{equation}
where $\alpha_i\cdot \alpha_{i+1}=\alpha_{i+1} \circ \alpha_{i}.$
For $n=1$ the face and degeneracy maps are defined by 
\begin{equation}
  d_0(\alpha)= \varepsilon(\alpha) h(\alpha), \hspace{1cm} 
d_1(\alpha)= t(\alpha)\varepsilon(\alpha),\hspace{1cm}
s_0(a)={\sf id}_a.  
\end{equation}
Note that all the maps $d_i$ are homomorphisms (thank to the augmentation $\varepsilon$ in the formulas). The fact that it is a simplicial module is straightforward. 

It is easy to see that for a category $\CCC$ we have
\begin{equation}
{\sf Lnerve}(\KK[\CCC]) \cong \KK[\Nrv (\CCC)].    
\end{equation}

A \emph{marked linear category} is a couple 
\begin{equation}
\LLL=(\AAA,M),
\end{equation}
where $\AAA$ is an small augmented linear category and $M\subseteq \AAA^1$ is a subset containing all identity morphisms, such that for any $a,a'\in {\sf ob}(\AAA)$, if we set $M(a,a')=M\cap \AAA(a,a')$,  the map  $\KK[M(a,a')] \to \AAA(a,a')$ is a split monomorphism (if $\KK$ is a field, this is equivalent to linear independence of $M(a,a')$). We denote by $Q_\LLL$ the quiver whose vertices are objects of $\LLL,$ and the set of arrows is $M.$ Therefore the path set $\KK[{\sf P}Q_\LLL]$ is embedded into ${\sf Lnerve}(\AAA)$ and each map $\KK[{\sf P}Q_\LLL]_n \to {\sf Lnerve}(\AAA)_n$ is a split monomorphism. We identify $\KK[{\sf P}Q_\LLL]$ with its image in ${\sf Lnerve}(\AAA).$ So we can consider a path pair of modules
\begin{equation}
    \PPP \LLL = ({\sf Lnerve}(\AAA), \KK[{\sf P}Q_\LLL])
\end{equation}
and set
\begin{equation}
    \Omega \LLL = \Omega(\PPP\LLL), \hspace{1cm} \Psi \LLL =\Psi(\PPP\LLL).
\end{equation}

The tensor product $\AAA\otimes \AAA'$ of two linear categories $\AAA$ and $\AAA'$ is defined so that ${\sf ob}(\AAA\otimes \AAA')={\sf ob}(\AAA)\times {\sf ob}(\AAA')$ and $(\AAA\otimes \AAA')((a,a'),(b,b'))= \AAA(a,b) \otimes \AAA(a',b').$ If $\AAA$ and $\AAA'$ are augmented, then $\AAA\otimes \AAA'$ inherits an obvious augmentation. The box product of two marked linear categories $\LLL=(\AAA,M)$ and $\LLL'=(\AAA',M')$ is defined by the formula
\begin{equation}
\LLL\square \LLL'=(\AAA \otimes \AAA', \{\mu\otimes 1_{a'}\} \cup \{1_a\otimes \mu'\}),
\end{equation}
where $\{\mu\otimes 1_{a'}\}$ stands for $\{\mu\otimes 1_{a'}\mid \mu\in M, a'\in {\sf ob}(\AAA')\}$ and $\{1_a\otimes \mu'\}$ stands for $\{1_a\otimes \mu' \mid a\in {\sf ob}(\AAA), \mu'\in M'\}.$ It is easy to see that $ {\sf Lnerve}(\AAA\otimes \AAA')\cong {\sf Lnerve}(\AAA) \otimes {\sf Lnerve}(\AAA')$ and $Q_{\LLL\square \LLL'} \cong  Q_\LLL \square Q_{\LLL'}.$ Using this and Proposition \ref{prop:paths_in_product}, we obtain 
\begin{equation}\label{eq:ple_box}
\PPP(\LLL\square \LLL')\cong \PPP \LLL \square \PPP\LLL'.    
\end{equation}

A morphism of marked linear categories $f:(\AAA,M)\to (\AAA',M')$ is a morphism of augmented linear categories $f:\AAA \to \AAA'$ such that $f(M)\subseteq M'.$ As usual we consider a model of an interval $I^{\sf lm}=(\KK[\mathpzc{F}(\qq^1)],\qq^1_1),$ a model of a point ${\sf pt}^{\sf lm}=(\KK[\mathpzc{F}(\qq^0)],\qq^0_1),$ define a weak cylinder functor
\begin{equation}
 {\sf cyl}(\LLL)=\LLL \square I^{\sf lm},   
\end{equation}
and define homotopic morphisms via this weak cylinder functor. The equations \eqref{eq:ple_box} and $\PPP(I^{\sf lm})= I^{\sf p}$ imply that homotopic morphisms $f\sim g:\LLL \to \LLL'$ induce homotopic morphisms of complexes 
\begin{equation}\label{eq:homotopy_le}
\Omega f\sim \Omega g: \Omega\LLL \to \Omega\LLL', \hspace{1cm} \Psi f \sim \Psi g: \Psi \LLL \to \Psi \LLL'.   
\end{equation}

\section{\textit{k}-power  homology of quivers}\label{sec:k-power}

In this section we show an approach to homology of quivers developed in \cite{grigor2018path} via marked linear categories. 

\subsection{Definition via marked linear categories}\label{subsection:def_k-power}

Let $\KK$ be a commutative ring and $k\geq 1$ be a natural number such that $k\cdot 1_\KK$ is invertible in $\KK.$ For a set $V$ and a natural number $k\geq 1$ we denote by ${\sf Q}^k_V$ a quiver with the set of vertices $V$ and with exactly $k$ non-degenerated edges from $v$ to $v'$ for any $v,v'\in V.$ The non-degenerated edges from $v$ to $v'$ are denoted by $\alpha_i^{v,v'}$ for $1\leq i\leq k.$   We denote by $\AAA^k_V$ the augmented linear category such that ${\sf ob}(\AAA^k_V)=V$ and $\AAA^k_V(v,v')=\KK[{\sf Q}^k_V(v,v') ].$ The composition is defined by the formula
\begin{equation}\label{eq:comp}
\alpha^{v',v''}_n \circ \alpha^{v,v'}_m = \zeta_{v,v''},
\end{equation}
where $\zeta_{v,v''}$ is defined as the average of all non-degenerate edges
\begin{equation}
\zeta_{v,v''} = \frac{1}{k}  \sum_{i=1}^k \alpha_i^{v,v''},
\end{equation}
for any $1\leq n,m\leq k.$ It is easy to see that 
\begin{equation}
\alpha^{v'',v'''}_n \circ (\alpha^{v',v''}_m \circ  \alpha^{v,v'}_l) = \zeta_{v,v'''} = (\alpha^{v'',v'''}_n \circ \alpha^{v',v''}_m) \circ  \alpha^{v,v'}_l,
\end{equation}
so the composition is associative. The augmentation $\varepsilon: \AAA^k_V\to \KK^c$ is defined so that $\varepsilon(\alpha_i^{v,v'})=1.$ Note that the composition \eqref{eq:comp} is compatible with the augmentation $\varepsilon(\frac{1}{k}  \sum_{i=1}^k \alpha_i^{v,v''})=1$ i.e. $\varepsilon$ is a functor.

The power of a quiver $Q$ is the supremum of cardinalities of $Q(v,v')$ for all $v,v'\in Q_0.$
Now assume that $Q$ is a quiver such that $Q_0=V$ and for any $v,v'\in V$ the cardinality of $Q(v,v')$ is at most $k.$ So there is an embedding
\begin{equation}
i:Q\mono {\sf Q}_V^k,    
\end{equation}
which induces an embedding  $i^\AAA:Q\mono  (\AAA^k_V)^1.$ Therefore we can consider a marked linear category and the corresponding complex
\begin{equation}
    \LLL(i) = (\AAA^k_V,i^\AAA(Q_1)), \hspace{1cm}
    \Omega(i) = \Omega(\LLL(i)).
\end{equation}

By the definition the homology depends on the embedding. However, they don't really depend on the embedding (up to \emph{canonical} isomorphism). Indeed, for any two embeddings $i_1, i_2 : Q\mono  {\sf Q}_V^k$ there exists an automorphism $\varphi:{\sf Q}_V^k \to {\sf Q}_V^k$ such that the diagram is commutative 
\begin{equation}
\begin{tikzcd}
& Q \ar[rd,rightarrowtail,"i_2"] \ar[ld,rightarrowtail,"i_1"'] & \\
{\sf Q}_V^k \ar[rr,"\cong","\varphi"'] && {\sf Q}_V^k
\end{tikzcd}
\end{equation}
Hence $\varphi$ defines an isomorphism 
\begin{equation}
\varphi_* : \Omega(i_1) \overset{\cong}\longrightarrow 
\Omega(i_2).    
\end{equation}
Moreover, this isomorphism does not depend of $\varphi,$ 
because if we have two different automorphisms $\varphi,\varphi'$ such that $\varphi i_1 =i_2 $ 
and $\varphi' i_1 = i_2,$ the maps $\varphi_*,\varphi'_*$ coincide by Proposition \ref{prop:restrictions}
\begin{equation}
    \varphi_* = \varphi'_* : \Omega(i_1) \to \Omega(i_2).
\end{equation} So the isomorphism is uniquely defined by $i_1$ and $i_2.$
This allows as to define $\Omega^{(k)} Q$ as
\begin{equation}
\Omega^{(k)} Q = \Omega(i)     
\end{equation}
for any chosen embedding $i:Q\mono {\sf Q}^k_V.$ The homology of this complex is called $k$-power homology $H^{(k)}_*(Q)=H_*(\Omega^{(k)} Q).$

\subsection{Definition of  Grigor’yan-Muranov-Vershinin-Yau}

\begin{proposition}
Let $\KK$ be a commutative ring and $k\cdot 1_\KK$ is invertible in $\KK.$ Then the chain complex $\Omega^{(k)} Q$ is isomorphic to the chain complex defined in \cite{grigor2018path}.
\end{proposition}
\begin{proof}
Grigor’yan-Muranov-Vershinin-Yau define a graded module $\Lambda (Q) $ for any quiver $Q$ such that $\Lambda_n(Q)$ consists of linear combinations of non-degenerated $n$-paths of $Q$. In other words $\Lambda_n(Q)$ is the $n$th graded component of the path algebra $\KK[Q].$ For the case $Q={\sf Q}^k_V$ they define a structure of chain complex on $\Lambda({\sf Q}^k_V).$
For two edges $\alpha^{v',v}_m,  \alpha^{v'',v'}_n$ they set 
\begin{equation}
[\alpha^{v',v}_m \alpha^{v'',v'}_n  ] =  \sum_{i=1}^k \alpha_i^{v'',v}.
\end{equation}
It is easy to see that is it related to the composition in $\AAA_V^k$ by the formula
\begin{equation}
\alpha^{v'',v'}_n \circ  \alpha^{v',v}_m =  \frac{1}{k} [  \alpha^{v',v}_m\alpha^{v'',v'}_n].  
\end{equation}
They define the maps $d_i:\Lambda_n({\sf Q}_V^k) \to \Lambda_{n-1}({\sf Q}^k_V)$ by the formulas
\begin{equation}
\begin{split}
d_0(a_1 \dots a_n) &= k \cdot  a_2  \dots  a_n, \\ 
d_i(a_1 \dots a_n) &= a_1  {\dots} \cdot [\alpha_i\alpha_{i+1}]\cdot   \dots a_n, \hspace{1cm} 1\leq i\leq n-1, \\
d_n(a_1 \dots a_n) &= k \cdot a_1  \dots  a_{n-1},
\end{split}
\end{equation}
where $a_i\in ({\sf Q}_V^k)_1,$
and  they define the differential $\partial_n = \sum (-1)^{i} d_i.$ It is easy to see that the multiplication by $\frac{1}{k}$ induces an isomorphism of complexes  
\begin{equation}\label{eq:1/k}
\frac{1}{k}\cdot :\Lambda_n ({\sf Q}^k_V) \overset{\cong}\longrightarrow \NN( {\sf Lnerve}(\AAA^k_V)).
\end{equation}
Further, for any embedding $i:Q \mono {\sf Q}^k_V$ they consider $\Lambda(Q)$ as a graded submodule of the chain complex $\Lambda({\sf Q}_V^k)$ and set 
\begin{equation}
\Omega^{(k)} Q= \omega(\Lambda({\sf Q}_V^k),\Lambda(Q)).     
\end{equation}
It is easy to see that the map \eqref{eq:1/k} takes $\Lambda(Q)$ to the image of 
\begin{equation}
\KK[{\sf P}Q]\to \NN( {\sf Lnerve}(\AAA^k_V)).    
\end{equation}
The assertion follows.
\end{proof}

\begin{remark} For the definition of $\Omega^{(k)} Q$ in \cite{grigor2018path} it is not assumed that $k\cdot 1_\KK$ is invertible. However, it is assumed for the proof that the homology $H_*(\Omega^{(k)} Q)$ is homotopy invariant  \cite[Th.5.5]{grigor2018path}. 
\end{remark}

\subsection{Functoriality, strong morphisms} We say that a morphism of quivers $f:Q\to Q'$ is non-degenerate, if any non-degenerate edge maps to a non-degenerate edge $f(Q_1^{\sf N}) \subseteq (Q')_1^{\sf N}.$ Otherwise it is degenerate. For example the projection $\qq^1\epi \qq^0$ is degenerate as well as the homotopy of the identical map with itself $Q\square \qq^1 \epi Q.$ To compare, in \cite{grigor2018path} the definition of quivers without degenerate edges is used, and so only non-degenerate morphisms are considered there. Using the terminology of \cite{grigor2018path}, a non-degenerate morphism $f:Q\to Q'$ is strong, if the  map $Q(v,u)\to Q'(f(v),f(u))$ is injective for any $v,u\in Q_0.$ So strong morphisms between quivers are analogues of faithful functors between categories. 

The wide subcategory of $\Quiv$ with strong morphisms is denoted by ${\sf SQuiv}.$ The full subcategory of ${\sf SQuiv}$ consisted of quivers of power at most $k$ is denoted by
\begin{equation}
{\sf SQuiv}^{k} \subseteq {\sf SQuiv}\subseteq \Quiv. 
\end{equation}

It is easy to check that any strong morphism $f:{\sf Q}^k_V \to {\sf Q}^k_U$ induces a linear functor $\AAA(f):\AAA^k_V\to \AAA^k_U$ (if $f$ is not strong, $\AAA(f)$ is not a functor, because $\AAA(f)(z_{v',v})\ne z_{f(v'),v}$). On the other hand, it is easy to check that any strong morphism $f:Q\to Q'$ of quivers of power at most $k$ can be embedded into a diagram 
\begin{equation}
\begin{tikzcd}
Q\ar[rr,"f"] \ar[d,rightarrowtail,"i"] && Q' \ar[d,rightarrowtail,"i'"] \\
{\sf Q}^k_V \ar[rr,"\tilde f"] && {\sf Q}^k_{V'},
\end{tikzcd}
\end{equation}
where $\tilde f$ is strong, $i,i'$ are embeddings and $V=Q_0, V'=Q'_0.$ This defines a morphism of marked linear categories 
\begin{equation} 
(\AAA(\tilde f),f): \LLL(i) \longrightarrow \LLL(i').
\end{equation}

By Proposition \ref{prop:restrictions}, the induced morphism $\Omega^{(k)} Q \to \Omega^{(k)} Q'$ does not depend on the choice of $\tilde f$ and we denote it by 
\begin{equation}
\Omega^{(k)} f : \Omega^{(k)} Q \longrightarrow \Omega^{(k)} Q'.
\end{equation}
This defines a functor 
\begin{equation}
\Omega^{(k)} : {\sf SQuiv}^k \longrightarrow {\sf Ch},    
\end{equation}
where ${\sf Ch}$ is the category of chain complexes.

\subsection{Homotopy invariance of \textit{k}-power homology} 

We can define a weak cylinder functor for ${\sf SQuiv}$ as follows
\begin{equation}
{\sf cyl} = - \square \qq^1 : {\sf SQuiv}^k \longrightarrow {\sf SQuiv}^k. \end{equation}
We say that two strong morphisms of quivers 
$f,g:Q\to Q'$ are strongly
homotopic if they are homotopic with respect to this weak cylinder functor. Note that in this definition we assume that the homotopy $h:{\sf cyl}(Q)\to Q'$ is also a strong morphism.

\begin{proposition}[{cf. \cite[Th.5.5]{grigor2018path}}] Let $\KK$ be a commutative ring and $k\geq 1$ be a natural number such that $k\cdot 1_{\KK}$ is invertible in $\KK$. Strongly homotopic strong morphisms $f\sim g:Q\to Q'$ of quivers of power at most $k$ induce homotopic morphisms of complexes
\begin{equation}
    \Omega^{(k)} f \sim \Omega^{(k)} g : \Omega^{(k)} Q \longrightarrow \Omega^{(k)} Q'.
\end{equation}
\end{proposition}
\begin{proof} We can assume that $f$ and $g$ are one-step homotopic. Let $h:Q\square \qq^1 \to Q'$ be a homotopy from $f$ to $g.$ 
Chose embeddings $i:Q\mono {\sf Q}^k_V$ and $i':Q'\mono {\sf Q}^k_{V'},$ where $V=Q_0$ and $V'=Q_0',$ and chose some  strong morphisms $\tilde f, \tilde g:{\sf Q}^k_V \to {\sf Q}^k_{V'}$ that extend $f$ and $g.$ They define linear functors $\AAA(f),\AAA(g):\AAA^k_V \to \AAA^k_{V'}.$ Consider a functor $H:\AAA^k_V \otimes \KK[\mathpzc{F}(\qq^1)]\to \AAA^k_{V'}$ which is defined on objects such that $H(v,0)=f(v)$ and $H(v,1)=g(v)$ and which is defined on morphisms such that
\begin{itemize}
\item for any edge $\alpha\in {\sf Q}^k_V(v,u)$ we have $H(\alpha \otimes {\sf id}_0)= \tilde f(\alpha),$ $H(\alpha\otimes {\sf id}_1)=\tilde g(\alpha);$
\item $H({\sf id}_v\otimes (0,1)) = h(({\sf id}_v,(0,1)));$

\item for a non-degenerate edge $\alpha\in {\sf Q}^k_V(v,u)$ we set $H(\alpha\otimes (0,1))=z_{g(u),f(v)}.$  
\end{itemize} 
It is easy to check that this functor is well-defined and it defies a morphism of marked linear categories
\begin{equation}
H: \LLL(i)\square I^{\sf lm}\longrightarrow \LLL(i')    
\end{equation}
that shows that the maps $(\AAA(\tilde f),f),(\AAA(\tilde g),g) : \LLL(i) \to \LLL(i') $ are homotopic in the sense of marked linear categories. Then the assertion follows from  \eqref{eq:homotopy_le}.
\end{proof}

\section{Square-commutative homology of quivers}
\label{sec:square-commutative}

 In this section we will present two equivalent (Proposition \ref{prop:C_equiv_Z}) definitions for a square-commutative homology of quivers and prove their basic properties. The corresponding theory coincides with path homology for graphs without non-degenerate directed triangles and double edges (Theorem \ref{th:path=sc}).
 
\begin{definition}\label{def: square-free-Z}
For a quiver $Q$ we denote by $\ZZZ(Q)$ a category such that ${\sf ob}(\ZZZ(Q))=Q_0$ and  $\ZZZ(Q)(v,u)=Q(v,u)\cup\{z_{v,u}\},$ where $z_{v,u}$ is a new formal arrow. The composition is defined so that for any \emph{non-degenerated} edges $\alpha:v\to v'$ and $\beta:v'\to v''$ we set $\beta \circ \alpha=z_{v,v''}.$ Degenerated edges $s(v)=1_v$ are the identity morphisms in this category. 
Then we define the complex $\Omega^{\sf sc}Q$ as the complex associated with the marked category $(\ZZZ(Q),Q_1)$
\begin{equation}
\Omega^{\sf sc}Q = \Omega(\ZZZ(Q),Q_1).
\end{equation}
\end{definition}

Note that the set $I=\{z_{v,u}\mid v,u\in Q_0\}$ is an ideal of $\ZZZ(Q).$ Hence, Proposition \ref{prop:ideal} implies that \begin{equation}\label{eq:omega_sc_d_i}
\Omega^{\sf sc}_nQ=\{x\in (\NN Q_1)_n \mid \tilde d_i(x)=0,\  \text{for}\ 1\leq i\leq n-1 \}.
\end{equation}

This definition of square-commutative homology is not always convenient for proving its properties. So, we define another category $\CCC(Q)$ such that $Q\subseteq \CCC(Q)$ and prove that the corresponding complexes coincide. 
\begin{definition}\label{def: square-free-C}
By $\mathpzc{F}(Q)$ we denote the path category (free category) of $Q.$ Consider its quotient $\CCC(Q)=\mathpzc{F}(Q)/\sim,$ where $\sim$ is the minimal congruence relation such that $\alpha \beta \sim \beta' \alpha'$ for any \emph{non-degenerate} arrows $\alpha,\beta,\alpha',\beta'\in Q_1^{\sf N}$ such that $t(\beta)=t(\alpha'), h(\alpha)=h(\beta'), h(\beta)=t(\alpha), h(\alpha')=t(\beta').$
\begin{equation}
\CCC(Q)=\mathpzc{F}(Q)/\{\alpha \beta = \beta' \alpha'\}, \hspace{1cm}
\begin{tikzcd}
\bullet \ar[r,"\beta"] \ar[d,"\alpha'"'] & \bullet \ar[d,"\alpha"] \\ 
\bullet \ar[r,"\beta'"'] & \bullet
\end{tikzcd}
\end{equation}
\end{definition}

Such squares in $Q$ will be called \emph{non-degenerate directed squares}. So, roughly speaking, in $\CCC(Q)$ we make all non-degenerate squares commutative. In the category $\mathpzc{F}(Q)$ there is a notion of the length of a morphism: the length of the path in the original quiver. We assume that the length of identity morphisms is zero. Since we take a quotient so that only morphisms of the same length are equivalent, the length of a morphism is well defined in $\CCC(Q).$ 

We denote by ${\sf NQuiv}$ the category of quivers and non-degenerate morphisms. The constructions $\ZZZ(Q)$ and $\CCC(Q)$ are natural with respect to non-degenerate morphisms and define functors to the category of small categories 
\begin{equation}
\ZZZ,\CCC : {\sf NQuiv} \longrightarrow {\sf Cat}. 
\end{equation}
In particular, we obtain a functor
\begin{equation}
\Omega^{\sf sc} : {\sf NQuiv}\longrightarrow {\sf Ch}.
\end{equation}

\begin{remark}
Note that the constructions $\ZZZ, \CCC$ are not natural on the whole category of quivers $\Quiv$ with not necessary  non-degenerate morphisms. For example, they are not well-defined for the projection $\qq^2\to \qq^0.$
\end{remark}

For any quiver $Q$ we consider a functor
\begin{equation}
\tau : \CCC(Q) \to \ZZZ(Q) 
\end{equation}
which is identical on $Q$. It sends a path from $v$ to $u$ of length at least two to $z_{v,u}.$ This functor is natural on $Q,$ so we can say that it is a natural transformation $\tau:\CCC\to \ZZZ.$  This functor induces a morphism of marked categories
\begin{equation}\label{eq:tau}
\tau : (\CCC(Q),Q_1) \longrightarrow (\ZZZ(Q),Q_1).
\end{equation}
\begin{proposition}\label{prop:C_equiv_Z}
For any commutative ring $\KK$ the morphism \eqref{eq:tau} induces a natural isomorphism 
\begin{equation}
\Omega(\CCC(Q),Q_1) \cong \Omega^{\sf sc}Q.
\end{equation}
\end{proposition}
\begin{proof} We denote by $Q$ the quiver $Q$ considered as a subquiver of $ \CCC(Q)$ and we denote by $Q'$ the quiver $Q$ considered as a subquiver of $\ZZZ(Q).$ Then $Q_1^2$ consists of all morphisms of length $\leq 2.$ So there are two types of morphisms in $Q_1^2:$  (1) morphisms from $Q_1$; (2) morphisms of length $2.$ Similarly there are two types of morphisms in $(Q'_1)^2$:  (1) morphisms from $Q'_1;$ (2) morphisms $z_{v,u}$ for such couples of $(v,u)$ that there exists a path of length $2$ from $v$ to $u$ in $Q'_1$. 
It is easy to see that $\tau$ induces a bijection between all  isomorphisms $Q_1\cong Q_1$ and $Q_1^2\cong (Q'_1)^2.$ 
Then the assertion follows from Proposition \ref{prop:isomorphism-lemma}.
\end{proof}

\begin{remark}\label{remark:Z'}
Note that the functor $\tau : \CCC(Q)\to \ZZZ(Q)$ generally is not surjective on morphisms. For example, if $Q=\qq^1,$ then $\CCC(\qq^1)$ has only one non-identical morphism and $\ZZZ(\qq^1)$ has two non-identical morphisms. The image $\ZZZ'(Q)$ of $\CCC(Q)$ in $\ZZZ(Q)$ can be described as follows: 
\begin{equation}
\ZZZ'(Q)(v,u) = 
\begin{cases} 
Q(v,u)\cup \{z_{v,u}\}, & \text{there is a path of length $\geq 2$ from $u$ to $v$} \\
Q(v,u), & \text{else}.
\end{cases}
\end{equation}
Since $Q\subseteq \ZZZ'(Q)\subseteq \ZZZ(Q)$ by Corollary \ref{cor:embedding_categories} we obtain that 
\begin{equation}
\Omega^{\sf sc} Q \cong \Omega(\ZZZ'(Q),Q_1).
\end{equation}
\end{remark}

\subsection{Vanishing of square commutative homology}

\begin{proposition}\label{prop:without_squares}
If $Q$ has no non-degenerated directed squares, then
\begin{equation}
\Omega_n^{\sf sc}Q=0,  \hspace{1cm}\text{for} \ n\geq 2,
\end{equation}
and $\Omega_n^{\sf sc}Q=\KK^{Q_n}$ for $n=0,1.$ Moreover, $H^{\sf sc}_n(Q)$ is isomorphic to the homology of the quiver treated as $1$-dimensional space.  
\end{proposition}
\begin{proof} It follows from Proposition \ref{prop:free} and the equation $\CCC(Q)=\mathpzc{F}(Q)$ for this kind of quivers.    
\end{proof}

\subsection{Box product and square-commutative homology}

\begin{lemma}\label{lemma:product_of_free}
For any two quivers $Q,Q'$ there is an isomorphism 
\begin{equation}
\mathpzc{F}(Q)\times \mathpzc{F}(Q') \cong \mathpzc{F}(Q\square Q')/\sim,  
\end{equation}
which is an identity on $Q\square Q',$ where $\sim$ is the minimal congruence relation such that $(\alpha,1_{u'}) (1_v,\beta) \sim (1_{u},\beta) ( \alpha ,1_{v'})$ for any $\alpha\in Q(v,u)$ and $\beta\in Q(v',u').$
\end{lemma}
\begin{proof}
The inclusion $Q\square Q' \hookrightarrow \mathpzc{F}(Q)\times \mathpzc{F}(Q')$ defines a full functor $\mathpzc{F}(Q\square Q') \to \mathpzc{F}(Q)\times \mathpzc{F}(Q'),$ which is an identity on objects. This functor sends both $(\alpha,{\sf id}_{u'}) ({\sf id}_v,\beta)$ and $({\sf id}_{u},\beta) ( \alpha ,{\sf id}_{v'})$ to $(\alpha,\beta).$ Hence, we have a full functor $\mathpzc{F}(Q\square Q')/\sim \  \to \mathpzc{F}(Q)\times \mathpzc{F}(Q').$
The relation $(\alpha,{\sf id}_{u'}) ({\sf id}_v,\beta) \sim ({\sf id}_{u},\beta) ( \alpha ,{\sf id}_{v'})$ allows to present any morphism of $\mathpzc{F}(Q\square Q')/\sim$ in the form $(\alpha_1,1)(\alpha_2,1)\dots (\alpha_n,1) (1,\beta_1)(1,\beta_2) \dots (1,\beta_m).$ It follows that the functor is injective on hom-sets. Hence the functor is an isomorphism. 
\end{proof}

\begin{proposition}\label{prop:C_square}
There is a functor
\begin{equation}
\CCC(Q)\times \CCC(Q') \longrightarrow    \CCC(Q\square Q')
\end{equation}
which is identical on $Q\square Q'.$ Moreover, if $Q=G$ and $Q=G'$ are directed graphs, then this is an isomorphism
\begin{equation}
 \CCC(G) \times \CCC(G')\cong \CCC(G\square G').    
\end{equation}
\end{proposition}
\begin{proof}
By the definition, we have a functor $\mathpzc{F}(Q\square Q') \to \CCC(Q\square Q').$ The compositions $(\alpha,1_{u'}) (1_v,\beta) $ and $ (1_{u},\beta) ( \alpha ,1_{v'})$ are mapped to to the same morphism. By Lemma \ref{lemma:product_of_free} we obtain a functor $\mathpzc{F}(Q)\times \mathpzc{F}(Q') \to \CCC(Q\square Q')$ which obviously induces a functor $\CCC(Q)\times \CCC(Q') \to \CCC(Q\square Q').$ 

Now assume that $Q=G$ and $Q'=G'$ are digraphs.  In this case there are three types of non-degenerate directed squares in $G\square G':$ (1) a non-degenerate square in $G$ times an object of $G'$; (2) an object of $G$ times a non-degenerate square in $G$; (3) squares of the type 
\begin{equation}
\begin{tikzcd}
(v,v') \ar[r,"{(\alpha,1)}"] \ar[d,"{(1,\beta)}'"'] & (u,v') \ar[d,"{(1,\beta)}"] \\
(v,u') \ar[r,"{(\alpha,1)}"'] & (u,u').
\end{tikzcd}    
\end{equation}
Here we use that there is at most one arrow  $v\to u$ and at most  one arrow $v'\to u'.$ All these relations, corresponding these three types of squares, are satisfied in $\CCC(G)\times \CCC(G').$ Hence, there is a functor $\CCC(G\square G')\to \CCC(G)\times \CCC(G'),$ which is identical on $G\square G'.$ Since $G\square G'$ is a generating set of morphisms in both categories, and both functors $\CCC(G\square G')\to \CCC(G)\times \CCC(G')$ and $\CCC(G)\times \CCC(G')\to \CCC(G\square G')$ are identical on $G\square G',$ the compositions $\CCC(G)\times \CCC(G')\to \CCC(G\square G')\to \CCC(G)\times \CCC(G')$ and $\CCC(G\square G')\to \CCC(G)\times \CCC(G') \to \CCC(G)\times \CCC(G')$ are also identical, the functors are isomorphisms.
\end{proof}

\begin{proposition}\label{prop:EZ-sc}
If $\KK$ is a principal ideal domain and $G,G'$ are directed graphs, then there is an isomorphism of complexes 
\begin{equation}
\Omega^{\sf sc}(G\square G') \cong   \Omega^{\sf sc} G \otimes \Omega^{\sf sc}G'.  
\end{equation}
 Moreover, there is a short exact sequence 
\begin{equation} 0 \to
\bigoplus_{i+j=n} H^{\sf sc}_i(G) \otimes H^{\sf sc}_j(G') \to    H^{\sf sc}_n( G \square G' ) \to \bigoplus_{i+j=n-1} {\sf Tor}^\KK_1( H^{\sf sc}_i(G),H^{\sf sc}_j(G')) \to 0.
\end{equation}
\end{proposition}
\begin{proof}
It follows from Proposition \ref{prop:C_square} and Proposition   \ref{prop:EZ_embedded}.
\end{proof}

\begin{remark}
   Proposition \ref{prop:EZ-sc} can't be generalised to the case of all quivers (see Example \ref{example:two_loops}). 
\end{remark}

\subsection{Homotopy invariance of square-commutative homology} 

As usual we define a weak cylinder functor ${\sf cyl} = - \square \qq^1:{\sf NQuiv}\to {\sf NQuiv}$ and define non-degenerately homotopic non-degenerate morphisms of quivers via this weak cylinder functor. 

\begin{proposition}\label{prop:homotopy-sc}
For any commutative ring $\KK$ two non-degenerately homotopic non-degenerate morphisms of quivers $f\sim g:Q\to Q'$ induce homotopic morphisms of chain complexes 
\begin{equation}
\Omega^{\sf sc} f \sim \Omega^{\sf sc} g : \Omega^{\sf sc} Q \longrightarrow \Omega^{\sf sc} Q'.     
\end{equation}
\end{proposition}
\begin{proof} We denote by $F:{\sf NQuiv}\to {\sf MarkCat}$ the functor given by $F(Q)=(\CCC(Q),Q_1).$ Using that $\CCC(\qq^1)=\mathpzc{F}(\qq^1)$ we obtain that ${\sf cyl}( F(Q)) = (\CCC(Q)\times \CCC(\qq^1),(Q\square\qq^1)_1)$ and $F({\sf cyl}(Q))=(\CCC(Q\square \qq^1),(Q\square \qq^1)_1).$
Proposition \ref{prop:C_square} implies that there is a natural transformation 
${\sf cyl}\: F \to F\: {\sf cyl}.$
Then the assertion follows from Proposition \ref{prop:cyl} and Proposition \ref{prop:homotopy_emb}.
\end{proof}

\subsection{Comparison of square-commutative and path homology}

For a set $V$ we denote by ${\sf c}(V)$ the category whose set of objects is $V$ and each hom-set is one-element ${\sf c}(V)(v,u)=\{(v,u)\}$.  Recall that the path homology of a digraph $G$ are defined as homology of the marked category $({\sf c}(V),G),$ where $V=G_0$ is the vertex set of $G.$ There is a unique functor $\ZZZ(G)\to {\sf c}(V)$ which is identical on objects. It sends a morphism $\alpha:v\to u$ to the morphism $(v,u):v\to u.$ It is easy to see that this defines a morphism of marked categories 
$(\ZZZ(G),G) \longrightarrow ({\sf c}(V),G),$ 
which defines a morphism of chain complexes 
\begin{equation}\label{eq:sctopath}
    \Omega^{\sf sc} G \longrightarrow \Omega G.
\end{equation}

A non-degenerated directed triangle of a quiver $Q$ is a triple of non-degenerated arrows $\alpha,\beta,\gamma$ such that $t(\alpha)=t(\gamma), h(\alpha)=t(\beta), h(\beta)=h(\gamma)$
\begin{equation}
\begin{tikzcd}
& \bullet \ar[rd,"\beta"] & \\
\bullet \ar[ru,"\alpha"] \ar[rr,"\gamma"] && \bullet
\end{tikzcd}    
\end{equation}
A double edge in a digraph $G=(V,E)$ is a is a two element set of vertices $\{u,v\}$ such that $(u,v),(v,u)\in E.$

\begin{theorem}\label{th:path=sc}
Let $G$ be a digraph without non-degenerate directed triangles and double edges. Then the morphism \eqref{eq:sctopath} is an isomorphism 
\begin{equation}
\Omega^{\sf sc} G \cong \Omega G.
\end{equation}
\end{theorem}
\begin{proof}
We denote by $M$ the arrows of the digraph $G$ considered as morphisms of $\ZZZ(G)$ and denote by $M'$ the arrows of digraph $G$ 
considered as morphisms of ${\sf c}(V).$ The set $M^2$ is a disjoint union of two sets of morphisms: (1) arrows of $M$; (2) morphisms $z_{v,u}$ for such couples that there is a path of length $2$ from $v$ to $u.$ The set $(M')^2$ is also a \emph{disjoint} union of two sets: (1) edges $(v,u)\in E$ (2) couples $(v,u)$ such that there is a path of length $2$ from $v$ to $u.$ The fact that the last two sets are disjoint follows from the fact that $G$ has no non-degenerate directed triangles and double edges. Then it is easy to see that the morphism $(\ZZZ(G),M)\to ({\sf c}(V),M')$ induces isomorphisms $M\cong M'$ and $M^2\cong (M')^2.$ Then the assertion follows from Proposition \ref{prop:isomorphism-lemma}. 
\end{proof}

\subsection{Simplicial complexes and square-commutative homology} 

Following \cite{grigor2014graphs} we can associate two graphs $G(S)$ and $G'(S)$  with a simplicial complex $S$, such that 
\begin{equation}
G(S) \subseteq G'(S).    
\end{equation}
The vertices of the graph $G'(S)$ are simplices of $S.$ For two simplices $\sigma,\tau\in S$ there is an arrow $\sigma\to \tau$ in $G'(S)$ if and only if $\tau\subseteq \sigma.$ The graph $G(S)$ is a subgraph of $G'(S)$ with the same set of vertices. The arrow $\sigma \to \tau$ is in $G(S)$ if and only if $\tau\subseteq \sigma$ and ${\sf dim}(\sigma) = {\sf dim}(\tau)+1.$  It is proved in  \cite{grigor2014graphs} that 
\begin{equation}\label{eq:cong_G(S)}
    H_*(S)\cong \PH_*(G(S)) \cong \PH_*(G'(S)).
\end{equation}

\begin{theorem}\label{th:simpl_complex_sc} For any simplicial complex $S$ there is an isomorphism 
\begin{equation}
    H_*(S) \cong H^{\sf sc}_*(G(S)).
\end{equation}
\end{theorem}
\begin{proof}
It follows from Theorem \ref{th:path=sc}, the fact that there are no non-degenerate directed triangles and double edges in $G(S),$ and the isomorphism \eqref{eq:cong_G(S)}.
\end{proof}

\subsection{Comparison of square-commutative and \textit{k}-power homology} 

Let $k\geq 1$ be an integer such that $k\cdot 1_\KK$ is invertible in $\KK$ and let $Q$ be a quiver of power at most $k.$ Consider an embedding $i:Q\mono {\sf Q}^k_V$ and the corresponding embedding $i^\AAA :Q \mono \AAA^k_V$ (see Subsection  \ref{subsection:def_k-power}). Consider a functor $i^\ZZZ:\ZZZ(Q) \to \AAA^k_V$ which sends $\alpha\in Q_1$ to $i(\alpha)$ and $z_{v,v'}$ to $\zeta_{v,v'}=\frac{1}{k} \sum_{i=1}^k \alpha_i^{v,v'}.$ It is easy to see that this is a functor which induces a morphism of marked linear categories
\begin{equation}
i^\AAA:(\KK[\ZZZ(Q)],Q) \longrightarrow (\AAA^k_V,i^\AAA(Q)).
\end{equation}
This induces a morphism of chain complexes 
\begin{equation}\label{eq:sc-k}
\Omega^{\sf sc} Q \longrightarrow \Omega^{(k)} Q.    
\end{equation}

\begin{proposition} Let $\KK$ be a commutative ring such that $k\cdot 1_\KK$ is invertible in $\KK$ and let $Q$ be a quiver of power at most $k-1.$ Then the morphism \eqref{eq:sc-k} is an isomorphism
\begin{equation}
\Omega^{\sf sc} Q \cong \Omega^{(k)} Q. 
\end{equation}
\end{proposition}
\begin{proof}
Since the power of $Q$ is less then $k,$ we have $i(Q)(v,u) \subsetneqq {\sf Q}^k_V(v,u)$ for each $v,u\in V.$ It follows that the set $i^\AAA(Q)(v,u) \cup \{ \zeta_{v,u} \} $ is linearly independent in $\AAA^k_V(v,u)$. Therefore $i^\ZZZ: \KK[\ZZZ(Q)] \to \AAA^k_V$ is an embedding.  Then the assertion follows from Remark \ref{rem:embedding}.
\end{proof}

\begin{corollary}[{cf. \cite[Th.4.4]{grigor2018path}}]\label{cor:kl-iso}
Let $k,$ $l$ be positive integers such that $k\cdot 1_\KK$ and $l\cdot 1_\KK$ are invertible in $\KK,$ and let $Q$ be a quiver of power strictly less that $k$ and $l.$ Then there are isomorphisms 
\begin{equation}
    \Omega^{(k)} Q \cong \Omega^{(l)} Q, \hspace{1cm} H^{(k)}_*(Q) \cong H^{(l)}_*(Q).
\end{equation}
\end{corollary}

\begin{remark} Let $k,l$ are two distinct positive integers and $Q$ is a quiver whose power is strictly less than $k$ and $l.$  On one hand, in 
\cite[Remark 4.5]{grigor2018path} it is stated that $H_*^{(k)}(Q)$ and $H_*^{(l)}(Q)$ are not necessarily isomorphic in this case. They state that there is an isomorphism of modules $\Omega^{(k)}_n Q \cong \Omega^{(l)}_n Q,$ which is not compatible with the differential. On the other hand, Corollary \ref{cor:kl-iso} states that the complexes $\Omega^{(k)}Q$ and $\Omega^{(l)}Q$ are isomorphic as well as their homology $H_*^{(k)}(Q)$ and $H_*^{(l)}(Q)$. This is because in \cite{grigor2018path} authors use another approach to the definition that allows to define $k$-power homology without the assumption that $k\cdot 1_\KK$ is invertible. But our approach to the definition (that uses marked linear categories) allows us to define the $k$-power homology only assuming that $k\cdot 1_\KK$ is invertible. So, the modules  $H_*^{(k)}(Q)$ and $H_*^{(l)}(Q)$ can be non-isomorphic only if at least one of the integers $k$ and $l$ is not invertible in $\KK$. The authors show that in the case of $\KK=\ZZ$ the groups $H^{(k)}_0(Q,\ZZ)$ and $H^{(l)}_0(Q,\ZZ)$ are non-isomorphic \cite[Prop 4.3]{grigor2018path}.
\end{remark} 

\subsection{Examples}

\begin{example}[{cf. \cite[Example 4.6]{grigor2018path}}]
Let $G$ be a digraph with three vertices and there non-degenerate arrows 
\begin{equation}\label{eq:graph:triangle}
\begin{tikzcd}
& \bullet \ar[rd]& \\
\bullet \ar[rr] \ar[ru] & & \bullet
\end{tikzcd}
\end{equation}
By Proposition \ref{prop:without_squares} we have $H_0^{\sf sc}(G)=H_1^{\sf sc}(G)\cong \KK$ and $H_n^{\sf sc}(G)=0$ for $n\geq 2.$  On the other hand this graph is contractible in the sense of GLMY-theory  \cite{grigor2014homotopy}, and hence 
\begin{equation}
H_1^{\sf sc}(G)\not \cong \PH_1(G) =0.    
\end{equation}
Note that we also have a variant of homotopy invariance theorem for square-commutative homology (Proposition \ref{prop:homotopy-sc}) but it works only for \emph{non-degenerately} homotopic \emph{non-degenerated} morphisms. This weak version of homotopy invariance theorem does not allow to prove that graphs contractible in the sense of GLMY-theory have trivial square-commutative homology.
\end{example}

\begin{example}
Take the graph $G$ from the previous example \eqref{eq:graph:triangle} and assume that $\KK$ is a principal ideal domain. Then by Proposition \ref{prop:EZ-sc} we have 
\begin{equation}
H_i^{\sf sc}(G^{\square n}) = \KK^{\binom{n}{i}},
\end{equation}
where $G^{\square n} = G \square \dots \square G.$
\end{example}

\begin{example}
A large source of examples is Theorem  \ref{th:simpl_complex_sc}. For instance, we can take a triangulation $S$ of the Klein bottle and obtain a digraph $G(S)$ such that $H^{\sf sc}_1(G(S),\ZZ)=\ZZ\oplus \ZZ/2.$ This example shows that there can be torsion in $H^{\sf sc}_*(G,\ZZ).$
\end{example}

\begin{example}
Consider the quiver $Q:$
\begin{equation}
\begin{tikzcd}
\bullet \ar[r,bend left=20,"\alpha_1"] \ar[r,bend right=20,"\alpha_2"']  & \bullet \ar[r,bend left=20,"\beta_1"] \ar[r,bend right=20,"\beta_2"']  & \bullet 
\end{tikzcd}    
\end{equation}
and the category $\ZZZ'(Q)$ (see Remark \ref{remark:Z'}). There are four non-degenerate composable $2$-sequences $(\alpha_i,\beta_j)$ for $i,j\in \{1,2\}$ that form a basis of $\NN_2 \ZZZ'(Q)$ and $\NN_n \ZZZ'(Q)=0$ for $n\geq 1.$ By \eqref{eq:omega_sc_d_i} we have $\Omega^{\sf sc}_2(Q)={\sf Ker}(\tilde d_1: \overline{\KK[{\sf P} Q]}_2 \to \NN\ZZZ'(Q)_1)$ and $\Omega_n^{\sf sc}(Q)=\overline{\KK[{\sf P}Q]}_n$ for $n=0,1.$  For an element $x=a_{11}(\alpha_1,\beta_1)+a_{12}(\alpha_1,\beta_2)+a_{21}(\alpha_2,\beta_1)+a_{22}(\alpha_2,\beta_2) \in (\NN \ZZZ'(Q))_2$ the condition $d_1(x)=0$ is equivalent to $a_{11}+a_{12}+a_{21}+a_{22}=0.$  Therefore we have a 3-element basis of $\Omega^{\sf sc}_2(Q)$ given by $(\alpha_1,\beta_1) - (\alpha_1,\beta_2),$ $(\alpha_1,\beta_2) - (\alpha_2,\beta_1 ), $ and $(\alpha_2,\beta_1) - (\alpha_2,\beta_2).$
\begin{equation}
    \Omega^{\sf sc}Q: \hspace{1cm} 0\to \KK^3 \to \KK^4 \to \KK^3.
\end{equation}
Simple computation shows that that homology of this complex is 
\begin{equation}
H_0^{\sf sc}(Q)=\KK, \hspace{5mm} H_1^{\sf sc}(Q)=0, \hspace{5mm}  H_2^{\sf sc}(Q)=\KK
\end{equation}
and the non-trivial $2$-cycle is given by $ (\alpha_1,\beta_1) - (\alpha_1,\beta_2) - (\alpha_2,\beta_1) + (\alpha_2,\beta_2).$
\end{example}

\begin{example}
\label{example:two_loops}
Consider the quiver with one loop and the quiver with two loops.
\begin{equation}
Q^{(1)}: \ \ \  
\begin{tikzcd}
\bullet \ar[loop, in=135,out=225,looseness=10,"\alpha"]
\end{tikzcd}    
\hspace{1cm}
Q^{(2)}: \ \ \  
\begin{tikzcd}
\bullet 
\ar[loop, in=135,out=225,looseness=10,"\alpha"] 
\ar[loop, in=45,out=-45,looseness=10,"\beta"']
\end{tikzcd}    
\end{equation}
Note that 
$Q^{(2)} = Q^{(1)} \square Q^{(1)}.$
It is easy to compute that  $\Omega^{\sf sc}(Q^{(1)})$ has the following form 
\begin{equation}
 \Omega^{\sf sc}(Q^{(1)}): \hspace{1cm} 0 \to \KK \overset{0}\to \KK   
\end{equation}
and the square commutative homology are equal to homology of the circle. 
\begin{equation}
H^{\sf sc}_*(Q^{(1)}) = H_*(S^1). 
\end{equation}
In particular $\Omega^{\sf sc}_n(Q^{(1)})=0$ for $n\geq 2.$ On the other hand, we claim that
\begin{equation}
 \Omega^{\sf sc}_n(Q^{(2)})\ne 0, \hspace{1cm} n\geq 0.   
\end{equation}
Let us prove it. We denote by $W_n$ the set of all sequences $s=(\gamma_1,\dots,\gamma_n),$ where $\gamma_i\in \{ \alpha,\beta \}.$ We set ${\sf sgn}(s)=(-1)^k,$ where $k$ the number of $\beta$ in the sequence. Then it is easy to check that $ \sum_{s\in W_n} {\sf sgn}(s) \langle s \rangle \in \Omega_n^{\sf sc}(Q^{(2)}).$  
This follows that 
\begin{equation}
    \Omega^{\sf sc}(Q^{(1)} \square Q^{(1)}) \not\cong \Omega^{\sf sc}(Q^{(1)}) \otimes \Omega^{\sf sc}(Q^{(1)}). 
\end{equation}
Hence, Proposition \ref{prop:EZ-sc} can't be generalised to all quivers. 
\end{example}

\section{Marked groups}\label{sec:homology_of_subsets}

\subsection{Definition and basic properties}

A subset of a group $M\subseteq G$ is called pointed, if it contains the unit. 
A marked group is a couple $(G,M),$ where $G$ is a group and $M$ is a pointed subset of $G.$ A group can be treated as a category with one object, so a marked group is a particular case of a marked category. Therefore a marked group $(G,M)$ 
defines a chain complex $\Omega(G,M)$ and its homology 
\begin{equation}
    \PH_*(G,M) := H_*(\Omega(G,M)) 
\end{equation}
is called path homology of the marked group. If we want to specify the ring $\KK$ we use the notation $\PH_*(G,M,\KK).$ Note that if $M$ is a subgroup of $G,$ then $\PH_*(G,M)=H_*(M)$ is just the ordinary group homology. 

A morphism marked groups $f:(G,M)\to (G',M')$ is a homomorphism $f:G\to G'$ such that $f(M)\subseteq M'.$ Such a morphism defines a morphism of complexes $\Omega(G,M)\to \Omega(G',M')$ and a morphism of modules $\PH_*(G,M)\to \PH_*(G',M').$

\ 

\newpage

\subsubsection{Conjugated homomorphisms} A ``homotopy invariance'' for marked groups has the following form. Let $f,g:(G,M)\to (G',M')$ be two morphisms of marked groups such that there exists an element $m\in M'$ such that the homomorphisms are conjugated by this element i.e. 
\begin{equation}
f(y)=m^{-1}g(y)m \hspace{1cm}    \text{for any } y\in G.
\end{equation}
Then $m$ can be treated as a natural transformation from $f$ to $g$ if we treat them as functors between categories $G$ and $G'.$  Then Proposition \ref{prop:nat_tranf} and Proposition \ref{prop:homotopy_emb} imply that the induced morphisms of chain complexes are homotopic 
\begin{equation}
\Omega f\sim \Omega g : \Omega(G,M) \longrightarrow \Omega(G',M').    
\end{equation}
In particular, the homomorphisms  $\PH_*(f)=\PH_*(g): \PH_*(G,M)\to \PH_*(G',M')$ are equal. 
 
\subsubsection{Isomorphism lemma for marked groups}\label{par:isomorphism} For a subset $M\subseteq G$ we set $M^2=\{mm' \mid m,m'\in M\}.$ If $f:(G,M)\to (G',M')$ is a morphism of marked groups such that $f$ induces bijections $M\cong M'$ and $M^2 \cong (M')^2,$ then Proposition \ref{prop:isomorphism-lemma} implies that $f$ induces an isomorphism 
\begin{equation}\label{eq:is:X^2}
\PH_*(G,M) \cong \PH_*(G',M').    
\end{equation}
In particular, if $G\subseteq G'$ then 
$
\PH_*(G,M)\cong \PH_*(G',M).    
$

\subsubsection{Eilenberg–Zilber theorem and the K\"unneth formula} If $\KK$ is a principal ideal domain and $(G,M)$ and $(G',M')$ are marked groups, then Proposition \ref{prop:EZ_embedded} implies that there is an isomorphism of complexes 
\begin{equation}\label{eq:EZ-groups}
    \Omega(G,M) \otimes \Omega(G',M') \cong \Omega(G\times G', M \vee M'),
\end{equation}
where $M\vee M'=(M\times 1) \cup (1\times M'),$
which implies that there is the corresponding K\"unneth-like short exact sequence: if we set $H_*=\PH_*(G,M),$ $H'_*={\sf PH}_*(G',M')$  and $H''_*={\sf PH}_*(G\times G',M\vee M')$ then there is a natural short exact sequence. \begin{equation} \label{eq:kunneth-groups} 0 \longrightarrow 
\bigoplus_{i+j=n} H_i\otimes H'_j \longrightarrow H''_n \longrightarrow \bigoplus_{i+j=n-1} {\sf Tor}^\KK_1(H_i,H'_j) \longrightarrow 0.
\end{equation}

\subsubsection{Free product}

For two groups $G$ and $G'$ we denote by $G*G'$ their free product and we denote by $\iota:G\to G*G'$ and $\iota':G'\to G*G'$  the canonical embeddings. 

Note that for any group $G$ we have $\Omega_0G =\NN_0 G\cong \KK$ is generated by one element, which is the only object of $G,$ treated a category. We set 
\begin{equation}
R := {\sf Ker}(\NN_0 G \oplus \NN_0 G' \epi \NN_0(G*G')) \cong \KK.
\end{equation}

\begin{proposition}\label{prop:free_product} For any marked groups $(G,M)$ and $(G',M')$ the morphisms $\iota,\iota'$ induce an isomorphism
\begin{equation}
\Omega(G*G',\iota(M)\cup \iota'(M')) \cong (\Omega(G,M)\oplus \Omega(G',M'))/R. 
\end{equation} 
In particular, we have an isomorphism
\begin{equation}\label{eq:free:iso}
 \PH_n(G*G',\iota(M)\cup \iota'(M')) \cong \PH_n(G,M)\oplus \PH_n(G',M')  
\end{equation}
for $n\geq 1.$
\end{proposition}
\begin{proof} For a group $\mathpzc{G}$ we set $\NN \mathpzc{G} = \NN(\KK[\Nrv( \mathpzc{G})]).$ It is easy to see that the morphism $(\NN G \oplus \NN G')/R \to \NN(G*G')$ is a monomorphism. Moreover, this monomorphism induces an isomorphism of graded submodules $ (\NN M\oplus \NN M')/R \cong \NN (\iota(M)\cup \iota' (M')).$ Then the assertion follows from Remark \ref{rem:embedding}.
\end{proof}

\begin{proposition}\label{prop:union}
Let $G$ be a group and $M,N$ be  subsets of $G,$ containing unit, such that the sets $M^2\setminus 1, MN\setminus 1 ,NM \setminus 1 ,N^2\setminus 1$ are disjoint and assume that for any $x_1,x_2\in M$ and $y_1,y_2\in N$ such that $(x_1,y_1)\ne (x_2,y_2)$ we have $x_1y_1\ne x_2y_2$ and $y_1x_1\ne y_2x_2.$ Then for any $n\geq 1$
\begin{equation}
\PH_n(G,M\cup N) \cong \PH_n(G,M) \oplus \PH_n(G,N).
\end{equation}
\end{proposition}
\begin{proof}
Consider the morphism of marked groups $(G*G,\iota(M)\cup \iota(N))\to (G,M\cup N).$ The assumptions imply that this morphism induces bijections $\iota(M)\cup \iota(N)\cong M\cup N$ and $(\iota(M)\cup \iota(N))^2\cong (M\cup N)^2.$ Then the assertion follows from Proposition \ref{prop:free_product} and the isomorphism lemma (Paragraph \ref{par:isomorphism}).
\end{proof}

\begin{example}
Assume that $\KK=\ZZ.$ Take 
\begin{equation}
G=\langle x,y\mid x^2=y^2=[[x,y],y]=1 \rangle.  
\end{equation}
Then the subsets $\{1,x\}$ and $\{1,y\}$ are subgroups and $\PH_*(G,\{1,x\})\cong \PH_*(G,\{1,y\})$ are isomorphic to the homology of the two-element group. Therefore 
Proposition \ref{prop:union} implies that 
\begin{equation}
\PH_n(G,\{1,x,y\}) = 
\begin{cases} \ZZ,&  n=0\\
(\ZZ/2)^2,&  n \text{ is odd} \\ 
0,& n\geq 2 \text{ is even}
\end{cases}    
\end{equation}
\end{example}

\subsubsection{Low dimensional homology} Proposotion \ref{prop:low_dim} implies that $\Omega_0(G,M)=\KK,$ $\Omega_1(G,M)=\KK[X]$ and 
\begin{equation}
\Omega_2(G,M) = {\sf span} \{ \langle x_1, y_1 \rangle - \langle x_2,y_2 \rangle\mid y_1x_1=y_2x_2, x_i,y_i\in X \}.    
\end{equation}
The differential $\Omega_1(G,M)\to \Omega_0(G,M)$ is trivial. Hence, we have $\PH_0(G,M)=\KK$ and 
\begin{equation}
\PH_1(G,M) = \KK[X]/{\sf span}\{ x_1+y_1-x_2-y_2  \mid y_1x_1=y_2x_2, x_i,y_i\in M  \}. 
\end{equation}

\subsubsection{Cohomology of marked groups, and their ideals} 

Since a marked group $(G,M)$ can be treated as a marked category, following Subsection \ref{subsec:cohomology_of_embedded_quivers} we can consider the cohomology $\PH^*(G,M)$ as a graded algebra. The construction is natural by $(G,M),$ and for any morphism $f:(G,M)\to (G',M')$ we obtain a homomorphism 
\begin{equation}
    \PH^*(G',M') \longrightarrow \PH^*(G,M).
\end{equation}
In particular, for any marked group $(G,M)$ we have a homomorphism $H^*(G) \to \PH^*(G,M),$ whose kernel is an ideal. Hence $M$ defines an ideal
$I_M\triangleleft H^*(G).$

\subsection{Abelian groups and Pontryagin product} If $G$ is an abelian group, then the product map $G\times G\to G$ is a homomorphism. Using the Eilenberg-Zilber morphism  $H_*(G)\otimes H_*(G)\to H_*(G\times G),$ this allows us to define a map 
\begin{equation}
    H_*(G)\otimes H_*(G)\to H_*(G\times G) \to H_*(G),
\end{equation}
which is called Pontryagin product on $H_*(G)$ that makes $H_*(G)$ graded-commutative algebra (\cite[Ch. V \S 5]{brown2012cohomology}). 

In the similar manner a marked abelian group $(G,M)$ there is a morphism
\begin{equation} \label{eq:pontryagin}
\PH_*(G,M)\otimes \PH_*(G,M) \to \PH_*(G\times G, M \vee M) \to \PH_*(G,M),
\end{equation}
where the map $\PH_*(G,M)\otimes \PH_*(G,M) \to \PH_*(G\times G, M \vee M)$ is induced by the Eilenberg-Zilber map (see \eqref{eq:kunneth-groups}). This product will be also called Pontryagin product on $\PH_*(G,M).$

\begin{proposition}\label{prop:Pontryagin}  For a marked abelian group $(G,M)$ the Pontryagin product \eqref{eq:pontryagin} defines a structure of graded-commutative (assotiative, unital) algebra on $\PH_*(G,M)$ that depends naturally of $(G,M).$ Moreover, if $M=G,$ then the Pontryagin product \eqref{eq:pontryagin} coincides with the classical Pontryagin product on $H_*(G).$
\end{proposition} 
\begin{proof}
It is well known \cite[p.220]{quillen1969rational} and can be easily checked that the Eilenberg-Zilber map is associative and commutative in the following sense. For any simplicial modules $A,A',A''$ the diagrams 
\begin{equation}
\begin{tikzcd}
\NN A \otimes \NN A' \otimes \NN A'' \ar[r,"\varepsilon \otimes 1"] \ar[d,"1\otimes \varepsilon"] & \NN (A\otimes A') \otimes \NN A'' \ar[d,"\varepsilon"] \\
\NN A \otimes \NN(A'\otimes A'') \ar[r,"\varepsilon"] & \NN(A\otimes A' \otimes A'')
\end{tikzcd}    
\end{equation}
and 
\begin{equation}
\begin{tikzcd}
\NN A \otimes \NN A' \ar[r,"\varepsilon"] \ar[d,"t"] & \NN(A\otimes A')\ar[d,"\NN(T)"] \\
\NN A' \otimes \NN A \ar[r,"\varepsilon"] & \NN(A'\otimes A)
\end{tikzcd}    
\end{equation}
are commutative, where $t:\NN A\otimes \NN A' \to \NN A' \otimes \NN A$ is a morphism of complexes given by $t(a\otimes a')= (-1)^{nn'} a' \otimes a$ for $a\in \NN_n A $ and $a'\in \NN_{n'} A',$ and $T:A\otimes A' \to A' \otimes A$ is given by $T(a\otimes a') = a'\otimes a.$ (Someone, who likes abstract nonsense,  can say that $(\NN,\varepsilon)$ is a symmetric lax monoidal functor ${\sf sMod}\to {\sf Ch}_{\geq 0}$. It is also related to ``monoidal Dold-Kan correspondence''). 
This follows that for any commutative simplicial algebra $\mathpzc{A}$ there is a structure of graded-commutative algebra on $\NN \mathpzc{A}$ defined by the map $\NN \mathpzc{A} \otimes \NN \mathpzc{A} \overset{\varepsilon}\longrightarrow \NN (\mathpzc{A} \otimes \mathpzc{A}) \longrightarrow \NN  \mathpzc{A},$ where $\mathpzc{A} \otimes \mathpzc{A} \to  \mathpzc{A} $ is the multiplication morphism.  

In our case we consider the commutative simplicial algebra $\mathpzc{A}=\KK[\Nrv(G)]$ with the multiplication defined by multiplication on $G.$ Then we obtain a structure of graded-commutative dg-algebra on $\NN G := \NN(\KK[\Nrv(G)])$ that induces the Pontryagin product on $H_*(G)=H_*( \NN G).$ Lemma \ref{lemma_restriction} implies that  $\Omega(G,M)$ is a dg-subalgebra of $\NN G$ and the following diagram is commutative.
\begin{equation}
\begin{tikzcd}
\Omega(G,M)\otimes \Omega(G,M) \ar[r,"\varepsilon"] \ar[d,hookrightarrow] & \Omega(G\times G,X\vee X) \ar[r,"\Omega(\mu)"] \ar[d] & \Omega(G,M) \ar[d] \\ 
\NN G \otimes \NN G \ar[r,"\varepsilon"] & \NN(G\times G) \ar[r,"\NN(\mu)"] & \NN G
\end{tikzcd}    
\end{equation}
It follows that the Pontryagin product on $\PH_*(G,M)$ is induced by the product on the subalgebra $\Omega(G,M)\subseteq \NN G,$ which is graded-commutative dg-algebra. It follows that $\PH_*(G,M)$ with Pontryagin product is also a graded-commutative algebra. If $M=G,$ then $\Omega(G,M)= \NN G$ and hence the Pontryagin product on $\PH_*(G,G)$ coincides with the classical Pontryagin product. 
\end{proof}

\begin{remark}
Proposition \ref{prop:Pontryagin} implies that for any pointed subset of an abelian group $M\subseteq G,$ the map $\PH_*(G,M)\to H_*(G)$ is an algebra homomorphism. In particular, the image of this homomorphism is a subalgebra. So any  $M\subseteq G$ defines a subalgebra of $H_*(G).$ 
\end{remark}

\subsection{Coacyclic subsets}

In this subsection we assume that $\KK=\ZZ.$ A pointed subset $M\subseteq G$ will be called coacyclic if the morphism $\PH_*(G,M)\to H_*(G)$ is an isomorphism. Further we list some properties of coacyclic subsets. 

\begin{proposition}\label{prop:coacycllic}
If $M\subseteq G,M'\subseteq G'$ are coacyclic subsets, then 
\begin{equation}
M \vee M'\subseteq G\times G',\hspace{1cm}\iota(M)\cup \iota'(M')\subseteq G*G'
\end{equation}
are coacyclic subsets. 
\end{proposition}
\begin{proof}
It follows from  \eqref{eq:kunneth-groups} and Proposition \ref{prop:free_product}. 
\end{proof}

\begin{example}
For any group $G$ the group itself $X=G\subseteq G$ is a coacyclic subset.
\end{example}

\begin{example} By Proposition \ref{prop:coacycllic} for any groups $G,G'$
the subsets $G\vee G'\subseteq G\times G$ and $\iota(G)\cup \iota'(G)\subseteq G*G'$ are coacyclic. 
\end{example} 

\begin{example}\label{examp:Z}
It is easy to check that $\{0,1\}\subseteq \ZZ$ is a coacyclic subset. 
\end{example}
\begin{example}
Proposition \ref{prop:coacycllic} and Example \ref{examp:Z} imply that $\{0,e_1,\dots,e_n\}\subseteq \ZZ^n$ is a coacyclic subset, where $e_1,\dots,e_n$ is the standard basis of $\ZZ^n.$
\end{example}
\begin{example}
Proposition \ref{prop:coacycllic} and Example \ref{examp:Z} imply that $\{1,x_1,\dots,x_n\}\subseteq F(x_1,\dots,x_n)$ is a coacyclic subset, where $F(x_1,\dots,x_n)$ is a free group.
\end{example}
\begin{example}
For example, if we consider the Higman group
\begin{equation}
G=\langle x_0,x_1,x_2,x_3 \mid x_i^{-1}x_{i+1}x_i=x_{i+1}^2, \ i\in \ZZ/4 \rangle,
\end{equation}
then $H_n(G)=0$ for $n\geq 1$ and the one-element set $M=\{1\}$ is coacyclic in this group. More generally, one element set $M=\{1\}$ is coacyclic in a group $G$ if and only if $G$ is acyclic.
\end{example}

\begin{example}
Here we present a non-example of coacyclic subset. Let $F=F(x_1,\dots,x_n)$ be a free group and $\gamma_i\subseteq F$ be the lower central series of $F,$ which is defined by the formula $\gamma_{i+1}=[\gamma_i,F],$ where $\gamma_1=F.$ Take $i\geq 3,$  set $G=F/\gamma_i$ and let $M\subseteq G$ to be the image of $\{1,x_1,\dots,x_n\}.$ 
Then the Hopf's formula says that $H_2(G)=\gamma_i/\gamma_{i+1}\ne 0.$ The equation \eqref{eq:is:X^2} implies that $\PH_*(F,\{1,x_1,\dots,x_n\})=\PH_*(G,M).$ Therefore $\PH_2(G,M)=0$ and the generating set $M\subseteq G$ is not coacyclic.
\end{example}

\section{Marked algebras} \label{sec:Hochschild}

\subsection{Path Hochschild homology}
In this section for simplicity we will assume that $\KK$ is a field. Let $\Lambda$ be an (associative, unital) $\KK$-algebra and $V$ be a $\Lambda$-bimodule. We denote by $A(\Lambda,V)$ the simplicial module such that $A(\Lambda)_n=V\otimes \Lambda^{\otimes n}$ and 
\begin{equation}
\begin{split}
d_0(v\otimes \lambda_1 \otimes \dots \otimes \lambda_n) &= v\lambda_1 \otimes \dots \otimes \lambda_n, \\
d_i(v\otimes \lambda_1 \otimes \dots \otimes \lambda_n) &=  v \otimes \lambda_1 \otimes \dots \otimes \lambda_i\lambda_{i+1}\otimes \dots \otimes \lambda_n, \hspace{5mm} 1\leq i\leq n-1, \\ 
d_n(v\otimes \lambda_1 \otimes \dots \otimes \lambda_n) &=  \lambda_n v \otimes \lambda_1 \otimes \dots \otimes \lambda_{n-1},\\
s_i(v\otimes \lambda_1 \otimes \dots \otimes \lambda_n) &= v\otimes \lambda_1 \otimes \dots \otimes \lambda_i \otimes 1 \otimes \lambda_{i+1} \dots \otimes \lambda_n.
\end{split}
\end{equation}

Then Hochschild homology of $\Lambda$ with coefficients in $V$ can be defined as 
\begin{equation}
HH_*(\Lambda,V) = H_*( \NN( A(\Lambda,V)) ).    
\end{equation}

A vector subspace $M$ of $\Lambda$ is called pointed if $1\in M.$ A marked algebra is a couple $(\Lambda, M),$ where $M$ is a pointed vector subspace of $\Lambda$. For a marked algebra we define a path submodule $B(\Lambda,M,V) \subseteq A( \Lambda,V )$ by the formula $B(\Lambda,M,V) = V \bar \otimes M^{\bar \otimes n}.$ It is easy to check that $B(\Lambda,M,V)$ is indeed a path submodule of $A(\Lambda,M).$ Then we define the path 
Hochschild homology of the marked algebra $(\Lambda,M)$ with coefficients in $V$ as the homology of the corresponding path pair of modules an set
\begin{equation}
\begin{split}
\PPP(\Lambda,M,V) &= (A(\Lambda,V), B(\Lambda,M,V)), \\
\Omega(\Lambda,M,V) &= \Omega(\PPP(\Lambda,M,V)),\\
{\sf PHH}_*(\Lambda,M,V) &= H_*( \Omega(\Lambda,M,V) ).
\end{split}
\end{equation}
The algebra $\Lambda$ can be considered as a bimodule over itself. We set $\Omega(\Lambda,M)=\Omega(\Lambda,M,\Lambda)$ and $ {\sf PHH}_*(\Lambda,M)={\sf PHH}_*(\Lambda,M,\Lambda).$

\subsection{Eilenberg-Zilber theorem for marked algebras} 

\begin{proposition}
Let $(\Lambda,M)$ and $(\Lambda',M')$ 
be marked algebras, and $V$ and $V'$ be bimodules over $\Lambda$ and $\Lambda'$ respectively. Then there is an isomorphism \begin{equation}
    \Omega(\Lambda\otimes \Lambda', M\otimes \KK + \KK\otimes M',V\otimes V') \cong \Omega(\Lambda,M,V) \otimes \Omega(\Lambda',M',V').
\end{equation} 
In particular, we have
\begin{equation} \Omega(\Lambda\otimes \Lambda', M\otimes \KK + \KK\otimes M') \cong \Omega(\Lambda,M) \otimes \Omega(\Lambda',M')
\end{equation}
and 
\begin{equation}
{\sf PHH}_*(\Lambda\otimes \Lambda', M\otimes \KK + \KK\otimes M') \cong {\sf PHH}_*(\Lambda,M) \otimes {\sf PHH}_*(\Lambda',M').
\end{equation}
\end{proposition}
\begin{proof}
In order to use Corollary \ref{cor:EZ}  we only need to prove that 
\begin{equation}
\PPP(\Lambda\otimes \Lambda', M\otimes \KK + \KK\otimes M',V\otimes V') \cong \PPP(\Lambda,M,V) \square \PPP(\Lambda',M',V').    
\end{equation}
Consider the map
\begin{equation}
\theta:A(\Lambda,V) \otimes A(\Lambda',V') \longrightarrow A(\Lambda\otimes \Lambda' ,V\otimes V')    
\end{equation}
defined by 
\begin{equation}
\begin{split}
&\theta ((v\otimes \lambda_1\otimes \dots \otimes \lambda_n) \otimes (v'\otimes \lambda_1'\otimes \dots \otimes \lambda_n' ))=\\
&= (v\otimes v') \otimes (\lambda_1\otimes \lambda_1')\otimes \dots \otimes (\lambda_n\otimes \lambda_n').
\end{split}
\end{equation}
Obviously, $\tau$ is an isomorphism of simplicial modules. So, it is sufficient to prove that  
\begin{equation}
\theta((B(\Lambda,M,V) \diamond B(\Lambda',M',V'))_n) = B(\Lambda \otimes \Lambda' ,M\otimes \KK+\KK\otimes M',V\otimes V')_n    
\end{equation}
for any $n.$ Set $M_1=M$ and $M_0=\KK.$ For a subset $I\subseteq \{0,\dots,n-1\}$ we set $M_I=M_{I(0)}\otimes \dots \otimes M_{I(n-1)},$ where $I(x)=1,$ if $x\in I,$ and $I(x)=0,$ if $x\notin I.$ 
Then for any surjective order preserving map $\sigma:[n]\to [k]$ we have $\sigma^*( V\otimes M^{\otimes k}) = V \otimes M_{\Ker(\sigma)}.$ 
Similarly we define $M'_I$ and obtain $\tau^*(V'\otimes (M')^{\otimes l})=V'\otimes M'_{\Ker(\tau)}$ for any surjective order preserving map $\tau:[n]\epi [l].$  
Therefore, it is sufficient to prove that 
\begin{equation}
\theta \left( \sum_{I\cup J =\{0,\dots,n-1\}} (V\otimes M_I)\otimes (V'\otimes M'_J)   \right) = (V\otimes V')\otimes (M\otimes \KK + \KK \otimes M')^{\otimes n}
\end{equation}
which follows from the fact that 
\begin{equation}
(M\otimes \KK + \KK \otimes M')^{\otimes n} = \sum_{I\sqcup J=\{0,\dots,n-1\}} (M_{I(0)}\otimes M'_{J(0)}) \otimes \dots \otimes (M_{I(n-1)} \otimes M'_{J(n-1)}).
\end{equation}
\end{proof}

\subsection{Isomorphism lemma for marked algebras}

For a vector subspace of an algebra $M\subseteq \Lambda$ we denote by $M^2$ the vector subspace generated by all pairwise products from $M.$

\begin{proposition} Let 
$f:\Lambda\to \Lambda'$ be an algebra homomorphism and $M,M'$ be pointed vector subspaces of these algebras such that $f$ induces isomorphisms $M\cong M'$ and $M^2\cong (M')^2.$ Then for any  $\Lambda'$-bimodule $V$ the homomorphism $f$ induces an isomorphism 
\begin{equation}
\Omega(\Lambda,M,V)\cong \Omega(\Lambda',M',V),   
\end{equation}
where $V$ is considered as a $\Lambda$-bimodule via $f.$
\end{proposition}
\begin{proof}
Set $\overline{\Lambda} = \Lambda /\KK,$  $\overline{M}=M/\KK$ and $\overline{M^2}=M^2/\KK.$ And similarly for $\overline{\Lambda'},\overline{M'},\overline{(M')^2}.$ Then 
\begin{equation}
\NN A(\Lambda,V)_n= V \otimes \overline{\Lambda}^{\otimes n}.
\end{equation}
Since $\KK$ is a field, $\bar M^{\otimes n}$ is a vector subspace of $ \bar \Lambda^{\otimes n}$ and we see that 
\begin{equation}
\overline{B(\Lambda,M,V)} = V\otimes \overline{M}^{\otimes n}. \end{equation}
Similar formulas hold for $\overline{B(\Lambda',M',V)},$ and for $\overline{B(\Lambda,M^2,V)}$ and $\overline{B(\Lambda',(M')^2,V)}.$ Since $f$ induces isomorphisms $\overline{M}\cong \overline{M'}$ and $\overline{M^2}\cong \overline{(M')^2},$ these formulas imply that the map $\NN A(\Lambda,V)\to \NN A(\Lambda',V)$ induces isomorphisms $\overline{B(\Lambda,M,V)}\cong \overline{B(\Lambda',M',V)}$ and $ \overline{B(\Lambda,M^2,V)}\cong \overline{B(\Lambda',(M')^2,V)}.$ Then the assertion follows from Proposition \ref{prop:E-iso}. 
\end{proof}

\section{Appendix. Box product of path sets via Day convolution}\label{Appendix}

The aim of this section is to present a more categorial point of view on box product of path pairs by introducing a box product of path sets. We show that the functor $\PPi_\square: \Pi^{\sf op}\times \Pi\times \Pi \to {\sf Sets}$ gives rise a structure of pro-monoidal category on $\Pi,$ that defines the box product on the category of path sets by the Day convolution \cite{day1970closed},  \cite{day1970construction}, \cite{im1986universal}.  For simplicity we will consider only path sets here, however, this can be easily generalised to path objects of a Benabou cosmos.

\subsection{Pro-functors} 

Here we remind the notion of a pro-functor (also called distributor). A more detailed review of this theory can be found in \cite[\S 7.8]{borceux1994handbook}. 

Let $\CCC$ and $\DDD$ be categories. A profuctor $\FFF:\CCC\rightsquigarrow \DDD$ is a functor $\FFF: \DDD^{\sf op}\times \CCC \to {\sf Sets}.$ The composition of two profunctors $\FFF:\CCC \rightsquigarrow \DDD$ and $\GGG:\DDD \rightsquigarrow \MMM$ is defined as the coend 
\begin{equation}\label{eq:pro-composition}
(\GGG\odot \FFF)(e,c) = \int^d \GGG(e,d)\times \FFF(d,c). 
\end{equation} This composition is associative up to natural isomorphism.

Every functor $f:\CCC\to \DDD$ defines a profunctor $\DDD(1,f):\CCC \rightsquigarrow \DDD $  given by $\DDD(1,f)(d,c)=\DDD(d,f(c)).$ An advantage of profunctors over functors is that for any subcategories $\CCC'\subseteq \CCC$ and $\DDD'\subseteq \DDD$ a profunctor $\CCC \rightsquigarrow \DDD$ induces a profunctor $\CCC'\rightsquigarrow \DDD'.$

For a category $\CCC$ we denote by ${\sf PSh}(\CCC)$ the category of presheaves. If we denote by $\mathbf{1}$ the category with one object, then a presheaf on $\CCC$ is a pro-functor $\mathbf{1}\to \CCC.$ Any pro-functor $\FFF:\CCC \rightsquigarrow \DDD$ defines a  functor 
\begin{equation}
    \FFF_* :{\sf PSh}(\CCC) \longrightarrow {\sf PSh}(\DDD)
\end{equation}
given by the composition \eqref{eq:pro-composition}. Moreover, we have a natural isomorphism 
\begin{equation}
\GGG_* \circ \FFF_* \cong (\GGG \odot \FFF)_*.
\end{equation}

\subsection{Pro-monoidal category} 

A pro-monoidal category is a category $\CCC$ together with the following data
\begin{itemize}
    \item a profunctor $\PPP:\CCC \times \CCC \rightsquigarrow \CCC;$
    \item a profunctor $\JJJ:\mathbf{1}\rightsquigarrow  \CCC;$
    \item associativity isomorphism $\alpha: \PPP \odot (\PPP\times 1) \cong \PPP\odot (1\times \PPP)$
    \item unit isomorphisms $\lambda:\PPP\odot \JJJ\cong \PPP$ and $\rho:  \JJJ\odot \PPP \cong \PPP.$
\end{itemize}
satisfying the pentagon and the unit conditions (see \cite[Def.2.1.1]{day1970construction} for details). Any monoidal category is a pro-monoidal category where $\PPP:\CCC^{\sf op}\times \CCC\times \CCC\to {\sf Sets}$ is defined as 
\begin{equation}
\PPP(c_1,c_2,c_3) = \CCC(c_1,c_2\otimes c_3)
\end{equation}
and $\JJJ:\CCC^{\sf op}\to {\sf Sets}$ is defined as $\JJJ(c)=\CCC(c,1_\CCC).$

An advantage of promonoidal categories over monoidal categories is that a subcategory of a promonoidal category has an induced structure of a promonoidal category. On the other hand the category of presheaves on a promonoidal category has a natural structure of a monoidal category, which is called Day convolution. 

\subsection{Day convolution}

Assume that $(\CCC,\PPP,\JJJ, \alpha  ,\rho,\lambda)$ is a pro-monoidal category. Then we can define the monoidal structure on the category of presheaves ${\sf PSh}(\CCC),$ where the tensor product is defined as the coend  
\begin{equation}
(X\otimes Y)(c) = \int^{c_1,c_2} \PPP(c,c_1,c_2) \times X(c_1)\times Y(c_2)
\end{equation}
and $\JJJ\in {\sf PSh}(\CCC)$ is the unit object.

\subsection{Box product of path sets}

Consider the functor $\PPP=\PPi_\square$
\begin{equation}
\PPi_\square : \Pi^{\sf op}\times \Pi \times \Pi \longrightarrow {\sf Sets}
\end{equation}
defined in \eqref{eq:PPi_square},
and the one-point path set $\JJJ=*:\Pi^{\sf op}\to {\sf Sets},$ which can be defined by the formula  $\JJJ_n=*=\Pi([n],[0]).$
We claim that they define a pro-monoidal structure on $\Pi.$ Indeed, consider the embedding to the category of quivers $\qq:\Pi\to {\sf Quiv}$ (Proposition \ref{prop:q_embedding}). It is easy to check that the box-product of quivers defines a monoidal structure on ${\sf Quiv},$ where the unit object is $\qq^0.$ Then it can be restricted to a structure of pro-monoidal category on the full subcategory $\qq(\Pi)\subseteq {\sf Quiv}$ consisted of the quivers $\qq^n.$ Since a subcategory of a monoidal category inherits a promonoidal structure, we obtain that $\qq(\Pi)$ is a promonoidal category, where the promonoidal structure is defined by the functors
$
\PPP(\qq^n,\qq^k,\qq^l)={\sf Quiv}(\qq^n,\qq^k\square \qq^l)    
$ and $\JJJ(\qq^n)={\sf Quiv}(\qq^n,\qq^0)=*.$
Since the category $\Pi$ is isomorphic to $\qq(\Pi),$  the isomorphism 
\begin{equation}
\PPi_\square(n;k,l)\cong {\sf Quiv}(\qq^n,\qq^k\square \qq^l)
\end{equation}
(Proposition \ref{proposition:square_nerve}) implies that $\PPi_\square$ defines a structure of promonoidal category on $\Pi$. 

Then we can define the box product of path sets $P,P'\in {\sf PSh}(\Pi)$ as the Day convolution:
\begin{equation}
    P\square P' = \int^{[k],[l]} \PPi_\square(-;k,l) \times P_k \times P'_l.
\end{equation}
Now we will define a map 
\begin{equation} \label{eq:theta_coend}
    \theta : P\square P' \longrightarrow P\times P'.
\end{equation}

For any $n,k,l$ there is a map 
\begin{equation}
\theta_{n,k,l} : \PPi_\square(n,k,l)\times P_k \times P_l \longrightarrow P_n \times P'_n,
\end{equation}
\begin{equation}
((f,g),x,y) \mapsto (f^*(x), g^*(y)).
\end{equation}

It is easy to check that for any two morphisms $\varphi:[k]\to [k']$ and $\psi:[l]\to [l']$ of $\Pi$ the diagram 
\begin{equation}
\begin{tikzcd}
& \PPi_\square(n,k,l)\times P_{k}\times P_{l} \ar[rd,"\theta_{n,k,l}"] & \\ 
\PPi_\square(n,k,l)\times P_{k'}\times P_{l'} \ar[ru,"1\times \varphi^*\times \psi^*"] \ar[rd,"{(\varphi,\psi)_*\times 1 \times 1}"'] && P_n\times P'_n \\
& \PPi_\square(n,k',l')\times P_{k'}\times P_{l'} \ar[ru,"\theta_{n,k',l'}"'] &
\end{tikzcd}    
\end{equation}
Then by the universal property of coend we obtain a map \eqref{eq:theta_coend}.

Note that ${\sf Im}(\theta_{n,k,l})=\bigcup_{(f,g)\in \PPi_\square(n,k,l)} f^*(P_k)\times g^*(P'_l).$ Therefore 
\begin{equation}
{\sf Im}(\theta_n : (P \square P')_n\to P_n\times P'_n) = \bigcup_{(f,g)\in \PPi_\square(n)} f^*( P_k ) \times g^*(P_l').     
\end{equation}
Therefore Lemma \ref{lemma:diamond_for_sets} implies that in the definition of the box product of path pairs $(X,Y)\square (X',Y') = (X\times X',Y\diamond Y')$ the path set $Y\diamond Y'$ is the image of the path set $Y\square Y':$
\begin{equation}
    Y\diamond Y' = {\sf Im}( Y\square Y' \longrightarrow X\times X' ).
\end{equation}

\printbibliography
\end{document}